\documentclass[a4paper]{amsart}

\usepackage{lmodern}
\usepackage{amsmath}
\usepackage{amsthm}
\usepackage{caption}
\usepackage{amssymb}
\usepackage{enumerate}
\usepackage{enumitem}
\usepackage{mathrsfs}

\usepackage{tikz}
\usetikzlibrary{arrows,shapes}
\usepackage{tikz-cd}

\usepackage[parfill]{parskip}

\usepackage{hyperref}
\usepackage{cleveref}
\numberwithin{equation}{section}
\newtheorem{theorem}{Theorem}[section]
\newtheorem{definition}{Definition}[section]
\newtheorem{lemma}{Lemma}[section]

\newtheorem{corollary}{Corollary}[section]

\newtheorem{problem}{Problem}[section]

\newtheorem{proposition}{Proposition}[section]

\newtheorem{remark}{Remark}[section]

\usepackage{amsfonts}
\usepackage{array}
\usepackage{verbatim}

\usepackage{twoopt}

\usepackage{color}
\newcommand{\norm}[1]{\left\|{#1}\right\|}

\DeclareFontEncoding{FMS}{}{}
\DeclareFontSubstitution{FMS}{futm}{m}{n}
\DeclareFontEncoding{FMX}{}{}
\DeclareFontSubstitution{FMX}{futm}{m}{n}
\DeclareSymbolFont{fouriersymbols}{FMS}{futm}{m}{n}
\DeclareSymbolFont{fourierlargesymbols}{FMX}{futm}{m}{n}
\DeclareMathDelimiter{\VERT}{\mathord}{fouriersymbols}{152}{fourierlargesymbols}{147}

\newcommand{\nnorm}[1]{\ensuremath{\left\VERT{#1}\right\VERT}}

\newcommand{\ip}[2]{\left<#1,#2\right>}  

\newcommand{\<}[1]{\left<{#1}\right>}

\DeclareMathOperator{\var}{Var}

\newcommand{\averN}[1][i]{\frac1N\sum_{#1=1}^N}

\author{Thomas Holding}
\address{Mathematics Institute\\
	Zeeman Building\\
	University of Warwick\\
	Coventry CV4 7AL, UK
	}
\email{T.Holding@Warwick.ac.uk}
\title[Propagation of chaos via {G}livenko-{C}antelli]{Propagation of chaos for {H}\"older continuous interaction kernels via {G}livenko-{C}antelli}
\date{\today}
\begin{document}

\begin{abstract}
We develop a new technique for establishing quantitative propagation of chaos for systems of interacting particles. Using this technique we prove propagation of chaos for diffusing particles whose interaction kernel is merely H\"older continuous, even at long ranges. Moreover, we do not require specially prepared initial data. On the way, we establish a law of large numbers for SDEs that holds over a class of vector fields simultaneously. The proofs bring together ideas from empirical process theory and stochastic flows.
\end{abstract}
\maketitle
\section{Introduction}
We consider the following system of $N$ particles diffusing in $\mathbb{R}^d$:
\begin{equation}\label{eq:many-particle-system}
\left\{\begin{aligned}
&dX^{i,N}_t=b^N_t(X_t^{i,N})dt+dB^{i,N}_t,\quad i=1,\dotsc,N,\\
&b^N_t(x)=\frac1N\sum_{i=1}^NK(x,X^{i,N}_t),\\
&(X_0^{i,N})_{i=1}^N\text{ i.i.d. with law }f_0.
\end{aligned}\right.
\end{equation}
where $B^{i,N}$ are i.i.d. standard $d$-dimensional Brownian motions and $K(x,y):\mathbb{R}^d\times\mathbb{R}^d\to\mathbb{R}^d$ is the interaction kernel. We are interested in the derivation of a \emph{mean-field} model in the $N\gg 1$ regime for the density of particles $f_t(x)$. One expects that $f$ should solve the non-linear convection-diffusion equation:
\begin{equation}\label{eq:limit-equation}
\left\{\begin{aligned}
&\partial_tf_t+\nabla\cdot(b^\infty_tf_t)-\tfrac12\Delta f_t=0,\qquad (t,x)\in(0,T)\times\mathbb{R}^d,\\
&b^\infty_t(x)=\int f_t(y)K(x,y)\,dy,\\
&f_0(x)\text{ initial condition}.
\end{aligned}\right.
\end{equation}
To rigorously derive this limit one has to show that the \emph{empirical measure}
\begin{equation}\label{eq:empirical-measure}
\mu_t^N=\averN\delta_{X^{i,N}_t}
\end{equation}
converges weakly to the solution $f_t$ to \eqref{eq:limit-equation} as $N\to\infty$. This convergence of the empirical measure to a deterministic limit is known as chaoticity of the particle system \cite{Sznitman}. At the initial time $t=0$ this property is given, as the particles are i.i.d.. At any later time the particles are no longer independent. Establishing that, nevertheless, they are chaotic and the empirical measure converges as $N\to\infty$ is the problem of showing \emph{propagation of chaos}. Of particular interest is making these notions \emph{quantitative} - obtaining explicit polynomial (in $N$) bounds on some probability distance $d(\mu^N_t,f_t)$.

Establishing propagation of chaos is a central part of the rigorous mathematical derivation of macroscopic or mesoscopic continuum models from microscopic laws governing the motion of particles \cite{Golse-Mouhot-Ricci-2013empirical,Jabin2014review}. The notion dates back to Boltzmann's idea of \emph{molecular chaos} used for the derivation of the Boltzmann equation from Newtonian collisions of gas particles. More generally, the notion of propagation of chaos is used in the derivation of the Vlasov-Poisson and Vlasov-Maxwell equations for galaxies and plasmas, for models of swarming \cite{Carrillo-Choi-Hauray,carrillo-flocking-2016review}, in the Vortex dynamics interpretation of the Euler equation \cite{Fournier-Haurray-Mischler-NS}, the particles method for numerical integration of PDEs, the theory of particle filters in statistics \cite{moral2013mean}, the derivation of the spatially homogeneous Boltzmann equation from Kac's model \cite{Michler-Mouhot-Kacs-program}, among many others.

The main general technique for establishing quantitative propagation of chaos in this regime is the coupling method of Sznitman \cite{Sznitman}, which requires $K$ to be \emph{Lipschitz continuous}. Following this, many authors have obtained results assuming that $K(x,y)$ is Lipschitz except for a singularity at $x=y$, and in the presence of specially prepared initial conditions \cite{Fournier-Hauray-Landau,Carrillo-Choi-Hauray,Jabin-Hauray-Vlasov,Jabin2014review}. For these techniques, the presence of the noise is a hindrance as it makes it harder to control the distances between particles.

\textbf{Main result:} In this work we develop a new method for establishing quantitative propagation of chaos, and apply it to give quantitative estimates of propagation of chaos of the system \eqref{eq:many-particle-system} under the assumption that \emph{$K$ is merely H\"older continuous}. In particular, this covers cases where $K$ is \emph{nowhere Lipschitz} and the result does not require specially prepared initial conditions. However, the result relies completely on the presence of noise, and fails at the first hurdle in its absence.

\emph{Second order systems.} We also consider second order Langevin systems of the form:
\begin{equation}\label{eq:2nd-order-many-particle-system}
\left\{\begin{aligned}&\begin{aligned}
&dX^{i,N}_t=V^{i,N}_tdt,\\
&dV^{i,N}_t=b^N_t(X_t^{i,N})dt-\kappa V_t^{i,N}dt+dB^{i,N}_t,\end{aligned} \quad i=1,\dotsc,N,\\
&b^N_t(x)=\frac1N\sum_{i=1}^NK(x,X^{i,N}_t),\\
&(X_0^{i,N},V^{i,N}_0)_{i=1}^N\text{ i.i.d. with law }f_0(x,v).
\end{aligned}\right.
\end{equation}
where again $B^{i,N}_t$ are i.i.d. standard $d$-dimensional Brownian motions, and $\kappa$ is a constant (possibly zero). $X^{i,N}\in\mathbb{R}^d$ is the spatial position of the $i$th particle and $V^{i,N}\in\mathbb{R}^d$ is its velocity. The associated \emph{mean-field} model is the non-linear kinetic Fokker-Planck equation:
\begin{equation}\label{eq:2nd-order-limit-equation}
\left\{\begin{aligned}
&\partial_tf_t+v\cdot\nabla_xf_t-\kappa\nabla_v\cdot(vf_t)+b^\infty_t\cdot\nabla_v f_t-\tfrac12\Delta_v f_t=0,\quad (t,x,v)\in(0,T)\times\mathbb{R}^d\times\mathbb{R}^d,\\
&b^\infty_t(x)=\int \left(\int f_t(y,v)\,dv\right)K(x,y)\,dy,\\
&f_0(x,v)\text{ initial condition}.
\end{aligned}\right.
\end{equation}
Such systems model Newtonian particles under pairwise interaction forces that depend only on the spatial positions of the particles, and whose velocities are driven by independent white noises.

The empirical measure is given by
\begin{equation}\label{eq:2nd-order-empirical-measure}
\mu_t^N=\averN\delta_{(X^{i,N}_t,V_t^{i,N})}.
\end{equation}

\textbf{Main result:} We give quantitative estimates of propagation of chaos of the system \eqref{eq:2nd-order-many-particle-system} under the assumption that \emph{$K$ is H\"older continuous with H\"older exponent greater than $2/3$.} Again this applies to interaction kernels that are \emph{nowhere Lipschitz} and the result does not require specially prepared initial conditions. The restriction of the H\"older exponent to be at least $2/3$ is due to the degeneracy of the generator of the diffusion process in the spatial variable. The generator is only \emph{hypoelliptic} rather than \emph{elliptic} and this reduces the regularising effect on the dynamics. 

On the way to proving the propagation of chaos, we also establish a \emph{Glivenko-Cantelli} theorem \cite{Van-der-Vaart-Wellner} for SDEs over all bounded H\"older continuous vector fields, (with a similar result in the second order case). This will be discussed in more detail below in \cref{subsubsec:empirical-process-theory} with a more precise statement, but we provide an informal statement below.

\emph{Glivenko-Cantelli theorem for SDEs (informal statement)}\label{thm:informal-GV-thm}
Let $X^{b}_t$ solve $dX_t=b_t(X_t)\,dt+dB_t$ for a vector field $b$, and $(X^{b,i,N}_t)_{i=1}^N$ be $N$ i.i.d. copies of $X^b$. Then
\begin{equation*}
\mathbb{E}\sup_b\sup_{t\in[0,T]}d(\mu^{b,N}_t,f^b_t)\le CN^{-\gamma}
\end{equation*}
where $\mu^{b,N}$ is the empirical measure associated with $(X^{b,i,N})_{i=1}^N$, $f^b$ is the law of $X^b$, $d$ is some metric on the space of probability measures and the supremum is over all smooth vector fields $b$ with $\alpha$-H\"older norm bounded by a uniform constant.

The proof uses recent results on stochastic flows for rough drifts (see e.g. \cite{Flandoli-Transport-equation,Banos-Duedahl,mohammed2015,Fedrizzi-Flandoli,wang2015degenerate,fedrizzi2016regularity} among others), and methods from empirical process theory.

\textbf{Acknowledgements:} The author would like to thank Vittoria Silvestri for a helpful discussion regarding the continuity of the stochastic process defined in \cref{subsec:1st-order-stochastis-process}, Zhenfu Wang for discussion of \cite{Jabin-Wang} and for comments on the presentation of an earlier draft of this manuscript and Jos\'e A. Carrillo for mathematical comments during the preparation of this work. We also wish to thank Franco Flandoli for a discussion that lead to the discovery of an error in a previous version of this manuscript. The author acknowledges the support of the EPSRC funded (EP/H023348/1) Cambridge Centre for Analysis.

\subsection{Layout of the paper}
The paper is laid out as follows. In \cref{subsec:preliminaries} we give preliminary definitions. \cref{sec:main-results} presents the main results of the paper.  In \cref{sec:discussion} we discuss prior work, compare our method to existing techniques and discuss open questions. The proofs of the results begin in \Cref{sec:empirical-process-theory} where we give the proof of \cref{thm:GV-intro}. Then in \cref{sec:propagation-of-chaos} we apply the results proved in \cref{sec:empirical-process-theory} to prove \cref{thm:propagation-of-chaos}. \cref{sec:counterexample} provides the proof of \cref{prop:counterexample}. Finally \cref{sec:appendix} presents some properties of metric entropy which are used in the earlier sections of the manuscript.

\subsection{Preliminaries}\label{subsec:preliminaries}
Before we state the main results of this work we require some preliminary definitions.

We will always work with a single probability space $(\Omega,\mathcal{F},\mathbb{P})$ that contains $N$ i.i.d. Brownian motions $(B^{i,N})_{i=1}^N$ defined for times $[0,T]$ for a fixed final time $T$. We emphasise that throughout this work $N$ is a fixed number and we will never take a limit $N\to\infty$. We denote the $\mathrm{L}^p$ norm on the probability space as $\nnorm{\cdot}_p$. \emph{Deterministic} norms are denoted with a double bar, e.g. $\norm{\cdot}_{\mathrm{L}^\infty(\mathbb{R}^d)}$. The space of Borel probability measures on $\mathbb{R}^d$ is denoted $\mathrm{P}(\mathbb{R}^d)$, those with finite $p$th moment are denoted $\mathrm{P}_p(\mathbb{R}^d)$. We also make use of the following norm.
\begin{definition}[Sub-Gaussian norm]\label{def:nnorm}
For a random variable $X$ we define the \emph{sub-Gaussian}\cite{Van-der-Vaart-Wellner,vershynin2010introduction} norm\footnote{The sub-Gaussian norm is sometimes called the $\psi_2$ norm and denoted $\norm{\cdot}_{\psi_2}$. This notation is used, for example, in \cite{Van-der-Vaart-Wellner}.} $\nnorm{X}$ by
\begin{equation*}
\nnorm{X}=\sup_{p\ge1}\frac{1}{\sqrt{p}}\nnorm{X}_p.
\end{equation*}
When $\nnorm{X}$ is finite we say that the random variable $X$ is \emph{sub-Gaussian}. If the random variable with law  $\mu\in \mathrm{P}(\mathbb{R}^d)$ is \emph{sub-Gaussian} then we say that $\mu$ is \emph{sub-Gaussian}.
\end{definition}
The space of random variables with finite sub-Gaussian norm coincides with an exponential Orlicz space \cite{Van-der-Vaart-Wellner}. It is easy to see that if $X$ is a random variable with $\nnorm{X}=c$, then $X$ obeys the tail bound
\begin{equation}\label{eq:nnorm-subgaussian-tail-bound}
\mathbb{P}(|X|>u)\le C\exp\left(-\frac{Cu^2}{c^2}\right)
\end{equation}
for any $u>0$, and absolute constants $C$. We refer the reader to \cite{Van-der-Vaart-Wellner,vershynin2010introduction} for more details.

We define $\mathrm{Lip1}$ to be the set of functions $h:\mathbb{R}^d\to\mathbb{R}$ that are $1$-Lipschitz and vanish at zero.

To metrise weak convergence in $\mathrm{P}(\mathbb{R}^d)$, we define the following metrics:
\begin{definition}[Bounded Lipschitz metric]
For $\mu,\nu\in \mathrm{P}(\mathbb{R}^d)$ we define the \emph{bounded-Lipschitz (BL) metric} $d_{\mathrm{BL}}$ by
\begin{equation*}
d_{\mathrm{BL}}(\mu,\nu)=\sup\left\{\int h\,d\mu-\int h\,d\nu: h\in \mathrm{Lip1}, \norm{h}_{\mathrm{L}^\infty(\mathbb{R}^d)}\le 1\right\}.
\end{equation*}
\end{definition}
\begin{definition}[Monge-Kantorovich-Wasserstein-(Rubinstein) metric]
For $\mu,\nu\in \mathrm{P}_1(\mathbb{R}^d)$ we define the \emph{Monge-Kantorovich-Wasserstein-(Rubinstein) (MKW) metric} $d_{\mathrm{MKW}}$ by
\begin{equation*}
d_{\mathrm{MKW}}(\mu,\nu)=\sup\left\{\int h\,d\mu-\int h\,d\nu:h\in\mathrm{Lip1}\right\}.
\end{equation*}
\end{definition}
The MKW metric can also be defined using optimal transportation as
\begin{equation*}
d_{\mathrm{MKW}}(\mu,\nu)=\inf\left\{\int |x-y|\,d\pi(x,y): \pi\in \mathrm{P}(\mathbb{R}^d\times\mathbb{R}^d)\text{ has marginals $\mu$ and $\nu$}\right \}.
\end{equation*}
We will mostly use the first definition as it more amenable to the tools of empirical process theory.

Integral to the proofs is the concept of metric entropy\cite{Van-der-Vaart-Wellner}.
\begin{definition}[Metric entropy]
Let $(\mathcal{X},d)$ be a totally bounded metric space. Then the \emph{metric entropy} of $\mathcal{X}$ is defined for $\varepsilon>0$ by
\begin{equation*}
H(\varepsilon,\mathcal{X},d)=\inf\{\log(m):(x_i)_{i=1}^m\text{ is an $\varepsilon$-net of }\mathcal{X}\},
\end{equation*}
where an $\varepsilon$-net is a set of points $(x_i)_{i=1}^m\subseteq \mathcal{X}$ for which every $x\in \mathcal{X}$ has $d(x,x_i)\le\varepsilon$ for some $i$.
\end{definition}
We summarise the various properties and estimates of metric entropy used throughout this work in \cref{sec:appendix} for the convenience of the reader.

The usual Lebesgue spaces will be denoted $\mathrm{L}^p(\mathbb{R}^d)$ for $p\in[1,\infty]$. We denote the usual H\"older spaces on $\mathbb{R}^d$ for $k$ a positive integer and $\alpha\in(0,1)$ as $\mathrm{C}^{k,\alpha}(\mathbb{R}^d)$, and the (fractional) Sobolev space of differentiability $s\ge0$ and integrability $p\in[1,\infty]$ as $\mathrm{W}^{s,p}(\mathbb{R}^d)$. When we wish to emphasis which variable a norm is respect to we will denote it with a subscript, e.g. $\mathrm{L}^q_x(\mathbb{R}^d)$.

To allow sets of functions that do not decay at infinity to have finite metric entropy we make use of weighted spaces. For $x\in\mathbb{R}^d$ we define $\<x=\sqrt{1+|x|^2}$.
\begin{definition}[Weighted $\mathrm{L}^p$ spaces]
For $r\in\mathbb{R}$ and $q\in[1,\infty]$ we define $\mathrm{L}^{r,q}(\mathbb{R}^d)$ as the space of functions $h$ such that $h\<x^{r}\in \mathrm{L}^q(\mathbb{R}^d)$ with the norm
\begin{equation*}
\norm{h}_{\mathrm{L}^{r,q}(\mathbb{R}^d)}=\norm{h\<x^{r}}_{\mathrm{L}^q(\mathbb{R}^d)}\left(=\left(\int |h|^q\<x^{rq}\,dx\right)^{1/q}\text{ when }q\ne\infty\right).
\end{equation*}
\end{definition}
In particular, $\mathrm{L}^{0,q}(\mathbb{R}^d)=\mathrm{L}^q(\mathbb{R}^d)$.

Next we define of `abstract H\"older spaces'. We will use these in the proofs rather than the usual Sobolev spaces because they behave more naturally under composition with H\"older continuous functions. 
\begin{definition}\label{def:parameterised-Holder-space}
Let $(\mathrm{V},\norm{\cdot}_\mathrm{V})$ be a Banach space of functions $U\to\mathbb{R}$ for $U=\mathbb{R}^d$ or $U=[0,T]\times\mathbb{R}^d$. Let $\alpha\in(0,1]$ and $k$ be a non-negative integer, then we define the space $\Lambda^{k,\alpha}(\mathrm{V})$ as those functions $h\in \mathrm{V}$ for which the following norm is finite
\begin{equation*}
\begin{aligned}
\norm{h}_{\Lambda^{k,\alpha}(\mathrm{V})}=&\sup_{|\beta|\le k}\norm{\partial^\beta h}_\mathrm{V}\\
&+\sup_{|\beta|=k}\sup_{y,z\in ,y\ne z}\norm{\frac{(\partial^\beta h)(\cdot+z)-(\partial^\beta h)(\cdot+y)}{|y-z|^\alpha}1_{\cdot+z\in U,\cdot+y\in U}}_\mathrm{V}.
\end{aligned}
\end{equation*}
where $\beta$ ranges over multi-indices.
\end{definition}
Note that for $\alpha\in(0,1)$, $\Lambda^{0,\alpha}(\mathrm{L}^\infty(\mathbb{R}^d))$ is the H\"older space $\mathrm{C}^{0,\alpha}(\mathbb{R}^d)$. More generally, if $q\in[1,\infty],s=k+\alpha$ is not an integer then $\Lambda^{k,\alpha}(\mathrm{L}^q(\mathbb{R}^d))$ is the Besov space $B^{s}_{q,\infty}$ (see \cite{Function-spaces-III} for the definitions and properties of the Besov spaces). In particular the fractional Sobolev space $\mathrm{W}^{s,q}(\mathbb{R}^d)$ embeds continuously in $\Lambda^{k,\alpha}(\mathrm{L}^q(\mathbb{R}^d))$.

Due to the driving Brownian motions, the natural regularity for the vector field $b^N$ in \eqref{eq:many-particle-system} is a parabolic space.
\begin{definition}[Parabolic space]
Let $(\mathrm{V},\norm{\cdot}_\mathrm{V})$ be a Banach space of functions $[0,T]\times \mathbb{R}^d\to\mathbb{R}$. Let $\alpha\in(0,1]$ then we define the parabolic space $\Lambda^{0,\alpha}_{para}(\mathrm{V})$ is the space of functions $h\in \mathrm{V}$ with the following norm finite
\begin{equation*}
\begin{aligned}
\norm{h}_{\Lambda^{0,\alpha}_{para}(\mathrm{W})}&=\norm{h}_\mathrm{W}+\\
&\sup_{\substack{(s,y),(t,z)\in\mathbb{R}^{1+d},\\(s,y)\ne (t,z)}}\norm{\frac{h(\cdot+t,\cdot+z)-h(\cdot+s,\cdot+y)}{|y-z|^\alpha+|s-t|^{\alpha/2}}1_{\cdot+(t,z)\in U,\cdot+(s,y)\in U}}_\mathrm{W}.
\end{aligned}
\end{equation*}
\end{definition}

We also define the particular case of the parabolic H\"older spaces.
\begin{definition}[Parabolic H\"older space]
For $\alpha\in(0,1)$ the \emph{parabolic H\"older space} $\mathrm{C}^{0,\alpha}_{para}$ is the space of continuous functions $\varphi:[0,T]\times\mathbb{R}^d\to\mathbb{R}$ with the norm
\begin{equation*}
\norm{\varphi}_{\mathrm{C}^{0,\alpha}_{para}([0,T]\times\mathbb{R}^d)}=\sup_{(t,x)\in [0,T]\times\mathbb{R}^d}|\varphi(t,x)|+\sup_{\substack{(s,y),(t,x)\in([0,T]\times \mathbb{R}^d)^2,\\(s,y)\ne (t,x)}} \frac{|\varphi(t,x)-\varphi(s,y)|}{|t-s|^{\alpha/2}+|x-y|^\alpha}.
\end{equation*}
\end{definition}
Just as $\Lambda^{0,\alpha}(\mathrm{L}^\infty(\mathbb{R}^d))$ is the H\"older space $\mathrm{C}^{0,\alpha}(\mathbb{R}^d)$ (for $\alpha\in(0,1)$) the parabolic H\"older space $\mathrm{C}^{0,\alpha}_{para}(\mathbb{R}^d)$ is equal to $\Lambda^{0,\alpha}_{para}(\mathrm{L}^\infty(\mathbb{R}^d))$.
 
Stochastic flows are the analogue of the flow map of an ODE in the stochastic setting \cite{kunita1997stochastic}.
\begin{definition}[$\mathrm{C}^{k,\beta}$ stochastic flow]\label{def:stochastic-flow}
We say that a random map $\phi_{s,t}:\mathbb{R}^n\to\mathbb{R}^n$ is a $\mathrm{C}^{k,\beta}$ ($k\ge1$,$\beta\in(0,1)$) stochastic flow of diffeomorphisms if it satisfies the following
\begin{enumerate}
\item $\phi_{t,t}$ is the identity map almost surely.
\item $\phi_{u,t}\circ\phi_{s,u}=\phi_{s,t}$ holds as maps, almost surely for all $s<u<t$.
\item $\phi_{s,t}(x)$ is $k$-times differentiable with respect to $x$ and all the derivatives are continuous, with the $k$-th derivative $\beta$-H\"older continuous. Furthermore, the map $\phi_{s,t}:\mathbb{R}^n\to\mathbb{R}^n$ is a diffeomorphism of $\mathbb{R}^n$ almost surely.
\end{enumerate}
\end{definition}
Note that a $\mathrm{C}^{k,\beta}$ stochastic flow need not be globally $\mathrm{C}^{k,\beta}$ as both it and its derivatives may grow without bound as $|x|\to\infty$.

We say that a stochastic differential equation (for $X\in\mathbb{R}^n$) generates if $\mathrm{C}^{k,\beta}$ stochastic flow of diffeomorphisms if the solution map 
\begin{align*}
\phi_{s,t}&:\mathbb{R}^n\to\mathbb{R}^n\\
&X_s\mapsto X_t
\end{align*}
has a version that is a $\mathrm{C}^{k,\beta}$ stochastic flow of diffeomorphisms.

\begin{definition}[Glivenko-Cantelli class]\label{def:glivenko-cantelli-class}
Let $\mathbb{Q}$ be a probability measure on a measurable space $(\mathcal{X},\mathcal{A})$ and $\mathcal{F}$ a class of measurable functions $\mathcal{X}\to\mathbb{R}$. We say that $\mathcal{F}$ is a \emph{Glivenko-Cantelli} class (with respect to $\mathbb{Q}$) if 
\begin{equation*}
\sup_{f\in\mathcal{F}}\left| \averN f(X^{i,N})-\int f\,d\mathbb{Q}\right|\to 0 
\end{equation*}
in probability or almost surely, where $(X^{i,N})_{i=1}^\infty$ are i.i.d. with law $\mathbb{Q}$.
\end{definition}
\begin{remark}
Strictly speaking, the convergence in the definition above should be in \emph{outer} probability or \emph{outer} almost surely (see \cite[\S1]{Van-der-Vaart-Wellner} for the definition and properties of the outer integral), as the supremum may fail to be measurable in general. In this work, however, all considered suprema will be measurable and we will have no need of the more technical definition of the outer integral.
\end{remark}

In this work we consider both the first order many particle system \eqref{eq:many-particle-system} and the second order many particle system \eqref{eq:2nd-order-many-particle-system} along with their respective limit equations. We hope that the reader will admit us the abuse of notation of using the same symbols $X,f,\mu^N$ for each, as which is considered will be clear from the context.

\section{Main results}\label{sec:main-results}
In this section we present the main results of this manuscript.
\subsection{Propagation of chaos}
The first such results are on propagation of chaos.
\subsubsection{First order systems.} 
For first order systems we have the full regularising effect of the driving noise and we only require that the interaction kernel $K$ is H\"older continuous for some positive exponent. Under this assumption we achieve sub-Gaussian concentration of the Wasserstein distance over compact time intervals around the initial distance. Recall that the norm $\nnorm{\cdot}$ is defined by \cref{def:nnorm} and is equivalent to a sub-Gaussian tail bound \eqref{eq:nnorm-subgaussian-tail-bound}. Note that we \emph{cannot} expect $\nnorm{d_{\mathrm{MKW}}(\mu^N_t,f_t)}$ to be small in general, as even at the initial time we have that
\begin{equation*}
\nnorm{d_{\mathrm{MKW}}(\mu^N_0,f_0)}=\infty 
\end{equation*}
unless $f_0$ has sub-Gaussian tails. For this reason we bound instead a compensated quantity  $d_{\mathrm{MKW}}(\mu^N_t,f_t)-d_{\mathrm{MKW}}(\mu_0^N,f_0)$. Bounds upon $d_{\mathrm{MKW}}(\mu^N_0,f_0)$ are a separate question to \emph{propagation} of chaos and are considered elsewhere (e.g. \cite{dobric1995asymptotics,fournier2015rate}).

\begin{theorem}[Propagation of chaos for first order systems]\label{thm:propagation-of-chaos}
Let $(X^{i,N})_{i=1}^N$ be a solution of the first order many particle system \eqref{eq:many-particle-system}, $\mu^N$ be the associated empirical measure given by \eqref{eq:empirical-measure} and $f_t(x)$ be the solution to the limit equation \eqref{eq:limit-equation}. Then the following hold:
\begin{enumerate}
\item \textbf{H\"older interactions:} Let $K\in \mathrm{C}^{0,\alpha}(\mathbb{R}^d\times\mathbb{R}^d;\mathbb{R}^d)$ and $f_0\in \mathrm{P}_p(\mathbb{R}^d)\cap \mathrm{L}^{r,2}(\mathbb{R}^d)$ for some $\alpha\in(0,1]$, $2\ne p>1$ and $r>1+(d/2)$.  Then it holds that
\begin{equation}\label{eq:propagation-of-chaos}
\nnorm{\sup_{t\in[0,T]}d_{\mathrm{MKW}}(\mu^N_t,f_t)-d_{\mathrm{MKW}}(\mu_0^N,f_0)}\le CN^{-\gamma}, \text{ for any }\gamma<\frac1{2+\max(\frac{d+2}{\alpha^2},\frac{d}{p-1})}.
\end{equation}
\item \textbf{Sobolev interactions:} Let $K(x,y)=W(x-y)$ for $W\in \mathrm{W}^{s,q}(\mathbb{R}^d;\mathbb{R}^d)$ with $1\ge s>(2+d)/q$ and $q\in(2,\infty]$. Assume that $f_0\in \mathrm{L}^{r,q'}(\mathbb{R}^d)\cap \mathrm{P}_p(\mathbb{R}^d)$ for some $2\ne p>d/q$, $r>(d/q)+(d/2)+1$ and where $(1/q)+(1/q')=1/2$. Then it holds that
\begin{equation}\label{eq:propagation-of-chaos-Sobolev}
\nnorm{\sup_{t\in[0,T]}d_{\mathrm{MKW}}(\mu^N_t,f_t)-d_{\mathrm{MKW}}(\mu_0^N,f_0)}\le CN^{-\gamma}, \text{ for any }\gamma<\frac1{2+\max(\frac{d+2}{s^2},\frac{d}{p-1})}.
\end{equation}

\end{enumerate}
\end{theorem}
By combining the above theorem and the results on the convergence of $d_{\mathrm{MKW}}(\mu^N_0,f_0)$ in \cite{fournier2015rate} we can obtain the following simple corollary.
\begin{corollary}
Under the assumptions of \cref{thm:propagation-of-chaos} we have
\begin{equation*}
\mathbb{E}\sup_{t\in[0,T]}d_{\mathrm{MKW}}(\mu^N_t,f_t)\le \mathbb{E}d_{\mathrm{MKW}}(\mu_0^N,f_0)+CN^{-\gamma}
\end{equation*}
with $\gamma$ as given in the respective cases (1), (2) of \cref{thm:propagation-of-chaos}. Furthermore, if $p$ is large enough (depending only on $d$) then it holds that 
\begin{equation*}
\mathbb{E}\sup_{t\in[0,T]}d_{\mathrm{MKW}}(\mu^N_t,f_t)\le CN^{-\gamma}.
\end{equation*}
\end{corollary}
We now give some remarks on \cref{thm:propagation-of-chaos}.
\begin{remark}
We do not require any special preparation for the initial condition $f_0$, and the initial particle positions are chosen i.i.d. according to $f_0$. In particular, any compactly supported uniformly bounded density will satisfy all the assumptions of the theorem. The restriction to i.i.d. initial particle positions is made throughout this work, both to simplify the proofs and as it is a natural initial condition. However, the author believes that there is no fundamental reason why this could not be dropped with additional work.
\end{remark}
\begin{remark}\label{rem:Lip1-metric-d/p-1}
In both cases in the above theorem the $d/(p-1)$ term in $\gamma$ comes from the metric entropy of the space $\mathrm{Lip1}$ in a weighted $\mathrm{L}^\infty$ norm (see \cref{lem:metric-entropy-of-Lip1}). The requirement that $p\ne2$ is merely to avoid complicating the theorem statements with logarithmic correction terms for this critical weight. Results for $p=2$ can be obtained by using the inclusion $\mathrm{P}_2\subset \mathrm{P}_p$ for any $p<2$. We maintain the avoidance of $p=2$ throughout the other results of the paper.
\end{remark} 
\begin{remark}
If we consider instead the bounded Lipschitz metric $d_{\mathrm{BL}}$, then the same method of proof allows us to estimate
\begin{equation*}
\nnorm{\sup_{t\in[0,T]}d_{\mathrm{BL}}(\mu^N_t,f_t)}\le CN^{-\gamma}
\end{equation*}
under the same assumptions as \cref{thm:propagation-of-chaos} and with the same corresponding values of $\gamma$. This is because the bounded Lipschitz metric is almost surely bounded by $2$ and there is no need to compensate for the initial error. Additionally, the exponent $\gamma$ can be improved to anything less than $1/(2+\max((d+2)/\alpha^2))$ in (1) and $1/(2+\max((d+2)/s^2))$ in (2), as the metric entropy estimate for the bounded Lipschitz functions is smaller than that for $\mathrm{Lip1}$ (see also \cref{rem:Lip1-metric-d/p-1}). Analogous results in the bounded Lipschitz metric can be formulated for the Glivenko-Cantelli theorems for SDEs (\cref{thm:GV-intro,thm:2nd-order-GV-intro}) presented later in this work, and for the second order propagation of chaos result \cref{thm:2nd-order-propagation-of-chaos} below.
\end{remark}
\begin{remark}\label{rem:relax-assumption-on-K} 
The assumption on $K$ in (2) may be weakened to 
\begin{equation*}
K\in \Lambda^{0,s}(\mathrm{L}^\infty_y(\mathbb{R}^d;\mathrm{L}^q_x(\mathbb{R}^d;\mathbb{R}^d)))
\end{equation*}
for the same $s$ and $q$. (This space is defined in \cref{def:parameterised-Holder-space}.) Note that this condition is implied by the assumption in (2). We provide the proof under this weaker assumption. Note that with this weakened assumption the case $q=\infty$ collapses into (1) with $\alpha=s$ as $\Lambda^{0,s}(\mathrm{L}^\infty(\mathbb{R}^d\times\mathbb{R}^d;\mathbb{R}^d))$ is exactly the H\"older space $\mathrm{C}^{0,s}(\mathbb{R}^d\times\mathbb{R}^d;\mathbb{R}^d)$.
\end{remark}
\begin{remark}
In case (2) in the above theorem the assumptions on $K$ imply that $K\in \mathrm{C}^{0,\alpha}(\mathbb{R}^d\times\mathbb{R}^d;\mathbb{R}^d)$ for $\alpha=s-(d/q)>0$ by Sobolev embedding. So in this sense case (2) is weaker than case (1). The advantage of (2) is that $\gamma$ is obtained from the Sobolev exponent $s$ instead of the (smaller) H\"older exponent $\alpha$. In particular, note that singularities of the form 
\begin{equation*}
K(x,y)=W(x-y)=|x-y|^\alpha, \qquad \alpha\in(0,1)
\end{equation*} 
 are (locally) only $\alpha$-H\"older, while they are (locally) in the Sobolev space $\mathrm{W}^{1,q}$ for any $q<d/(1-\alpha)$. For this type of singularity the exponent $\gamma$ obtained by (2) is substantially better than that from (1).
\end{remark}
\begin{remark}
The integrability assumptions on $f_0$ in the above theorem are not optimal, and the author has made little attempt to optimise them. In particular, the proof does not use the regularising effect (gain of integrability and smoothness) of the parabolic limit equation in the PDE estimates, instead relying on more elementary $\mathrm{L}^2$ energy estimates. A more careful analysis is beyond the scope of this manuscript, in which we are primarily concerned with the assumptions on $K$ and upon the exponent $\gamma$ obtained, and not on optimal integrability of the initial data.
\end{remark}

\subsubsection{Second order systems.} In the second order case we have a similar result, but with the restriction that $\alpha>2/3$ as we have merely a gain of $2/3$ derivatives from the hypoelliptic regularising effect of the noise. However, the particle spatial trajectories have better time regularity properties than in the first order case (at least $\mathrm{C}^1$ rather than $\mathrm{C}^{0,(1/2)-\varepsilon}$) and as a result we obtain a better exponent $\gamma$.
\begin{theorem}[Propagation of chaos for second order systems]\label{thm:2nd-order-propagation-of-chaos}
Let $(X^{i,N},V^{i,N})_{i=1}^N$ be a solution of the second order many particle system \eqref{eq:2nd-order-many-particle-system}, $\mu^N$ be the associated empirical measure given by \eqref{eq:2nd-order-empirical-measure} and $f_t(x)$ be the solution to the limit equation \eqref{eq:2nd-order-limit-equation}. Then the following hold:
\begin{enumerate}
\item \textbf{H\"older interactions:} Let $K\in \mathrm{C}^{0,\alpha}(\mathbb{R}^d\times\mathbb{R}^d;\mathbb{R}^d)$ and $f_0\in \mathrm{P}_p(\mathbb{R}^d\times\mathbb{R}^d)\cap \mathrm{L}^{r,2}(\mathbb{R}^d\times\mathbb{R}^d)$ for $\alpha\in(2/3,1]$, $2\ne p>1$ and $r >1+d$. Then there are finite constants $c,C$ such that the following holds
\begin{equation}\label{eq:2nd-order-propagation-of-chaos}
\nnorm{\left[\sup_{t\in[0,T]}d_{\mathrm{MKW}}(\mu^N_t,f_t)-cd_{\mathrm{MKW}}(\mu^N_0,f_0)\right]_+}\le CN^{-\gamma},
\end{equation}
where here and throughout $[a]_+$ is the positive part of $a$, and $\gamma$ is given by
\begin{equation*}
\gamma=\frac1{2+\max(\tfrac{d+1}{\alpha^2},\tfrac{d}{p-1})}.
\end{equation*}
\item \textbf{Sobolev interactions:} Let $K(x,y)=W(x-y)$ for some $W\in \mathrm{W}^{s,q}(\mathbb{R}^d;\mathbb{R}^d)$ with $q\in(2,\infty]$, $q>(d+1)/s$ and $3/2\ge s>(2/3)+(d/q)$. Let $f_0\in \mathrm{P}_p(\mathbb{R}^d\times\mathbb{R}^d)\cap \mathrm{L}^{r,q'}(\mathbb{R}^d\times\mathbb{R}^d)$ for some $p>d/q$ with also $p>2$, $r>(d/q)+d+1$ and where $(1/q)+(1/q')=1/2$. Then there are finite constants $c,C$ such that the following holds
\begin{equation*}
\begin{aligned}
\nnorm{\left[\sup_{t\in[0,T]}d_{\mathrm{MKW}}(\mu^N_t,f_t)-cd_{\mathrm{MKW}}(\mu^N_0,f_0)\right]_+}\le CN^{-\gamma},\\
\end{aligned}
\end{equation*}
where 
\begin{equation*}
\gamma=\begin{cases}
\dfrac{1}{2+\tfrac{d+1}{s^2}},&\text{if }s\le 1,\\
\dfrac{1}{2+\max(\tfrac{d+1}s,d)}, &\text{otherwise.}
\end{cases}
\end{equation*}
\end{enumerate} 
\end{theorem}
Using the elementary inequality $x\le [x-y]_++y$ for $x,y\ge0$, we can use \cref{thm:2nd-order-propagation-of-chaos} to obtain bounds on the expectation:
\begin{corollary}
	Under the assumptions of \cref{thm:2nd-order-propagation-of-chaos} we have
	\begin{equation*}
	\mathbb{E}\sup_{t\in[0,T]}d_{\mathrm{MKW}}(\mu^N_t,f_t)\le C\mathbb{E}d_{\mathrm{MKW}}(\mu_0^N,f_0)+CN^{-\gamma}
	\end{equation*}
	with $\gamma$ as given in the respective cases (1), (2) of \cref{thm:2nd-order-propagation-of-chaos}. Furthermore, if $p$ is large enough (depending only on $d$) then it holds that 
	\begin{equation*}
	\mathbb{E}\sup_{t\in[0,T]}d_{\mathrm{MKW}}(\mu^N_t,f_t)\le CN^{-\gamma}.
	\end{equation*}
\end{corollary}

We make some remarks on \cref{thm:2nd-order-propagation-of-chaos} (see also the remarks after \cref{thm:propagation-of-chaos} which are applicable here as well).
\begin{remark}
In the first order case (\cref{thm:propagation-of-chaos}) we admitted as evident the well-posedness of both the limit PDE \eqref{eq:limit-equation} and the particle system \eqref{eq:many-particle-system}. In the second order case neither is a priori obvious due to the degeneracy of the noise and the roughness of the coefficients. We note here that the well-posedness of the particle system \eqref{eq:2nd-order-many-particle-system} is a consequence of the existence of a differentiable stochastic flow (see \cite{wang2015degenerate}). That the limit PDE \eqref{eq:2nd-order-limit-equation} is also well-posed may be obtained from this by standard methods. However, we point out to the reader that for most of the paper the proofs are done on mollified ($C^1_b$) vector fields, and so existence and uniqueness is not a concern.
\end{remark}
\begin{remark}
We need to subtract $cd_{\mathrm{MKW}}(\mu_0^N,f_0)$ (for $c>1$) and take the positive part, while in \cref{thm:propagation-of-chaos} we could choose $c=1$ and the non-negativity of the expression was automatic. This is because the initial error may be amplified by the dynamics as the vector fields involved are not bounded (see the proof of \cref{lem:reference-process-mkw-bound} for details).
\end{remark}
\begin{remark}\label{rem:2nd-order-relax-assumption-on-K}
The assumption on $K$ in (1) can be weakened to 
\begin{equation*}
K\in \Lambda^{0,s}(\mathrm{L}^\infty_y(\mathbb{R}^d;\mathrm{L}^q_x(\mathbb{R}^d;\mathbb{R}^d)))
\end{equation*}
for $s\in(2/3,1]$, $q>(d+1)/s$ and $p>d/q$. (This space is defined in \cref{def:parameterised-Holder-space}). This includes the result stated in the theorem as the case $q=\infty$. 

The assumption on $K$ in (2) can be weakened when $s>1$ to 
\begin{equation*}
W\in \Lambda^{1,(s-1)}(\mathrm{L}^{q}(\mathbb{R}^d;\mathbb{R}^d))
\end{equation*}
under the additional assumption that $p\ge4$. This implies the assumption on $K$ in the theorem statement. Note that the $s\le 1$ case of this weakened assumption was included in the relaxation of (1) directly above.

In each case the proof is given for these weakened assumptions.
\end{remark}

\subsection{Empirical process \& Glivenko-Cantelli theorems for SDEs}\label{subsubsec:empirical-process-theory}

In this subsection we define the empirical process hinted at in the informal statement \cref{thm:informal-GV-thm}. We present first the definitions and results for first order systems.

\subsubsection{First order systems.}\label{subsec:1st-order-stochastis-process} Let $\tilde{\mathcal{C}}^{\alpha}$ be the set of vector fields given by
\begin{equation}
\tilde{\mathcal{C}}^{\alpha}=\{b:\norm{b}_{\mathrm{C}([0,T];\mathrm{C}^{0,\alpha}(\mathbb{R}^d;\mathbb{R}^d))}\le C\}
\end{equation}
for some $\alpha\in(0,1)$ and $C<\infty$ fixed. For any $b\in\tilde{\mathcal{C}}$ and any $N\in\mathbb{N}$ denote $(X^{b,i,N})^{N}_{i=1}$ as the solution to
\begin{equation}\label{eq:many-particle-system-empirical}
\left\{\begin{aligned}
&dX^{b,i,N}_t=b_t(X_t^{b,i,N})dt+dB^{i,N}_t,\\
&X^{b,i,N}_0=X_0^{i,N}
\end{aligned}\quad i=1,\dotsc,N\right.
\end{equation}
where $X_0^{i,N}$ and $B^{i,N}$ are the same as in \eqref{eq:many-particle-system}. Note that $(X^{b,i,N})_{i=1}^N$ are i.i.d. by construction. The \emph{empirical process} $(\mu^{b,N}_t)_{t\in[0,T],b\in\tilde{\mathcal{C}}^\alpha}$ is defined by
\begin{equation}\label{eq:empirical-process}
\mu^{b,N}_t=\averN\delta_{X^{b,i,N}_t}.
\end{equation}
Note that for any $b\in\tilde{\mathcal{C}}^{\alpha}$ and any $i\in\{1,\dotsc,N\}$, the law of $X^{b,i,N}_t$ is given by $f^b_t$, the solution to the following parabolic PDE:
\begin{equation}\label{eq:empirical-limit-pde}
\left\{\begin{aligned}
&\partial_t f_t^b+\nabla\cdot(b_tf_t^b)-\tfrac12 \Delta f_t^b=0,\qquad (t,x)\in(0,T)\times\mathbb{R}^d,\\
&f^b_0(x)=f_0(x)\text{ initial condition}.\\
\end{aligned}\right.
\end{equation}
We would like to be able to consider $\mu^{b,N}_t$ for \emph{random} $b\in\tilde{\mathcal{C}}^\alpha$. To do so we need that the stochastic process $(X^{b,i,N})_{i=1}^N$ indexed by $t\in[0,T]$ and $b\in\tilde{\mathcal{C}}^\alpha$ be (almost surely) continuous. Let us be precise about this for the benefit of readers less familiar with such notions. We wish to construct a (random) map (in other words a stochastic process) $\varphi$ defined by
\begin{align*}
\varphi:([0,T],|\cdot|)\times (\tilde{\mathcal{C}}^{\alpha},\mathrm{L}^\infty([0,T];\mathrm{L}^{-r,\infty}(\mathbb{R}^d;\mathbb{R}^d)))&\to (\mathbb{R}^d)^N\\
(t,b)&\mapsto (X^{b,i,N}_t)_{i=1}^N,
\end{align*}
for some $r>0$. The statement that $\varphi$ is continuous at a point $(t,b)\in [0,T]\times\tilde{\mathcal{C}}^{\alpha}$ means that for any sequence $t_n,b_n\in [0,T]\times\tilde{\mathcal{C}}^{\alpha}$ converging to $t,b$ as $n\to\infty$ in the topologies in the above display, we have that $ (X^{b_n,i,N}_{t_n})_{i=1}^N\to (X^{b,i,N}_t)_{i=1}^N$ as $n\to\infty$ in $\mathbb{R}^{d\cdot N}$.

 We ask that $\varphi$ be almost surely continuous, i.e.
\begin{equation}\label{eq:stochastic-process-is-almost-surely-cts}
\mathbb{P}(\,\forall (t,b)\in[0,T]\times\tilde{\mathcal{C}}^{\alpha}, \varphi\text{ is continuous at }(t,b))=1.
\end{equation}
This is a \emph{much stronger} requirement than with the quantifiers switched, i.e.
\begin{equation}\label{eq:stochastic-process-is-cts-at-each-point-almost-surely}
\forall (t,b)\in[0,T]\times\tilde{\mathcal{C}}^{\alpha}, \mathbb{P}(\varphi\text{ is continuous at }(t,b))=1.
\end{equation}
The former implies the latter but not vice versa.

The size of the index set $\tilde{\mathcal{C}}^{\alpha}$ causes a technical issue that although \eqref{eq:stochastic-process-is-cts-at-each-point-almost-surely} can be shown, it is impossible to show \eqref{eq:stochastic-process-is-almost-surely-cts}:
\begin{proposition}\label{lem:impossible-to-have-cts-process}
Let $f_0\in \mathrm{P}_p(\mathbb{R}^d)$ for some $p>1$. Then the process $(X^{b,i,N}_t)_{i=1}^N$ indexed by $b\in\tilde{\mathcal{C}}^{\alpha}$ and $t\in[0,T]$ cannot be modified to give an almost surely continuous process with the $\mathrm{L}^\infty([0,T];\mathrm{L}^{-r,\infty}(\mathbb{R}^d;\mathbb{R}^d))$ (any $r>0$) topology on $\tilde{\mathcal{C}}^{\alpha}$.
\end{proposition}
This is because, roughly speaking, constructing this process would give uniqueness for SDEs with \emph{random} $\alpha$-H\"older coefficients, and there are simple counterexamples. We refer the reader to the proof of \cref{lem:impossible-to-have-cts-process} for further details.

For this reason we define $\mathcal{C}^{\alpha}$ as the set of \emph{smooth} ($\mathrm{C}^1_b$ in $x$) vector fields in $\tilde{\mathcal{C}}^{\alpha}$, i.e.
\begin{equation}
\mathcal{C}^{\alpha}=\mathrm{C}([0,T];\mathrm{C}^1_b(\mathbb{R}^d;\mathbb{R}^d))\cap \tilde{\mathcal{C}}^{\alpha}.
\end{equation}
Note that $\mathcal{C}^\alpha$ contains vector fields with $\mathrm{C}^1_b$ norm arbitrarily large. Also $\mathcal{C}^\alpha$ is dense in $\tilde{\mathcal{C}}^\alpha$ in the $\mathrm{L}^\infty([0,T];\mathrm{L}^{-r,\infty}(\mathbb{R}^d;\mathbb{R}^d))$ topology for any $r>0$.

For this set of vector fields we \emph{can} construct a continuous stochastic process.
\begin{theorem}\label{thm:cts-empirical-process}
Let $f_0\in \mathrm{P}_p(\mathbb{R}^d)$ for some $p>1$. Then the process $(X^{b,i,N}_t)_{i=1}^N$ defined by \eqref{eq:many-particle-system-empirical} and indexed by $t\in[0,T],b\in\mathcal{C}^{\alpha}$ has a modification that is continuous, where $\mathcal{C}^{\alpha}$ is equipped with the $\mathrm{L}^\infty([0,T];\mathrm{L}^{-r,\infty}(\mathbb{R}^d;\mathbb{R}^d))$ topology (any $r\in(0,p)$). As a consequence, the same holds for the empirical process $(\mu^{b,N}_t)_{t\in[0,T],b\in\mathcal{C}^{\alpha}}$ given by \eqref{eq:empirical-process} above, mapping into the space of probability measures equipped with the weak topology.
\end{theorem}
In the style of language of \eqref{eq:stochastic-process-is-almost-surely-cts}, this theorem states that we can construct the process $\mu^{b,N}_t$ in such a way that
\begin{equation*}
\mathbb{P}\left(\begin{aligned}&\text{For any sequence }(t_n,b_n)\to(t,b)\text{ as }n\to\infty \text{ in }[0,T]\times\mathcal{C}^\alpha,\\
&\text{ we have }\mu^{b_n,N}_{t_n}\to \mu^{b,N}_t\text{ as }n\to\infty\text{ weakly in }\mathrm{P}(\mathbb{R}^d)).
\end{aligned}\right)=1
\end{equation*}
where the convergence of $b_n$ is in $\mathrm{L}^\infty([0,T];\mathrm{L}^{-r,\infty}(\mathbb{R}^d;\mathbb{R}^d))$. 

The inability to construct the process on the full set of $\alpha$-H\"older continuous vector fields $\tilde{\mathcal{C}}^{\alpha}$ means that the following results are a priori, in the sense that they must be applied to smoothed $(\mathrm{C}^1_b)$ vector fields, but are uniform in the degree of smoothness.

Our main result on this empirical process is that the Wasserstein distance between the empirical measure $\mu^{b,N}_t$ and the law $f_t^b$ has (polynomial in $N$) sub-Gaussian concentration about the initial distance $d_{\mathrm{MKW}}(\mu_0^N,f_0)$.
\begin{theorem}[Glivenko-Cantelli theorem for SDEs]\label{thm:GV-intro}
Let $f_0\in \mathrm{P}_p(\mathbb{R}^d)$ for some $2\ne p>1$. Assume that $\mathcal{C}\subset \mathcal{C}^{\alpha}$ obeys the metric entropy bound
\begin{equation}\label{eq:metric-entropy-of-C-assumption}
H(\varepsilon,\mathcal{C},\norm{\cdot}_{\mathrm{L}^\infty([0,T];\mathrm{L}^{-r,\infty}(\mathbb{R}^d;\mathbb{R}^d))})\le C\varepsilon^{-k}
\end{equation}
for some $r\in(1,p)$. Then it holds that
\begin{equation}\label{eq:supremum}
\nnorm{\sup_{t\in[0,T],b\in\mathcal{C}}d_{\mathrm{MKW}}(\mu_t^{b,N},f_t^b)-d_{\mathrm{MKW}}(\mu^N_0,f_0)}\le CN^{-\gamma},
\end{equation}
with
\begin{equation*}
\gamma=\frac1{2+\max(d,d/(p-1),k)}.
\end{equation*}
\end{theorem}
As discussed below the statement of \cref{thm:propagation-of-chaos}, we can easily use this bound to obtain estimates on the expectation of the Wasserstein distance.
\begin{corollary}\label{cor:GV-intro-expectation}
Under the assumptions of \cref{thm:GV-intro} it holds that
\begin{equation}
\mathbb{E}\sup_{t\in[0,T],b\in\mathcal{C}}d_{\mathrm{MKW}}(\mu_t^{b,N},f_t^b)\le \mathbb{E}d_{\mathrm{MKW}}(\mu^N_0,f_0)+CN^{-\gamma}, \quad \gamma=\frac{1}{2+\max(d,d/(p-1),k)}.
\end{equation}
Moreover, if $p$ is large enough depending only on $d,k$ then it holds that
\begin{equation}
\mathbb{E}\sup_{t\in[0,T],b\in\mathcal{C}}d_{\mathrm{MKW}}(\mu_t^{b,N},f_t^b)\le CN^{-\gamma}, \quad \gamma=\frac{1}{2+\max(d,k)}.
\end{equation}
\end{corollary}

\begin{remark}
Similar results with weaker non-polynomial rates may be easily obtained with minor modification of the proof for the case that different types of estimates on the metric entropy hold. In particular convergence to zero of \eqref{eq:supremum} as $N\to\infty$ will hold for any set $\mathcal{C}\subset \mathcal{C}^\alpha$ that is totally bounded  in the norm used in \eqref{eq:metric-entropy-of-C-assumption}.
\end{remark} 
\begin{remark}
Despite the density of $\mathcal{C}^\alpha$ in $\tilde{\mathcal{C}}^\alpha$ in the $\mathrm{L}^\infty([0,T];\mathrm{L}^{-r,\infty}(\mathbb{R}^d;\mathbb{R}^d))$ norm, we cannot replace the subset $\mathcal{C}$ with its closure in $\mathcal{C}^\alpha$ in this norm in the above theorem. This is due to the difficulty in defining the process considered over such a large index set (see \cref{lem:impossible-to-have-cts-process}).
\end{remark}
\begin{remark}
We call this result a Glivenko-Cantelli theorem as it implies that
\begin{equation*}
\mathcal{F}=\{\omega\mapsto h(X^b_t(\omega)):h\text{ is $1$-Lipschitz}, t\in[0,T],b\in\mathcal{C}\},
\end{equation*}
is a Glivenko-Cantelli class with respect to the Weiner measure (see \cref{def:glivenko-cantelli-class}).
\end{remark}

The proof of \cref{thm:GV-intro} also provides better estimates of weaker measures of distance, see \cref{prop:ULLN-for-h} in \cref{sec:empirical-process-theory} below. 

Applications of \Cref{thm:GV-intro} combined with well known metric entropy of function spaces \cite{Edmunds-Triebel,Function-spaces-III} gives explicit convergence rates. In particular, for the parabolic H\"older scale of spaces we obtain:
\begin{corollary}
Let $f_0\in \mathrm{P}_p(\mathbb{R}^d)$ for some $2\ne p>1$, and let $\mathcal{C}^{\alpha}_{para}$ be given by
\begin{equation*}
\mathcal{C}^{\alpha}_{para}=\mathrm{C}([0,T];\mathrm{C}^1_b(\mathbb{R}^d;\mathbb{R}^d))\cap \{b:\norm{b}_{\mathrm{C}^{0,\alpha}_{para}([0,T]\times\mathbb{R}^d;\mathbb{R}^d)}\le C\}
\end{equation*}
for some constants $C\in(0,\infty)$ and $\alpha\in(0,1)$. Then it holds that
\begin{equation}\label{eq:lln-holder-scale}
\nnorm{\sup_{t\in[0,T],b\in\mathcal{C}}d_{\mathrm{MKW}}(\mu_t^{b,N},f_t^b)-d_{\mathrm{MKW}}(\mu_0^N,f_0)}\le CN^{-\gamma}, \quad \gamma=\frac1{2+\max(\tfrac{d+2}\alpha,\tfrac{d}{p-1})}.
\end{equation}
\end{corollary}
As before, this estimate can be combined with estimates of the initial distance $d_{\mathrm{MKW}}(\mu_0^N,f_0)$ to obtain results corresponding to \cref{cor:GV-intro-expectation}. Similar results can be easily obtained for the non-parabolic spaces. While we do not claim that the $\gamma$ in \eqref{eq:lln-holder-scale} is optimal, the result is optimal for the H\"older scale in the sense that no such estimate is possible for $\alpha=0$. In fact:
\begin{proposition}\label{prop:counterexample}
Let $f_0\in \mathrm{P}_p(\mathbb{R}^d)$ for some $p>1$ and let $\mathcal{C}^{0}$ be given by
\begin{equation*}\label{eq:counterexample-class}
\mathcal{C}^{0}=\mathrm{C}([0,T];\mathrm{C}^1_b(\mathbb{R}^d;\mathbb{R}^d))\cap \{b:\norm{b}_{\mathrm{C}([0,T]\times\mathbb{R}^d;\mathbb{R}^d)}\le C\}
\end{equation*}
for some constant $C\in(0,\infty)$. Then it holds that
\begin{equation}\label{eq:inf-lower-bound-counterexample}
\inf_{N\ge1}\mathbb{E}\sup_{t\in[0,T],b\in\mathcal{C}}d_{\mathrm{BL}}(\mu_t^{b,N},f_t^b)\ge c>0.
\end{equation}
\end{proposition}
Note that the use of $d_{\mathrm{BL}}$ in \eqref{eq:inf-lower-bound-counterexample} is a stronger statement than if $d_{\mathrm{MKW}}$ were used as the Wasserstein metric generates a stronger topology. The proof of \cref{prop:counterexample} is provided in \cref{sec:counterexample} where a stochastic control problem is introduced, which is solvable only if no such uniform law of large numbers can hold.

\subsubsection{Second order systems.} The definitions in the second order case are analogous to those in the first order case, but with give them in full for completeness. Let $\tilde{\mathcal{C}}^{\alpha}$ be the set of vector fields given by
\begin{equation}
\tilde{\mathcal{C}}^{\alpha}=\{b:\norm{b}_{\mathrm{C}([0,T];\mathrm{C}^{0,\alpha}(\mathbb{R}^{d};\mathbb{R}^d))}\le C\}
\end{equation}
for some constants $C,\alpha$ with $\alpha\in(2/3,1)$. Define $(X^{b,i,N},V^{b,i,N})_{i=1}^N$ for $b\in\tilde{\mathcal{C}}^\alpha$ and $N\in\mathbb{N}$ as the solution to 
\begin{equation}\label{eq:2nd-order-many-particle-system-empirical}
\left\{\begin{aligned}
&dX^{b,i,N}_t=V^{b,i,N}_tdt,\\
&dV^{b,i,N}_t=b_t(X_t^{b,i,N})dt-\kappa V_t^{b,i,N}dt+dB^{i,N}_t,\\
&(X_0^{b,i,N},V^{i,N}_0)=(X^{i,N}_0,V^{i,N}_0),
\end{aligned} \quad i=1,\dotsc,N,
\right.
\end{equation}
where $X_0^{i,N},V_0^{i,N}$ and $B^{i,N}$ are the same as in \eqref{eq:2nd-order-many-particle-system}.

The \emph{empirical process} $(\mu^{b,N}_t)_{t\in[0,T],b\in\tilde{\mathcal{C}}^\alpha}$ is defined by
\begin{equation}\label{eq:2nd-order-empirical-process}
\mu^{b,N}_t=\averN\delta_{(X^{b,i,N}_t,V^{b,i,N}_t)}.
\end{equation}
For any $b\in\tilde{\mathcal{C}}^\alpha$ and $i\in\{1,\dotsc,N\}$, the law of $(X^{i,N}_t,V^{i,N}_t)$ is $f^b_t$ which solves the following degenerate parabolic PDE:
\begin{equation}\label{eq:2nd-order-empirical-limit-pde}
\left\{\begin{aligned}
&\partial_tf^b_t+v\cdot\nabla_xf^b_t-\kappa\nabla_v\cdot(vf^b_t)+b_t\cdot\nabla_vf_t^b-\tfrac12\Delta_v f_t^b=0,\quad (t,x)\in(0,T)\times\mathbb{R}^d\times\mathbb{R}^d,\\
&f^b_0(x,v)=f_0(x,v)\text{ initial condition}.\\
\end{aligned}\right.
\end{equation}
As before we must work with a dense smooth subset
\begin{equation}
\mathcal{C}^\alpha=\mathrm{C}([0,T];\mathrm{C}^1_b(\mathbb{R}^d;\mathbb{R}^d))\cap \tilde{\mathcal{C}}^\alpha.
\end{equation}
The stochastic process indexed by $\mathcal{C}^\alpha$ has a continuous modification.
\begin{theorem}\label{thm:2nd-order-cts-empirical-process}
Let $f_0\in \mathrm{P}_p(\mathbb{R}^{2d})$ for some $p>1$. Then the process $(X^{b,i,N}_t,V_t)_{i=1}^N$ indexed by $t\in[0,T],b\in\mathcal{C}^\alpha$ has a modification that is continuous, where $\mathcal{C}^\alpha$ is equipped with the $\mathrm{L}^\infty([0,T];\mathrm{L}^{-r,\infty}(\mathbb{R}^d;\mathbb{R}^d))$ topology (any $r>0$). As a consequence, the same holds for the empirical process $(\mu^{b,N}_t)_{t\in[0,T],b\in\mathcal{C}}$ in the weak topology on $\mathrm{P}(\mathbb{R}^d)$.
\end{theorem}
For the second order system the main result of this subsection is the following.
\begin{theorem}[Glivenko-Cantelli theorem for SDEs (second order case)]\label{thm:2nd-order-GV-intro}
Let $f_0\in \mathrm{P}_p(\mathbb{R}^{2d})$ for some $2\ne p>1$ and assume $\alpha\in(2/3,1)$. Assume that $\mathcal{C}\subset \mathcal{C}^\alpha$ obeys the metric entropy bound
\begin{equation}\label{eq:2nd-order-metric-entropy-of-C-assumption}
H(\varepsilon,\mathcal{C},\norm{\cdot}_{\mathrm{L}^\infty([0,T];\mathrm{L}^{-r,\infty}(\mathbb{R}^d;\mathbb{R}^d))})\le C\varepsilon^{-k}
\end{equation}
for some $r\in(1,p)$. Let $\mu^{b,N}$ and $f^b$ be defined by \eqref{eq:2nd-order-empirical-process} and \eqref{eq:2nd-order-empirical-limit-pde} respectively. Then it holds that
\begin{equation}\label{eq:2nd-order-supremum}
\nnorm{\left[\sup_{t\in[0,T],b\in\mathcal{C}}d_{\mathrm{MKW}}(\mu_t^{b,N},f_t^b)-cd_{\mathrm{MKW}}(\mu_0^N,f_0)\right]_+}\le CN^{-\gamma},
\end{equation}
where $\gamma$ is given by
\begin{equation}\label{eq:2nd-order-GV-gamma}
 \gamma=\frac1{2+\max(d,d/(p-1),k)}.
\end{equation}
\end{theorem}
As before we can use this result to obtain estimates like the following.
\begin{corollary}
Under the assumptions of \cref{thm:2nd-order-GV-intro}, the following holds
\begin{equation}
\mathbb{E}\sup_{t\in[0,T],b\in\mathcal{C}}d_{\mathrm{MKW}}(\mu_t^{b,N},f_t^b)\le C\mathbb{E}d_{\mathrm{MKW}}(\mu^N_0,f_0)+CN^{-\gamma},
\end{equation}
with $\gamma$ given by \eqref{eq:2nd-order-GV-gamma}. Moreover, if $p$ is large enough depending only on $d,k$ then it holds that
\begin{equation}
\mathbb{E}\sup_{t\in[0,T],b\in\mathcal{C}}d_{\mathrm{MKW}}(\mu_t^{b,N},f_t^b)\le CN^{-\gamma}, \quad \gamma=\frac{1}{2+\max(d,k)}.
\end{equation}
\end{corollary}
As in the first order case we can obtain results in the H\"older scale of spaces. In this case, however, it makes more sense to consider the usual non-parabolic spaces.
\begin{corollary}\label{cor:2nd-order-glivenco-cantelli-for-holder-functions}
Let $f_0\in \mathrm{P}_p(\mathbb{R}^{2d})$ for some $2\ne p>1$. Let $\alpha\in(2/3,1)$ and consider the class of $\alpha$-H\"older functions defined by
\begin{equation*}
\mathcal{C}=\mathrm{C}([0,T];\mathrm{C}^1_b(\mathbb{R}^d;\mathbb{R}^d))\cap \{b:\norm{b}_{\mathrm{C}^{0,\alpha}([0,T]\times\mathbb{R}^d;\mathbb{R}^d)}\le C\}.
\end{equation*}
Then it holds that
\begin{equation}
\nnorm{\left[\sup_{t\in[0,T],b\in\mathcal{C}}d_{\mathrm{MKW}}(\mu_t^{b,N},f_t^b)-cd_{\mathrm{MKW}}(\mu_0^N,f_0)\right]_+}\le CN^{-\gamma},
\end{equation}
where
\begin{equation*}
\gamma=\frac1{2+\max(\tfrac{d+1}\alpha,\tfrac d{p-1})}.
\end{equation*}
\end{corollary} 
As before this can be used to bound the expectation of the supremum.

\begin{remark}
In the second order case the exponent $\gamma$ is bounded below on the range of $\alpha$ considered, (for $p>2$), i.e.
\begin{equation}
\gamma=\frac1{2+\tfrac{d+1}\alpha}> \frac2{7+3d}>0.
\end{equation}
\end{remark}
\begin{remark}
	The lower bound of $2/3$ on $\alpha$ seems unlikely to be optimal in the sense that compactness methods would likely yield a Glivenko-Cantelli theorem for $\alpha\in(0,1)$ but without an explicit convergence rate. We do not pursue such results here.
\end{remark}
\section{Prior work and discussion}\label{sec:discussion}
There has been much prior work on propagation of chaos of the particle system \eqref{eq:many-particle-system}. This has been split between the \emph{noisy case} considered in this manuscript, and the \emph{noiseless case} where the driving Brownian motions are absent.
\subsection{Lipschitz interactions}
The first quantitative results in propagation of chaos are due to Dobrushin \cite{Dobrushin} in the noiseless case, and then later Sznitman in the case with noise considered in this work. Both these results rely on the interaction kernel $K$ being Lipschitz continuous. Dobrushin observed that the empirical measure $\mu^N$ is a weak solution to the limit equation and then established that, under the assumption that $K$ is Lipschitz, the limit equation is well-posed in the space of measures using the MKW distance, from which propagation of chaos then follows from the convergence of initial data. In the case with noise, the empirical measure is no longer a weak solution to the limit equation. To get around this problem, Sznitman \cite{Sznitman} developed a coupling method to prove propagation of chaos. This will be described in detail below in \cref{subsec:coupling-method}.
\subsection{Singularity only at the origin} A subsequent line of enquiry was into interaction kernels which are Lipschitz apart from a single singularity where $K$ or its derivative blows up in a specified manner. The case $K(x,y)=W(x-y)$ with $W$ Lipschitz away from the origin\footnote{In these cases one must disallow self-interaction, so that the force term $b^{N}_t(X^{i,N}_t)$ on the $i$th particle in \eqref{eq:many-particle-system} is replaced with $\frac1{N-1}\sum_{j=1,j\ne i}^NK(X^{i,N}_t,X^{j,N}_t)$. The results of this paper also apply to this case, see \cref{subsec:simple-extensions} below.} has received much attention as it models, for example, gravitational attraction.

\subsubsection{Noiseless case} In \cite{Jabin-Hauray-Vlasov} and later papers by various authors, propagation of chaos is established in the case without noise for interaction kernels satisfying the bounds $|W(x)|\le C|x|^{-\alpha}$ and $|\nabla W(x)|\le C|x|^{-\alpha-1}$ for some $\alpha<1$. As in the proof of Dobrushin, these works rely on weak-strong stability estimates on the limit equation. However, to avoid the singularity at the origin, control must be obtained over the minimum distance between particles, and this requires specially prepared initial particle positions to control these distances at the initial time. A comprehensive review is given in \cite{Jabin2014review}. 

\subsubsection{Noisy case} Subsequently to proving the results of this manuscript, the author was surprised to find that the noisy case has been considered \emph{harder} than the noiseless case. This is in stark contrast with the comparison of existence and uniqueness theory for ODEs and SDEs where noise allows for less regular vector fields. When, however, one considers that to handle a singularity at the origin one must control the distances between particles and avoid near collisions, this makes more sense. Among recent work along these lines is \cite{Fournier-Hauray-Landau} where propagation of chaos is obtained for a system similar to \eqref{eq:many-particle-system} with $W(z)=z|z|^{\alpha-1}$ for some $\alpha\in(0,1)$, (so $W$ is $\alpha$-H\"older continuous). Another recent work is \cite{Hauray-Salim-2015propagation} where the $1$-dimensional Vlasov-Poisson-Fokker-Planck equation is considered, and the interaction kernel is the $\operatorname{sign}$ function, i.e. constant except for a jump at the origin. As in the works mentioned in the previous paragraph, the proof in \cite{Fournier-Hauray-Landau} uses control over particle distances. A review is given in \cite{JabinWang2016review}.

\subsection{Bounded interactions or bounded potentials}
In a recent work \cite{Jabin-Wang} an intriguing combinatorial argument is made to prove propagation of chaos for systems with \emph{bounded} interaction kernels $W(z)\in \mathrm{L}^\infty$, later extended to bounded potentials \cite{Jabin-Wang-bounded-potentials} (roughly speaking $W(z)\in \mathrm{W}^{-1,\infty}$) in both the noisy and noiseless cases under the condition that $\operatorname{div} W=0$. These works rely on controlling the \emph{relative-entropy} between the solution to the $N$-particle Liouville equation and the limit solution. An advantage of these works over the results in this manuscript in the $K(x,y)=W(x-y)$ case is that the assumptions on the interaction kernel are weaker in the sense that $\mathrm{L}^\infty$ (even $\mathrm{W}^{-1,\infty}$) rather than H\"older regularity is asked. However, this comes at the cost of assuming that $W$ is divergence free, and rather surprisingly, rather non-generic assumptions on the initial datum $f_0$, which cannot be taken to be smooth with compact support, for example.

\subsection{The coupling method of Sznitman}\label{subsec:coupling-method} To prove propagation of chaos for Lipschitz interactions $K$ in the noisy case Sznitman introduced a coupling method, where the particles \eqref{eq:many-particle-system} are coupled to an auxiliary particle system with the vector field $b^N$ replaced by the vector field of the limit equation $b^\infty$. In the notations of \eqref{eq:many-particle-system-empirical}, the auxiliary particle system is $(X^{b^\infty,i,N})_{i=1}^N$. We give a heuristic description of the proof below.

\subsubsection{Heuristic description}
By the triangle inequality we observe that
\begin{equation}\label{eq:Sznitman-coupling}
\mathbb{E} d(\mu^N_t,f_t)\le \mathbb{E}d(\mu^N_t,\mu^{b^\infty,N}_t)+\mathbb{E}d(\mu^{b^\infty,N}_t,f_t).
\end{equation}
The second term is the expected difference between the empirical measure of $N$ i.i.d. samples from their law $f_t$, and so tends to zero as $N\to\infty$ by the law of large numbers. Using Lipschitz continuity of the vector field $b^\infty$ we obtain the bound
\begin{equation*}
|X^{i,N}_t-X^{b^\infty,i,N}_t|\le \norm{\nabla b^\infty}_{\mathrm{L}^\infty}\int^t_0|X^{i,N}_s-X^{b^\infty,i,N}_s|\,ds+\int^t_0|b^N_s(X^{i,N}_s)-b^\infty(X^{i,N}_s)|\,ds,
\end{equation*}
and one concludes via the Gr\"onwall inequality that
\begin{equation*}
d(\mu^N_t,\mu^{b^\infty,N}_t)\le e^{t\norm{\nabla b^\infty}_{\mathrm{L}^\infty}}\averN\int^t_0|b^N_s(X^{i,N}_s)-b^\infty(X^{i,N}_s)|\,ds.
\end{equation*}
That is, the particle system depends smoothly on the vector field. Then one uses Lipschitz continuity of $K$ to obtain that the vector field depends smoothly on the particle positions, and closes the argument.

\subsubsection{Limitations}
This coupling method relies heavily on stability estimates on the particle system, and uses no stability estimates on the limit PDE. This can be seen in \eqref{eq:Sznitman-coupling} where for the first term, one uses stability estimates, and on the second term, the law of large numbers is used. By coupling in this way, we are philosophically viewing the limit PDE as a perturbation of the particle system; viewing a smoother system as a perturbation of a rougher system.

\subsection{A new coupling method}
To get around the limitation of the coupling method above, we reverse the roles of the particle system and the limit PDE in \eqref{eq:Sznitman-coupling}. We wish to apply stability estimates on the limit equation to the second term in \eqref{eq:Sznitman-coupling} and the law of large numbers to the first term. This way we make better use of the two powerful tools at our disposal: \emph{estimates on parabolic PDEs} and \emph{the law of large numbers}.

To apply the law of large numbers to the first term, we must necessarily couple with a continuum object and not a discrete particle system. The only choice is to couple with $f^{b^N}_t$ defined by \eqref{eq:empirical-limit-pde}, that is, with the limit PDE with the vector field $b^\infty$ replaced with vector field of the particle system $b^N$. (Note that $f^{b^N}_t$ is a random variable, even though it is a continuum object). Again the discussion below is heuristic. We refer the reader to the proof of \cref{thm:propagation-of-chaos} in \cref{sec:propagation-of-chaos} below for a more complete and rigorous presentation.

\subsubsection{Heuristic description} By the triangle inequality we have
\begin{equation}\label{eq:new-coupling}
\mathbb{E} d(\mu^N_t,f_t)\le \mathbb{E}d(\mu^N_t,f^{b^N}_t)+\mathbb{E}d(f^{b^N}_t,f_t).
\end{equation}
This can be rewritten as
\begin{equation*}
\mathbb{E} d(\mu^N_t,f_t)\le \mathbb{E}d(\mu^{b^N,N}_t,f^{b^N}_t)+\mathbb{E}d(f^{b^N}_t,f^{b^\infty}_t).
\end{equation*}
Using stability estimates on the limit equation one readily obtains that it is sufficient to bound the first term on the right hand side by something that tends to zero as $N\to\infty$. 

For the first term we wish to apply the law of large numbers. We have, however, a problem. The particle system is identically distributed, but not independent, so we cannot apply the law of large numbers directly. Moreover, we know very little about $b^N$. Indeed, our lack of knowledge of how to estimate $b^N$ was our motivation for constructing this method. Because of this, we give up all hope of understanding $b^N$, and instead use the trivial bound
\begin{equation*}
d(\mu^{b^N,N}_t,f^{b^N}_t)\le \sup_{b\in\mathcal{C}}d(\mu^{b,N}_t,f^{b}_t)
\end{equation*}
where $b\in\mathcal{C}$ ranges over all possible vector fields.\footnote{Of course, in practice $\mathcal{C}$ will not be all possible vector fields, but merely those in some norm bounded set. Thus there will be some asymptotically (as $N\to\infty$) small chance that $b\not\in\mathcal{C}$, which must be handled separately. We omit this here for brevity.} This coupling and the coupling of Sznitman are illustrated in \cref{fig:couplings}.

One is then left with the problem of bounding $\mathbb{E} \sup_{b\in\mathcal{C}}d(\mu^{b,N}_t,f^{b}_t)$. If the supremum where outside the expectation this would be easy, as it is the empirical measure of $N$ i.i.d. samples compared to their law $f^b_t$ and the usual law of large numbers applies. In this way, we have exchanged the non-independence of the particles with taking a supremum over a very large set. This technique is commonly used in proving consistency of estimators in theoretical statistics \cite{Van-der-Vaart-Wellner}, but to the authors knowledge has not been applied to the problem of propagation of chaos in this way before.

\begin{figure}
\centering
\begin{tikzcd}[column sep=10em,row sep=10em]
{\mu^N=\mu^{b^N,N}} \arrow{r}[above]{\text{Stability of SDE}}[below]{d(\mu^{b^N,N},\mu^{b^\infty,N})} \arrow{d}[sloped,below]{\text{Uniform LLN}}[sloped,above]{\sup_{b\in\mathcal{C}}d(\mu^{b,N},f^b)}&
{\mu^{b^\infty,N}}  \arrow{d}[sloped,above]{\text{LLN}}[sloped,below]{d(\mu^{b^\infty,N},f^{b^\infty})}\\
{f^{b^N}}\arrow{r}[below]{\text{Stability of PDE}}[above]{d(f^{b^N},f^{b^\infty})}&
{f^{b^\infty}=f}
\end{tikzcd}
\caption{Illustration of the coupling method of Sznitman and the coupling method proposed in this manuscript. To compare $\mu^N$ and $f$ one can either go right then down, which is the coupling method of Sznitman, or down then right, which is the coupling method proposed here.}
\label{fig:couplings}
\end{figure}
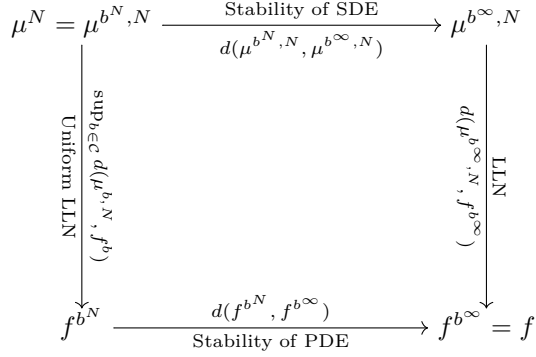
That this supremum can be bounded (\cref{thm:GV-intro}) is perhaps surprising. The proof crucially relies on the existence of a differentiable stochastic flow associated with the SDE \eqref{eq:many-particle-system-empirical} for a single particle.

\subsection{Discussion}

\subsubsection{Simple extensions}\label{subsec:simple-extensions}
\paragraph{\it No self-interaction}
The propagation of chaos result \cref{thm:propagation-of-chaos} can be easily extended to the case of no self-interaction, where $b^{N}_t(X^{i,N}_t)$ on the $i$th particle in \eqref{eq:many-particle-system} is replaced with $\frac1{N-1}\sum_{j=1,j\ne i}^NK(X^{i,N}_t,X^{j,N}_t)$. This may be done by considering instead the interaction kernel $\tilde{K}$ given by
\begin{equation*}
\tilde{K}(x,y)=K(x,y)-K(x,x).
\end{equation*}
Note that as $K$ is always at least bounded and continuous, there is no problem defining $K(x,x)$ and it will be uniformly bounded.
\paragraph{\it Multi-particle interactions} \Cref{thm:propagation-of-chaos} also easily extends to the case where $b^N$ is instead given by
\begin{equation*}
b^N_t(x)=\frac1{N^m}\sum_{i_1,\dotsc,i_m=1}^NK(x,X^{i_1,N}_t,\dotsc, X^{i_m,N}_t)
\end{equation*}
for $K\in \mathrm{C}^{0,\alpha}$. All the additional work happens at the PDE level and is straightforward.

\subsection{Open questions}

\subsubsection{The $\mathrm{C}^{0,0+}$ barrier and the curse of dimensionality}
The results of this paper show that, in the presence of noise, the regularity barrier of Lipschitz continuity of $K$ for quantitative propagation of chaos can be reduced to $K$ being H\"older continuous. However, stochastic flows are known to exist for vector fields in $\mathrm{L}^p$, $p>d$ (see \cite{Fedrizzi-Flandoli} and \cite{Banos-Duedahl} for different approaches to this problem). This leads to the following question:

\textit{Can the barrier in the noisy case be reduced to $K\in \mathrm{W}^{s,p}$, $p>d$, $s\ge0$ or $s>0$?}

The method used in this work fails in this case, as it requires the vector fields to be uniformly bounded for the key estimate \cref{cor:lipschitz-dependence-of-SDEs-upon-b}. However, one might consider applying stochastic flows to the original $Nd$ dimensional system, leading to the following related question:

\textit{Can one use these stochastic flow results on the particle system \eqref{eq:many-particle-system} and adapt the coupling method of Sznitman directly?}

In this work we have avoided this by using only the existence of a differentiable stochastic flow for a \emph{single particle}, a system of fixed dimension $d$. The main obstacle in applying stochastic flow results for $L^p$ ($p>d$) drifts to the whole system is that the dimension of the particle system \eqref{eq:many-particle-system} is $Nd$ which blows up as $N\to\infty$, and is eventually bigger than any finite $p$. However, it is conceivable that the special structure of \eqref{eq:many-particle-system} could be exploited to bypass this obstacle.

\section{Empirical process \& Glivenko-Cantelli}\label{sec:empirical-process-theory}
In this section we will prove the Glivenko-Cantelli results (\cref{thm:GV-intro,thm:2nd-order-GV-intro}) and \cref{prop:ULLN-for-h} which will be used in the proof of the propagation of chaos result in \cref{sec:propagation-of-chaos}. Before we move on to these results we will begin by establishing that the stochastic process $X^{b,i,N}$ is almost surely continuous.
\subsection{The stochastic process}
In this subsection we will prove that the stochastic process $(X_t^{b,i,N})_{i=1}^N$ indexed by $b\in\mathcal{C}$ has a continuous modification (\cref{thm:cts-empirical-process,thm:2nd-order-cts-empirical-process}) and show that it is impossible to construct a continuous modification of the same process indexed by the larger set $\tilde{ \mathcal{C}}$ (\cref{lem:impossible-to-have-cts-process}). As the proof in the second order case is no harder, we leave it to the reader. Furthermore, we may without loss of generality consider the single particle ($N=1$) case due to independence of the particles. For this reason we drop the $i,N$ indices in this subsection.

Key to understanding both continuity for $\mathcal{C}$ and discontinuity for $\tilde{\mathcal{C}}$ is the observation that if $Y_t=X_t-B_t$ and $X_t$ solves \eqref{eq:SDE-first-order-no-indices}, then $Y_t$ solves 
\begin{equation}\label{eq:SDE-transformation-to-ODE}
dY_t=b_t(Y_t+B_t)dt, \quad Y_0=X_0,
\end{equation}
which is an ODE with drift vector field $\tilde{b}_t(x,\omega)=b_t(x+B_t(\omega))$.
\begin{proof}[Proof of \cref{lem:impossible-to-have-cts-process}]
	We argue by contradiction. Suppose that an almost surely continuous version exists, and without loss of generality that $d=1$. To start with assume that $f_0=\delta_0$. Define the random vector field $\tilde{b}_t(x,\omega)=\min(|x-B_t(\omega)|^\alpha,1)$. Then by the computation above we deduce that $Y^{\tilde{b}}_t=X^{\tilde{b}}_t-B_t$ almost surely solves the ODE
	\begin{equation}\label{eq:ODE-non-unique}
	dY_t=\min(|Y_t|^\alpha,1)dt,\quad Y_0=0.
	\end{equation}
	Note that this does not uniquely determine $Y$ as the above ODE does not have unique solutions. However, as the process is continuous we can identify $Y_t^{\tilde{b}}$ as the unique limit of any sequence of $Y^{b^n}$ with (random) $b^n\in \mathrm{C}^1_b$ and $b^n\to \tilde{b}$ in $\mathrm{L}^{-r,\infty}(\mathbb{R}^d;\mathbb{R}^d)$ almost surely. As these approximating vector fields are in $\mathrm{C}^1_b$ the path $X^{b^n}_t$ is unique for each $n$. But we can easily construct a sequence $b^n$ such that $Y^{b^n}$ to converge to either the zero solution to \eqref{eq:ODE-non-unique} or a non-zero solution, contradicting uniqueness of this limit.
	
	Extending this proof to more general initial conditions than $f_0=\delta_0$ may be done by replacing the function $\min(|Y_t|^\alpha,1)$ with a vector field in $\mathrm{C}^{0,\alpha}$ that exhibits non-uniqueness for the corresponding ODE at every point in $\mathbb{R}$. We leave this to the reader.
\end{proof} 
\begin{proof}[Proof of \cref{thm:cts-empirical-process}]
	Away from a fixed null set, $\mathcal{N}$ say, $B_t$ is a continuous path on $[0,T]$. By the computation around \cref{eq:SDE-transformation-to-ODE}, we have that for any $b\in\mathcal{C}$, $Y^b_t=X^b_t-B_t$ is the solution to a random ODE, and the random vector field is continuous and has spatial Lipschitz constant bounded by $L_b:=\norm{b}_{\mathrm{C}([0,T];\mathrm{C}^1_b(\mathbb{R}^d;\mathbb{R}^d))}$ away from the null set $\mathcal{N}$. Hence we can solve this ODE to construct $Y_t^b$ (and thus also $X_t^b$) uniquely. Moreover, by standard Gr\"onwall estimates on the solution we deduce that for any $\tilde{b}\in\mathcal{C}$ it holds that
	\begin{equation*}
	\sup_{t\in[0,T]}|X_t^b-X_t^{\tilde{b}}|\le T\exp(L_bT)\norm{b-\tilde{b}}_{\mathrm{L}^\infty([0,T];\mathrm{L}^\infty(\mathbb{R}^d;\mathbb{R}^d))}
	\end{equation*}
	away from $\mathcal{N}$. To complete the proof it now suffices to replace this $\mathrm{L}^\infty$ estimate with an $\mathrm{L}^{-r,\infty}$ estimate. This may be done by a simple localisation argument. We leave this to the reader.
\end{proof}
\begin{remark}
 Although as evidenced by \cref{lem:impossible-to-have-cts-process} the process indexed by $\tilde{\mathcal{C}}$ cannot have an almost surely continuous version, and in the proof of continuity (\cref{thm:cts-empirical-process}) of the process indexed by $\mathcal{C}$ we used the $\mathrm{C}^1_b$ bound, we nevertheless have uniform `Lipschitz continuity at each point' in the sense that, due to \cref{cor:lipschitz-dependence-of-SDEs-upon-b} (presented below), we have
	\begin{equation*}
	\sup_{b\in{\tilde{\mathcal{C}}}}\mathbb{E}\sup_{t\in[0,T],\tilde{b}\in{\mathcal{C}}}\frac{|X^{b}_t-X^{\tilde{b}}_t|}{\norm{b-\tilde{b}}}<\infty
	\end{equation*} 
	where the norm on $b-\tilde{b}$ is $\mathrm{L}^{-r,\infty}([0,T]\times\mathbb{R}^d;\mathbb{R}^d)$. It is this estimate that will be key to the later analysis, and we will never use the $\mathrm{C}^1_b$ norm of the vector fields considered.
\end{remark}

\subsection{Estimates on the SDEs}
 Before we begin the proof of \cref{thm:GV-intro} proper, we will obtain some preliminary estimates on the SDEs:
  \begin{equation}\label{eq:SDE-first-order-no-indices}
 dX_t=b_t(X_t)dt+dB_t
 \end{equation}
 and
 \begin{equation}\label{eq:SDE-for-stochastic-flow}
 \left\{\begin{aligned}
&dX_t=V_tdt,\\
&dV_t=-b(X_t)dt-\kappa V_tdt+dB_t.
\end{aligned}\right.
\end{equation}

\subsubsection{Growth bounds}
We first obtain some simple a priori growth estimates for \eqref{eq:SDE-first-order-no-indices} and \eqref{eq:SDE-for-stochastic-flow} which will be used throughout the sequel.
\begin{lemma}\label{lem:SDE-bounds}
Let $X$ solve \eqref{eq:SDE-first-order-no-indices} then 
\begin{equation*}
\sup_{t\in[0,T]}|X_t|\le |X_0|+C\norm{b}_{\mathrm{L}^\infty([0,T]\times\mathbb{R}^d;\mathbb{R}^d)}+\sup_{t\in[0,T]}|B_t|.
\end{equation*}
Let $(X,V)$ solve \eqref{eq:SDE-for-stochastic-flow} then
\begin{equation*}
\sup_{t\in[0,T]}(|X_t|+|V_t|)\le C(|X_0|+|V_0|+\norm{b}_{\mathrm{L}^\infty([0,T]\times\mathbb{R}^d;\mathbb{R}^d)}+\sup_{t\in[0,T]}|B_t|).
\end{equation*}
As a consequence, $\sup_{t\in[0,T]}|X_t|$ (respectively $\sup_{t\in[0,T]}(|X_t|+|V_t|)$) possesses as many moments as $|X_0|$ (respectively $|X_0|+|V_0|$).
\end{lemma}
\begin{proof}
The first claim on \eqref{eq:SDE-first-order-no-indices} is immediate from the definition of solution in integral form. For the first claim on \eqref{eq:SDE-for-stochastic-flow} we first estimate
\begin{equation*}
\begin{aligned}
|X_t|&\le |X_0|+\int^t_0|V_s|\,ds\\
|V_t|&\le |V_0|+t\norm{b}_{\mathrm{L}^\infty([0,T]\times\mathbb{R}^d;\mathbb{R}^d)}+\sup_{s\in[0,t]}|B_s|+|\kappa|\int^t_0|V_s|\,ds
\end{aligned}
\end{equation*} 
and then conclude with the Gr\"onwall inequality on $|X_t|+|V_t|$. The remaining claims now follow from the triangle inequality and well known results on Brownian motion (see e.g. \cite{brownian-motion}). We omit the details.
\end{proof}
\subsubsection{Reference processes}
Next we define a `reference process' in each of the first and second order cases. This process will have the property that the difference between it and our actual process will be sub-Gaussian. For the first order case we define $(\widetilde{X}^{i,N}_t)_{i=1}^N$ as 
\begin{equation}\label{eq:1st-order-reference-process}
\widetilde{X}^{i,N}_t=X_0^{i,N},\qquad i=1,\dotsc,N,\qquad t\in[0,T].
\end{equation}
While in the second order case we instead define $(\widetilde{X}^{i,N},\widetilde{V}^{i,N}_t)_{i=1}^N$ as the solution to the following ODE with random initial condition:
\begin{equation}\label{eq:2nd-order-reference-process-ODE}
\begin{aligned} 
d\widetilde{X}^{i,N}_t&=\widetilde{V}^{i,N}_tdt,\\
d\widetilde{V}^{i,N}_t&=-\kappa \widetilde{V}^{i,N}_tdt,\qquad i=1,\dotsc,N\\
(\widetilde{X}^{i,N}_0,\widetilde{V}^{i,N}_0)&=(X^{i,N}_0,V^{i,N}_0).
\end{aligned} 
\end{equation} 
In both first and second order cases the reference process is nothing other than the solution to the corresponding SDE with $b=0$ and the driving noise removed.

Being a linear ODE, the equation \eqref{eq:2nd-order-reference-process-ODE} can be explicitly solved to give
\begin{equation}
(\widetilde{X}^{i,N}_t,\widetilde{V}^{i,N}_t)=\left(X_0^{i,N}+V_0^{i,N}\frac{1-e^{-\kappa t}}{\kappa},V^{i,N}_0e^{-\kappa t}\right),\qquad i=1,\dotsc,N.
\end{equation}
We further define (in each case) the empirical measure corresponding to the reference process as $\widetilde{\mu}^N_t$ and the common law of each reference particle as $\widetilde{f}_t$. (Although $\widetilde{f}_t$ is the solution to a transport equation, we do not have explicit need of this fact.)

As discussed above, the reason we consider these reference processes is that the increment between the stochastic process we care about and the reference process is sub-Gaussian, even though each individual process may not be sub-Gaussian (which will be the case if the initial measure $f_0$ is not sub-Gaussian).
\begin{lemma}\label{lem:SDE-subgaussian-bounds}
Let $Z_t^{b,i,N}=X^{b,i,N}_t-\widetilde{X}_t^{i,N}$ for the first order case (resp. $Z_t^{b,i,N}=(X^{b,i,N}_t-\widetilde{X}_t^{i,N},V^{b,i,N}_t-\widetilde{V}_t^{i,N})$ for the second order case), where $X^{b,i,N}_t$ solves \eqref{eq:many-particle-system-empirical} for $b\in\mathcal{C}$ and $\widetilde{X}^{i,N}_t$ is the reference process defined by \eqref{eq:1st-order-reference-process}, (resp. $(X^{b,i,N}_t,V^{b,i,N})$ solves \eqref{eq:2nd-order-many-particle-system-empirical} for $b\in\mathcal{C}$ and $(\widetilde{X}^{i,N}_t,\widetilde{V}^{i,N}_t)$ is the reference process defined by \eqref{eq:2nd-order-reference-process-ODE}.) Then for each $i$, we have the following almost sure bound:
\begin{equation*}
\sup_{t\in[0,T]}|Z^{b,i,N}_t|\le C\left(1+\sup_{t\in [0,T]}|B_t^{i,N}|\right),
\end{equation*}
and the following time increment bound for any $t\in[0,T],\varepsilon>0$,
\begin{equation*}
\mathbb{E}\sup_{s\in[0,T],|s-t|^{1/3}\le \varepsilon}|Z_t^{b,i,N}-Z_s^{b,i,N}|\le C\varepsilon.
\end{equation*}
\end{lemma}
\begin{proof}
We first prove the almost sure bounds. The first inequality in the first order case follows directly from the integral form of the SDE \eqref{eq:SDE-first-order-no-indices}, noting that $\widetilde{X}^{i,N}_t$ is nothing other than the initial condition $X^{i,N}_0$. For the second order case we first observe that
\begin{equation}
d(e^{\kappa t}V_t)=e^{\kappa t}b_t(X_t)dt+e^{\kappa t}dB_t,
\end{equation}
(where we have omitted the $b,i,N$ indices for brevity), so that
\begin{equation}\label{eq:sde-integrating-factor}
\begin{aligned}
(e^{\kappa t}V_t-V_0)&=\int^t_0e^{\kappa s}b_s(X_s)\,ds+\int^t_0e^{\kappa s}dB_s\\
&=\int^t_0e^{\kappa s}b_s(X_s)\,ds+e^{\kappa t}B_t-\kappa \int^t_0e^{\kappa s}B_s\,dt
\end{aligned}
\end{equation} 
by stochastic integration by parts. From this the bound on $|V_t-\widetilde{V}_t|=|V_t-e^{-\kappa t}V_0|$ is easily deduced as $b\in\mathcal{C}$ is uniformly bounded. The bound on $|X_t-\widetilde{X}_t|$ is now deduced by integrating this bound on $[0,t]$.

Now we prove the time increment estimates. For the first order system we have, from the integral form of the ODE,
\begin{align*}
\sup_{s\in[0,T],|s-t|^{1/3}\le \varepsilon} |(X_t-\widetilde{X}_t)-(X_s-\widetilde{X}_s)|&\le \varepsilon^2\sup_{b\in\mathcal{C}}\norm{b}_{\mathrm{L}^\infty([0,T]\times\mathbb{R}^d;\mathbb{R}^d)}\\
&\quad+\sup_{s\in[0,T],|t-s|^{1/3}\le \varepsilon}|B_s-B_t|
\end{align*}
The supremum over $\mathcal{C}$ is bounded by a constant. That the remaining part has expectation bounded by a constant times $\varepsilon$ can either be seen as a consequence of the law of the iterated logarithm (see e.g. \cite{brownian-motion}) or that the $1/3$-H\"older norm of Brownian motion has finite expectation (see e.g. \cite{Hytonen}).

The corresponding estimate for the second order case is similar using instead the integral form of \eqref{eq:sde-integrating-factor}. We leave it to the reader.
\end{proof}
Using the reference processes we can estimate the Wasserstein distance $d_{\textrm{MKW}}(\mu^{b,N}_t,f^b)$ using the following inequality. Note that, as is evident from its proof, establishing \eqref{eq:mkw-triangle-with-reference} requires no properties of the particle system other than the relationship of the empirical measures to the laws.
\begin{lemma}\label{lem:mkw-triangle-with-references}
Let $X^{b,i,N}_t$ solve \eqref{eq:many-particle-system-empirical} for $b\in\mathcal{C}$ and $\widetilde{X}^{i,N}_t$ be the reference process defined by \eqref{eq:1st-order-reference-process}. Then the following holds:
\begin{equation}\label{eq:mkw-triangle-with-reference}
\begin{aligned}
&d_{\mathrm{MKW}}(\mu^{b,N}_t,f^b_t)\le d_{\mathrm{MKW}}(\widetilde{\mu}^N_t,\widetilde{f}_t)+\\
&\quad +\sup_{h\in\mathrm{Lip1}}\left(\averN (h(X^{b,i,N}_t)-h(\widetilde{X}^{i,N}_t))-\mathbb{E}(h(X^{b,i,N}_t)-h(\widetilde{X}^{i,N}_t))\right).
\end{aligned}
\end{equation}
In the second order case where $(X^{b,i,N}_t,V^{b,i,N})$ solve \eqref{eq:2nd-order-many-particle-system-empirical} for $b\in\mathcal{C}$ and $(\widetilde{X}^{i,N}_t,\widetilde{V}^{i,N}_t)$ is the reference process defined by \eqref{eq:2nd-order-reference-process-ODE} the corresponding inequality to \eqref{eq:mkw-triangle-with-reference} holds, i.e. with $X^{b,i,N}$ replaced with $(X^{b,i,N},V^{b,i,N})$ and so on. We omit writing this inequality for brevity.
\end{lemma}
\begin{proof}
We only give the proof in the first order case for brevity. The second order case is analogous and we leave it to the reader. We have
\begin{align*}
d_{\mathrm{MKW}}(\mu^N_t,f^b_t)&=\sup_{h\in\mathrm{Lip1}}\left(\averN h(X^{b,i,N}_t)-\mathbb{E}h(X^{b,i,N}_t)\right)\\
&\le\sup_{h\in\mathrm{Lip1}}\left(\averN( h(X^{b,i,N}_t)-h(\widetilde{X}^{i,N}_t))-\mathbb{E}(h(X^{b,i,N}_t)-h(\widetilde{X}^{i,N}_t))\right)\\
&\quad+ \sup_{h\in\mathrm{Lip1}}\left(\averN h(\widetilde{X}^{i,N}_t)-\mathbb{E}h(\widetilde{X}^{i,N}_t)\right).
\end{align*}
The final supremum is nothing other that $d_{\mathrm{MKW}}(\widetilde{\mu}^N_t,\widetilde{f}_t)$. The proof is complete.
\end{proof}
As the reference processes are simple (by choice) and their evolution is deterministic, we have the following control over the distance of the reference empirical measure to the law.
\begin{lemma}\label{lem:reference-process-mkw-bound}
In the first and second order cases the following holds:
\begin{equation}
\sup_{t\in[0,T]}d_{\mathrm{MKW}}(\widetilde{\mu}^N_t,\widetilde{f}_t)\le cd_{\mathrm{MKW}}(\mu^N_0,f_0).
\end{equation}
Moreover, in the first order case $c$ may be taken to be equal to $1$.
\end{lemma}
\begin{proof}
For the first order case the claim (with $c=1$) is obvious from the definitions. For the second order case we argue directly by evolving an initial coupling between $f_0$ and $\mu^N_0$ along the trajectories of the ODE flow given by \eqref{eq:2nd-order-reference-process-ODE}. Indeed, let $\pi_0\in \mathrm{P}(\mathbb{R}^{2d}\times\mathbb{R}^{2d})$ be any coupling between $\mu^N_0$ and $f_0$, and define $\pi_t$ as the pushforward of $\pi_0$ by the (autonomous, deterministic, smooth) flow $\phi_{t}$ of the ODE \eqref{eq:2nd-order-reference-process-ODE}. Then $\pi_t\in \mathrm{P}(\mathbb{R}^{2d}\times\mathbb{R}^{2d})$ is a coupling of $\widetilde{\mu}^N_t$ and $\widetilde{f}_t$ and we have the bound
\begin{align*}
\int |(x^1,v^1)-(x^2,v^2)|&\,d\pi_t(x^1,v^1,x^2,v^2)\\
&=\int |\phi_t(x^1,v^1)-\phi_t(x^2,v^2)|\,d\pi_0(x^1,v^1,x^2,v^2)\\
&\le \norm{\nabla \phi_t}_{\mathrm{L}^\infty(\mathbb{R}^{2d};\mathbb{R}^{2d})}\int  |(x^1,v^1)-(x^2,v^2)|\,d\pi_0(x^1,v^1,x^2,v^2).
\end{align*} 
The ODE flow $\phi_t$ is uniformly Lipschitz in $(x,v)$ over times $t\in[0,T]$, so the $\mathrm{L}^\infty$ norm above is bounded by a constant. Thus $d_{\mathrm{MKW}}(\widetilde{f}_t,\widetilde{\mu}^N_t)$ is bounded by a constant times the final integral in the above display. The claim of the lemma now follows by taking the infimum over all couplings $\pi_0$ between $f_0$ and $\mu_0^N$.
\end{proof}

\subsubsection{Local Lipschitz dependence upon the field}
We now recall that in both the first and second order cases the SDE generates a differentiable stochastic flow. The first order case is established in \cite{Flandoli-Transport-equation} (see also \cite{Banos-Duedahl,mohammed2015,Fedrizzi-Flandoli} for results along the same lines). For the second order case see \cite{wang2015degenerate} (see also \cite{fedrizzi2016regularity}). We have need of a simple corollary that provides global bounds in space with a weight.
\begin{theorem}\label{thm:existence-of-stochastic-flow}
Let $b\in \mathrm{L}^\infty([0,T];\mathrm{C}^{0,\alpha}(\mathbb{R}^d;\mathbb{R}^d))$, with $\alpha\in(0,1)$ (resp. $\alpha\in(2/3,1)$) then the SDE \eqref{eq:SDE-first-order-no-indices} (resp. \eqref{eq:SDE-for-stochastic-flow}) generates a $\mathrm{C}^{1,\alpha'}$ stochastic flow (see \cref{def:stochastic-flow}) $\phi_{s,t}:\mathbb{R}^d\to\mathbb{R}^d$ (resp. $\mathbb{R}^{2d}\to\mathbb{R}^{2d}$), for $\alpha'$ depending only upon $\alpha$. For any $p\in[1,\infty)$ and $r>0$ there is a constant $C_{p,r}$ depending only on $p,r$ and $\norm{b}_{\mathrm{L}^\infty([0,T];\mathrm{C}^{0,\alpha}(\mathbb{R}^d;\mathbb{R}^d))}$ but not $b$ itself, such that
\begin{equation}\label{eq:Lrinfty-estimate-on-first-order-flow}
\nnorm{\sup_{0\le s\le t\le T}\norm{\nabla\phi_{s,t}}_{\mathrm{L}^{-r,\infty}(\mathbb{R}^d;\mathbb{R}^{d\times d})}}_p\le C_{p,r}<\infty
\end{equation}
(resp. the same claim with $\mathbb{R}^d$ replaced with $\mathbb{R}^{2d}$).
\end{theorem}
As discussed above, the existence of the stochastic flow is shown elsewhere, so it suffices to prove \eqref{eq:Lrinfty-estimate-on-first-order-flow}.
\begin{proof}[Proof of \eqref{eq:Lrinfty-estimate-on-first-order-flow}]
	We give the proof in the first order case, the second order proof being analogous.
 Local estimates of the form
		\begin{equation*}
		\nnorm{\sup_{0\le s\le t\le T}\norm{\nabla\phi_{s,t}}_{\mathrm{L}^\infty(Q;\mathbb{R}^{d\times d})}}_p\le C
		\end{equation*}
		for $p\in[1,\infty)$ and $Q$ a unit cube in $\mathbb{R}^d$ follow from \cite{Flandoli-Transport-equation}, and as the bounds upon $b$ are global, these estimates are uniform over unit cubes $Q$. Now let $A_n$ for $n=1,2,3,\dotsc$ be the annulus $\{2^n\le |x|\le 2^{n+1}\}$ and $A_0=\{|x|\le 2\}$ be subsets of $\mathbb{R}^d$. Then it holds that
		\begin{equation}\label{eq:split-into-annuli}
		\sup_{0\le s\le t\le T}\norm{\nabla\phi_{s,t}}_{\mathrm{L}^{-r,\infty}(\mathbb{R}^d;\mathbb{R}^{d\times d})}\le \sum_{n=0}^\infty 2^{-nr}\sup_{0\le s\le t\le T}\norm{\nabla \phi_{s,t}}_{\mathrm{L}^\infty(A_n;\mathbb{R}^{d\times d})}.
		\end{equation}
		Each $A_n$ can be covered by $m_n$ unit cubes $Q_{n,i}$ and $m_n$ can be chosen to be at most $C2^{nd}$. Hence, it holds that
		\begin{equation*}
		\nnorm{\norm{\nabla \phi_{s,t}}_{\mathrm{L}^\infty(A_n;\mathbb{R}^{d\times d})}}_p\le \nnorm{\max_{i=1}^{m_n}\norm{\nabla \phi_{s,t}}_{\mathrm{L}^\infty(Q_{n,i};\mathbb{R}^{d\times d})}}_p.
		\end{equation*}
		We apply the elementary inequality $\nnorm{\max_{i=1}^m |X_i|}_p\le C_pm^{1/p}\max_{i=1}^m\nnorm{X_i}$ which may be obtained by bounding the maximum of the $X_i$ with the sum of the $|X_i|$. This yields
		\begin{equation}\label{eq:final-annulus-bound}
		\nnorm{\norm{\nabla \phi_{s,t}}_{\mathrm{L}^\infty(A_n;\mathbb{R}^{d\times d})}}_p\le C_p2^{nd/p}.
		\end{equation}
		Therefore, combining \eqref{eq:final-annulus-bound} and \eqref{eq:split-into-annuli} we obtain 
		\begin{equation*}
		\nnorm{\sup_{0\le s\le t\le T}\norm{\nabla\phi_{s,t}}_{\mathrm{L}^{-r,\infty}(\mathbb{R}^d;\mathbb{R}^{d\times d})}}_p
		\le C_p\sum_{n=0}^\infty 2^{nd/p-nr}
		\end{equation*}
		and this sum is convergent for all $p$ sufficiently large. The estimate for smaller $p$ then follows by bounding with the estimate for larger $p$.	
\end{proof}

We obtain the following corollary of \cref{thm:existence-of-stochastic-flow}, which holds for both the first and second order systems.
\begin{corollary}\label{cor:lipschitz-dependence-of-SDEs-upon-b}
Let $b\in\mathcal{C}$,  $f_0\in \mathrm{P}_p(\mathbb{R}^d)$ (respectively $\mathrm{P}_p(\mathbb{R}^{2d})$) for some $p>r>1$. Then there exists a random variable $L$ with finite expectation uniform over $b\in\mathcal{C}$, such that for any $\tilde{b}\in\mathcal{C}$ and $t\in[0,T]$ it holds that
\begin{equation*}
|X_t^{b}-X^{\tilde{b}}_t|\le L\int^t_0\norm{b_s-\tilde{b}_s}_{\mathrm{L}^{-r,\infty}(\mathbb{R}^d;\mathbb{R}^d)}\,ds
\end{equation*}
in the first order case, and 
\begin{equation*}
|X_t^{b}-X^{\tilde{b}}_t|+|V_t^{b}-V^{\tilde{b}}_t|\le  L\int^t_0\norm{b_s-\tilde{b}_s}_{\mathrm{L}^{-r,\infty}(\mathbb{R}^d;\mathbb{R}^d)}\,ds
\end{equation*}
respectively in the second order case (with a different $L$).
\end{corollary}
\begin{proof}
We give the proof in the first order case. The second order case is analogous and no harder. Let $\phi,\tilde{\phi}$ be the associated stochastic flows given by \cref{thm:existence-of-stochastic-flow} and let
\begin{equation*}
J=\sup_{0\le s\le t\le T}\sup_{|x|\le\sup_{t\in[0,T]}|X^{\tilde{b}}_t|}|\nabla\phi_{s,t}(x)|.
\end{equation*}

Define the function $\psi(u)=(\phi_{t,u}\circ\tilde{\phi}_{0,u})(X_0)$. Then $\psi(0)=X^b_t$ and $\psi(t)=X^{\tilde{b}}_t$. We wish to estimate $|\psi(t)-\psi(0)|\le \int^t_0\left|\frac{d\psi}{ds}\right|\,ds$, but it is not immediately clear how to evaluate the derivative due to the presence of the non-differentiable Brownian motions. Instead we prove the moral equivalent using Riemann sums. Let $0=t_0<t_1<\dotsb<t_n=t$ be a partition of $[0,t]$ of maximum width $h$. Now consider
\begin{equation}\label{eq:long-equation-for-J}
\begin{aligned}
|\psi(t_{k+1})-\psi(t_k)|&=|(\phi_{t_{k+1},t}\circ\tilde{\phi}_{0,t_{k+1}})(X_0)-(\phi_{t_{k},t}\circ\tilde{\phi}_{0,t_k})(X_0)|\\
&= |(\phi_{t_{k+1},t}\circ\tilde{\phi}_{0,t_{k+1}})(X_0)-(\phi_{t_{k+1},t}\circ \phi_{t_k,t_{k+1}}\circ \tilde{\phi}_{0,t_k})(X_0)|\\
&\le J|\tilde{\phi}_{0,t_{k+1}}(X_0)-(\phi_{t_k,t_{k+1}}\circ\tilde{\phi}_{0,t_k})(X_0)|\\
&\le J|(\tilde{\phi}_{t_k,t_{k+1}}\circ\tilde{\phi}_{0,t_{k}})(X_0)-(\phi_{t_k,t_{k+1}}\circ\tilde{\phi}_{0,t_k})(X_0))|\\
&\le J\left|\int^{t_{k+1}}_{t_k}\tilde{b}_s(\tilde{\phi}_{t_k,s}(X^{\tilde{b}}_{t_k}))-b_s(\phi_{t_k,s}(X^{\tilde{b}}_{t_k}))\,ds\right|\\
&\le J\int^{t_{k+1}}_{t_k}|\tilde{b}_s(X^{\tilde{b}}_{t_k})-b_s(X^{\tilde{b}}_{t_k})|\,ds\\
&\quad+J\int^{t_{k+1}}_{t_k}|b_s(\tilde{\phi}_{t_k,s}(X^{\tilde{b}}_{t_k}))-b_s(\phi_{t_k,s}(X^{\tilde{b}}_{t_k}))|\,ds,
\end{aligned}
\end{equation}
Note that we have the simple estimate
\begin{equation*}
\sup_x|\tilde{\phi}_{u,s}(x)-\phi_{u,s}(x)|\le \int^s_u\norm{b_\tau}_{\mathrm{L}^\infty(\mathbb{R}^d;\mathbb{R}^d)}+\norm{\tilde{b}_\tau}_{\mathrm{L}^\infty(\mathbb{R}^d;\mathbb{R}^d)}\,d\tau\le C|s-u|.
\end{equation*}
Therefore, as $b_s$ is $\alpha$-H\"older continuous, the final integral in \eqref{eq:long-equation-for-J} is bounded by $CJ|t_{k+1}-t_k|^{1+\alpha}$ for some constant $C$, and when we take the partition width $h$ to zero in the sum in the display below, this term contributes nothing. Hence,
\begin{align*}
|X^b_t-X^{\tilde{b}}_t|&=|\psi(t)-\psi(0)|\\
&\le \lim_{h\to0}\sum_{k=0}^{n-1}|\psi(t_{k+1})-\psi(t_k)|\\
&=J\int^t_0|b_s(X^{\tilde{b}}_s)-\tilde{b}_s(X^{\tilde{b}}_s)|\,ds.
\end{align*}
Next we note that
\begin{equation*}
J\int^t_0|b_s(X^{\tilde{b}}_s)-\tilde{b}_s(X^{\tilde{b}}_s)|\,ds\le J\sup_{0\le s\le T}\<{X^{\tilde{b},i,N}_s}^{r}\int^t_0\norm{b_s-\tilde{b}_s}_{\mathrm{L}^{-r,\infty}(\mathbb{R}^d;\mathbb{R}^d)}\,ds.
\end{equation*}
Now define 
\begin{equation*}
\begin{aligned}
L:&=J\sup_{0\le s\le T}\<{X^{\tilde{b},i,N}_s}^{r}\\
&=\left(\sup_{0\le s\le T}\<{X^{\tilde{b},i,N}_s}^{r}\right)\left(\sup_{0\le s\le t\le T}\sup_{|x|\le\sup_{t\in[0,T]}|X^{\tilde{b}}_t|}|\nabla\phi_{s,t}(x)|\right).
\end{aligned}
\end{equation*}
Note that for any $\varepsilon>0$,
\begin{equation*}
L\le \left(\sup_{0\le t\le T}\<{X^{\tilde{b}}_t}^{r+\varepsilon}\right) \sup_{0\le s\le t\le T}\norm{\nabla\phi_{s,t}}_{\mathrm{L}^{-\varepsilon,\infty}}
\end{equation*}
so that by H\"older's inequality 
\begin{equation*}
\mathbb{E}L\le \nnorm{\sup_{0\le t\le T}\<{X^{\tilde{b}}_t}}_p^\varepsilon\nnorm{\sup_{0\le s\le t\le T}\norm{\nabla\phi_{s,t}}_{\mathrm{L}^{-\varepsilon,\infty}}}_q
\end{equation*} 
with $(1/p')+(1/q)=1$ with $p'(r+\varepsilon)=p$. (We ensure $p'>1$ by taking $\varepsilon$ sufficiently small and using $r<p$.) These are both finite by \cref{lem:SDE-bounds} and \cref{thm:existence-of-stochastic-flow} respectively. The proof is complete.
\end{proof}

\subsection{The empirical process theory argument}
To control the supremum in \eqref{eq:supremum} we will use the following key proposition. 

Recall that a semi-metric space is a metric space without the triangle inequality. The definition of metric entropy extends without modification to semi-metric spaces. For a random variable $X$ valued in a Banach space $\mathrm{V}$ we say that $X$ is centred if $\mathbb{E}g(X)=0$ for all $g$ in the dual of $\mathrm{V}$.
\begin{proposition}\label{prop:main-entropy-argument}
Let $(\mathcal{X},d)$ be a totally bounded semi-metric space and $(\mathrm{V},\norm{\cdot}_{\mathrm{V}})$ be a separable Banach space. Let $\varphi:\mathcal{X}\to \mathrm{V}$ be a centred random map, $Y$ a non-negative sub-Gaussian random variable and $(G(x,\varepsilon))_{x\in \mathcal{X},\varepsilon\in(0,1]}$ be a family of random variables. Let $(\varphi^{1,N},Y^{i,N},G^{i,N}),\dotsc,(\varphi^{N,N},Y^{N,N},G^{N,N})$ be i.i.d. copies of $(\varphi,Y,G)$ and assume the following:
\begin{enumerate}[font=\bf,label=(\roman*)]
\item \textbf{Metric entropy bounds:} The semi-metric space $\mathcal{X}$ obeys the following bound
\begin{equation*}
H(\varepsilon,\mathcal{X},d)\le C_h\varepsilon^{-k}
\end{equation*}
for constants $C_h,h>0$.
\item \textbf{`Pointwise Lipschitz' condition:} For every $x\in \mathcal{X},\varepsilon\in(0,1]$ we have
\begin{equation}\label{eq:pointwise-Holder-condition}
\sup_{d(x,\tilde{x})\le\varepsilon}\norm{\varphi(x)-\varphi(\tilde{x})}_{\mathrm{V}}\le G(x,\varepsilon)
\end{equation}
and there is a constant $C_G$ such that for all $\varepsilon\in(0,1]$,
\begin{equation}\label{eq:pointwise-Holder-condition-finite-expectation}
\sup_{x\in \mathcal{X}}\mathbb{E}G(x,\varepsilon)\le C_G\varepsilon.
\end{equation}
\item \textbf{Dominating sub-Gaussians (envelope function):} We have
\begin{equation*}
\sup_{x\in \mathcal{X}}\norm{\varphi(x)}_{\mathrm{V}}\le Y
\end{equation*}
with $\nnorm{Y}\le C_Y$, and if $V\ne\mathbb{R}$ then also $Y\le C_Y$ almost surely.
\item \textbf{Pointwise law of large numbers: }We have
\begin{equation*}
\sup_{x\in \mathcal{X}}\mathbb{E}\norm{\averN \varphi^{i,N}(x)}_{\mathrm{V}}\le C_{\mathrm{V}}N^{-1/2}.
\end{equation*}
\end{enumerate}
Then it holds that
\begin{equation*}
\nnorm{\sup_{x\in \mathcal{X}}\norm{\averN \varphi^{i,N}(x)}_{\mathrm{V}}}\le C(C_G+(C_Y+C_{\mathrm{V}})\sqrt{C_h})N^{-\gamma},\quad \gamma=\frac{1}{2+k}.
\end{equation*}
\end{proposition}
Note that in the above proposition the usual case is that $V=\mathbb{R}$. In which case assumption (iv) follows from assumption (iii) and the usual law of large numbers under second moment conditions.

In order to prove this proposition and also for later proofs, we will need a couple of standard results on the sub-Gaussian norm.
\begin{lemma}[Law of large numbers]\label{lem:sub-Gaussian-LLN}
Let $X_1,\dotsc, X_N$ be i.i.d. centred sub-Gaussian random variables. Then
\begin{equation*}
\nnorm{\averN X_i}\le C\nnorm{X_1}N^{-1/2}
\end{equation*}
for an absolute constant $C$.
\end{lemma}
The proof is a simple corollary of \cite[Lemma 5.9.]{vershynin2010introduction}.
\begin{lemma}[Orlicz maximal inequality]\label{lem:Orlicz-maximal-inequality}
Let $X_1,\dotsc, X_m$ be real sub-Gaussian random variables, not necessarily independent. Then
\begin{equation*}
\nnorm{\max_{i=1}^m|X_i|}\le C\max_{i=1}^m\nnorm{X_i}\sqrt{\log (1+m)}
\end{equation*}
for an absolute constant $C$.
\end{lemma}
We refer the reader to \cite[\S 2.2.]{Van-der-Vaart-Wellner} for details of the proof.

We will also need a variant of Talagrand's inequality for empirical processes \cite{talagrand1996new}.
\begin{theorem}[Talagrand's inequality for Banach spaces]\label{thm:talagrand}
Let $(\mathrm{V},\norm{\cdot}_{\mathrm{V}})$ be a separable Banach space and $X_1,\dotsc,X_N$ be i.i.d. centred $\mathrm{V}$-valued random variables with $\norm{X_1}_{\mathrm{V}}\le 1$ almost surely. Then,
\begin{equation*}
\nnorm{\norm{\averN X_i}_{\mathrm{V}}-\mathbb{E}\norm{\averN X_i}_{\mathrm{V}}}\le CN^{-1/2}
\end{equation*}
For an absolute constant $C$.
\end{theorem} 

\begin{proof}[Proof of \cref{prop:main-entropy-argument}]
Let $\varepsilon>0$ to be chosen and let $(x^m)_{m=1}^M$ be an $\varepsilon$-net of $\mathcal{X}$. By assumption (i) $M$ may be taken to be at most $\exp(C_h\varepsilon^{-k})$. Let $m\in\{1,\dotsc,M\}$ be arbitrary, and $x\in \mathcal{X}$ be in the $\varepsilon$-ball centred at $x^m$. Then we have the bound
\begin{align}\label{eq:main-entropy-first-eq}
\norm{\averN \varphi^{i,N}(x)}_{\mathrm{V}}&\le\norm{\averN(\varphi^{i,N}(x)-\varphi^{i,N}(x^m))}_{\mathrm{V}}+\norm{\averN\varphi^{i,N}(x^m)}_{\mathrm{V}}.
\end{align}
Consider the summands in the first term on the right hand side. By assumptions (ii) and (iii), we have
\begin{equation}\label{eq:key-prop-first-bound}
\begin{aligned}
\norm{\varphi^{i,N}(x)-\varphi^{i,N}(x^m)}_{\mathrm{V}}&\le \min(2Y^{i,N},G^{i,N}(x^m,\varepsilon))\\
&\qquad=:\mathbb{E}\min(2Y^{i,N},G^{i,N}(x^m,\varepsilon))+A^{i,N,m}
\end{aligned}
\end{equation}
where $(A^{i,N,m})_{i=1}^N$ (defined by the last equality) are i.i.d. uniformly sub-Gaussian centred random variables. To control the last term in \eqref{eq:main-entropy-first-eq} we split into two cases. Firstly, if $\mathrm{V}\ne \mathbb{R}$ then we write the last term in \eqref{eq:main-entropy-first-eq} as
\begin{align*}
\norm{\averN\varphi^{i,N}(x^m)}_{\mathrm{V}}&=\mathbb{E}\norm{\averN\varphi^{i,N}(x^m)}_{\mathrm{V}}\\
&\quad+\left( \norm{\averN\varphi^{i,N}(x^m)}_{\mathrm{V}}-\mathbb{E}\norm{\averN\varphi^{i,N}(x^m)}_{\mathrm{V}}\right).
\end{align*}
Hence, by assumption (iv) and Talagrand's inequality (\cref{thm:talagrand}), we have the bound
\begin{equation}\label{eq:key-prop-second-bound1}
\nnorm{\norm{\averN\varphi^{i,N}(x^m)}_{\mathrm{V}}}\le C(C_{\mathrm{V}}+C_Y)N^{-1/2}.
\end{equation}
Secondly, if $\mathrm{V}=\mathbb{R}$ then we can directly apply the law of large numbers for sub-Gaussian random variables (\cref{lem:sub-Gaussian-LLN}) to obtain that
\begin{equation}\label{eq:key-prop-second-bound2}
\nnorm{\norm{\averN\varphi^{i,N}(x^m)}_{\mathrm{V}}}\le CC_YN^{-1/2}.
\end{equation} 
Hence, by assumption (ii), \eqref{eq:key-prop-first-bound} and whichever of \eqref{eq:key-prop-second-bound1} or \eqref{eq:key-prop-second-bound2} applies we have
\begin{align*}
\sup_{x\in \mathcal{X},d(x,x^m)\le\varepsilon}\norm{\averN \varphi^{i,N}(x)}_{\mathrm{V}}&\le \mathbb{E}\min(2Y^{i,N},G^{i,N}(x^m,\varepsilon))\\
&\quad+\underbrace{\left|\averN A^{i,N,m}\right|+\norm{\averN \varphi^{i,N}(x^m)}_{\mathrm{V}}}_{=:B^m}\\
&\le C_G\varepsilon+B^m.
\end{align*}
By the law of large numbers for sub-Gaussian random variables (\cref{lem:sub-Gaussian-LLN}), the uniform sub-Gaussian bounds on $A^{i,N,m}$ and the above bounds on the average of $\varphi^{i,N}(x^m)$, we have
\begin{equation*}
\nnorm{B^m}\le C\nnorm{A^{1,N,m}}N^{-1/2}+C(C_{\mathrm{V}}+C_Y)N^{-1/2}\le 4C(C_Y+C_{\mathrm{V}})N^{-1/2}.
\end{equation*} 
Putting the estimates over the $\varepsilon$-net together, and using the Orlicz maximal inequality (\cref{lem:Orlicz-maximal-inequality}), we obtain
\begin{align*}
\nnorm{\sup_{x\in \mathcal{X}}\norm{\averN\varphi^{i,N}(x)}_{\mathrm{V}}}&\le\nnorm{\max_{m=1,\dotsc,M}\sup_{x\in \mathcal{X},d(x,x^m)\le\varepsilon}\norm{\averN\varphi^{i,N}(x)}_{\mathrm{V}}}\\
&\le \nnorm{\max_{m=1,\dotsc,M}C_G\varepsilon+B^m}\\
&\le C_G\varepsilon+\nnorm{\max_{m=1,\dotsc,M}B^m}\\
&\le C_G\varepsilon+C\sqrt{\log(1+M)}\max_{i=1,\dotsc,M}\nnorm{B^m}\\
&\le C_G\varepsilon+C(C_Y+C_{\mathrm{V}})\sqrt{C_h}\varepsilon^{-k/2}N^{-1/2}.
\end{align*}
By choosing $\varepsilon=N^{-1/(2+k)}$ we obtain the claimed result.
\end{proof}
In addition to \cref{thm:GV-intro,thm:2nd-order-GV-intro} we will also prove a proposition that will be used in the proofs of \cref{thm:propagation-of-chaos,thm:2nd-order-propagation-of-chaos}.
\begin{proposition}\label{prop:ULLN-for-h}
Let $f^b,\mu^{b,N}$ be as in \eqref{eq:empirical-limit-pde},\eqref{eq:empirical-process} in the first order case, and respectively \eqref{eq:2nd-order-empirical-limit-pde},\eqref{eq:2nd-order-empirical-process} in the second order case. Let $f_0\in \mathrm{P}_p(\mathbb{R}^d)$ (respectively $f_0\in \mathrm{P}_p(\mathbb{R}^{2d})$) for some $p>1$. Let $\mathcal{C}$ be a bounded subset of $\mathcal{C}^\alpha$ for some $\alpha\in(0,1)$ (respectively $\alpha\in(2/3,1)$) which satisfies 
\begin{equation}\label{eq:metric-entropy-assumption-ULLN}
H(\varepsilon,\mathcal{C},\norm{\cdot}_{\mathrm{L}^\infty([0,T];\mathrm{L}^{-r,\infty}(\mathbb{R}^d;\mathbb{R}^d))})\le C\varepsilon^{-k}
\end{equation}
for some $r\in(1,p)$. Let $h:\mathbb{R}^d\times\mathbb{R}^d\to\mathbb{R}$ be a bounded function satisfying
\begin{equation*}
\sup_{y,\delta\in\mathbb{R}^d,\delta\ne0}\frac{\norm{h(\cdot,y+\delta)-h(\cdot,y)}_{\mathrm{L}^{-r',q}(\mathbb{R}^{d})}}{|\delta|^\beta}<\infty ,
\end{equation*}
where $\beta\in(0,1]$, $q\in[1,\infty]$ and $r'>d/q$. Then we have
\begin{equation*}
\nnorm{\sup_{b\in\mathcal{C}}\sup_{t\in[0,T]}\norm{\averN h(\cdot,X^{b,i,N}_t)-\mathbb{E}h(\cdot,X^{b,i,N}_t)}_{\mathrm{L}^{-r',q}(\mathbb{R}^d)}}\le CN^{-\gamma}, \quad \gamma=\frac{1}{2+\frac k\beta}.
\end{equation*}
\end{proposition}
\begin{proof}[Proof of \cref{thm:GV-intro,thm:2nd-order-GV-intro,prop:ULLN-for-h}]
We begin with the proofs of \cref{thm:GV-intro,thm:2nd-order-GV-intro}. We first note that by \cref{lem:mkw-triangle-with-references,lem:reference-process-mkw-bound} we have 
\begin{equation}\label{eq:mkw-inequality-in-LLN}
\begin{aligned}
&\left[\sup_{t\in[0,T],b\in\mathcal{C}}d_{\mathrm{MKW}}(\mu^{b,N}_t,f^b_t)-cd_{\mathrm{MKW}}(\mu^N_0,f_0)\right]_+\le \\
&\quad \sup_{h\in\mathrm{Lip1},t\in[0,T],b\in\mathcal{C}}\left(\averN 
(h(X^{b,i,N}_t)-h(\widetilde{X}^{i,N}_t))-\mathbb{E}(h(X^{b,i,N}_t)-h(\widetilde{X}^{i,N}_t))\right),
\end{aligned}
\end{equation}
in the first order case (with $c=1$) and the corresponding inequality in the second order case (with $c>1$).

From here on we give the proof for the first order system (\cref{thm:GV-intro}). The proof for the second order system (\cref{thm:2nd-order-GV-intro}) is analogous (using instead the second order versions of the above lemmas) and we leave it to the reader. The proof follows from the application of \cref{prop:main-entropy-argument} with a carefully chosen map $\varphi$ and semi-metric space $(\mathcal{X},d)$.

We set $(\mathcal{X},d)$ to be
\begin{equation*}
([0,T],|\cdot|^{1/3})\times (\mathcal{C},\norm{\cdot}_{\mathrm{L}^{\infty}([0,T];\mathrm{L}^{-r,\infty}(\mathbb{R}^d;\mathbb{R}^d))})\times(\mathrm{Lip1},\norm{\cdot}_{\mathrm{L}^{-p,\infty}(\mathbb{R}^d;\mathbb{R}^d)}) 
\end{equation*} with the product metric, where $|\cdot|$ is the standard Euclidean norm. By assumption the metric entropy of the second space in the above display is bounded by $C\varepsilon^{-k}$. By the results in \cite{Edmunds-Triebel,Nickl} the metric entropy of the third space in the above display is bounded by $C\varepsilon^{-\min(d,d/(p-1))}$ (see \cref{lem:metric-entropy-of-Lip1} for details.) As the metric entropy of $([0,T],|\cdot|^{1/3})$ is logarithmic in $\varepsilon$, the metric entropy of $(X,d)$ is controlled, using \cref{lem:metric-entropy-of-product-spaces} by
\begin{equation}\label{eq:metric-entropy-estimate-i}
H(\varepsilon,X,d)\le C\varepsilon^{-d/(p-1)}+C\varepsilon^{-k}+C\log(1/\varepsilon)\le C \varepsilon^{-\max(d,d/(p-1),k)}.
\end{equation}
We define $\varphi^{i,N}:X\to\mathbb{R}$ for $i=1,\dotsc,N$, by
\begin{equation}\label{eq:varphi-i}
\varphi^{i,N}(t,b,h)=(h(X^{b,i,N}_{t})-h(\widetilde{X}^{i,N}_t))-\mathbb{E}(h(X^{b,i,N}_{t})-h(\widetilde{X}^{i,N}_t)).
\end{equation}
With this choice of $\varphi$ and $X$ the supremum on the right hand side of \eqref{eq:mkw-inequality-in-LLN} will be equal to 
\begin{equation*}
\sup_{x\in X}\left|\averN\varphi^{i,N}(x)\right|.
\end{equation*}

Using that $h$ is $1$-Lipschitz, it follows from \cref{lem:SDE-subgaussian-bounds} that $\sup_{x\in X}|\varphi^{i,N}(x)|$ is bounded by a sub-Gaussian random variable $Y^{i,N}$. Thus assumptions (i) and (iii) of \cref{prop:main-entropy-argument} are satisfied, and as the target space $\mathrm{V}$ is simply $\mathbb{R}$, assumption (iv) is also satisfied. This just leaves the verification of assumption (ii). We will compute this for each variable $(t,b,h)$ in turn. For brevity we compute this for $\varphi^{i,N}(t,b,h)=h(X^{b,i,N}_{t})-h(\widetilde{X}^{i,N}_t)$, the estimate for the centred version \eqref{eq:varphi-i} follows easily.

Let $L^{i,N}$ be as in \cref{cor:lipschitz-dependence-of-SDEs-upon-b} applied respectively to $X^{b,i,N}$, so that $L^{i,N}$ are i.i.d. with finite expectation. Then as $h\in \mathrm{Lip1}$ we have
\begin{equation*}
\begin{aligned}
|\varphi^{i,N}(t,b,h)-\varphi^{i,N}(t,\tilde{b},h)|&\le L^{i,N}T\norm{b-\tilde{b}}_{\mathrm{L}^{\infty}([0,T];\mathrm{L}^{-r,\infty}(\mathbb{R}^d;\mathbb{R}^d))}.
\end{aligned}
\end{equation*}
Next we consider
\begin{equation*}
\begin{aligned}
|\varphi^{i,N}(t,b,h)-\varphi^{i,N}(t,b,\tilde{h})|&= |(h(X^{b,i,N}_{t})-h(\widetilde{X}^{i,N}_t))-(\tilde{h}(X^{b,i,N}_{t})-\tilde{h}(\widetilde{X}^{i,N}_t))|\\
&\le \left(\sup_{s\in[0,T]}\<{X^{b,i,N}_s}^p+\<{\widetilde{X}^{i,N}_s}^p\right)\norm{h-\tilde{h}}_{\mathrm{L}^{-p,\infty}(\mathbb{R}^d)}
\end{aligned} 
\end{equation*}
and we can control the expectation of this supremum by \cref{lem:SDE-bounds} for $X^{b,i,N}$ (the estimate for $\tilde{X}^{i,N}$ being easier) and the $p$th moment of $X_0$.

Lastly we control the dependence on time. Let $t\in[0,T]$, then we have
\begin{equation*}
\begin{aligned}
&\sup_{s\in[0,T],|t-s|^{1/3}\le\varepsilon}|\varphi^{i,N}(t,b,h)-\varphi^{i,N}(s,b,h)|\le\\
&\quad \sup_{s\in[0,T],|t-s|^{1/3}\le\varepsilon} \norm{h}_{\mathrm{Lip}}|(X^{b,i,N}_{t}-\widetilde{X}^{i,N}_t)-(X^{b,i,N}_{s}-\widetilde{X}^{i,N}_s)|
\end{aligned} 
\end{equation*}
and the expectation of the supremum on the right hand side is controlled by \cref{lem:SDE-subgaussian-bounds}.

This completes the proof of \cref{thm:GV-intro}.

We now prove \cref{prop:ULLN-for-h}. The proof again relies on \cref{prop:main-entropy-argument} and a carefully chosen map $\varphi$ and semi-metric space $(\mathcal{X},d)$. We split the proof into two cases. Firstly we handle $q\in[1,\infty)$. Let $\mathrm{V}=\mathrm{L}^{-r',q}(\mathbb{R}^d)$ and 
\begin{equation*}
\mathcal{X}=([0,T]\times|\cdot|^{\alpha\beta/3})\times (\mathcal{C},\norm{\cdot}_{\mathrm{L}^{\infty}([0,T];\mathrm{L}^{-r,\infty}(\mathbb{R}^d;\mathbb{R}^d))}^\beta)
\end{equation*}
 with the product metric where $r\in(1,p)$. As in the above proof, we estimate the metric entropy of $X$,
\begin{equation*}
H(\varepsilon,\mathcal{X},d)\le C\varepsilon^{-k/\beta}
\end{equation*}
where the dominant term comes from the second space in the definition of $X$ (using \cref{lem:change-of-metric}). We define $\varphi^{i,N}:\mathcal{X}\to \mathrm{V}$ by
\begin{equation*}
(\varphi^{i,N}(t,b))(x)=h(x,X^{b,i,N}_t)-\mathbb{E}h(x,X^{b,i,N}_t).
\end{equation*}
This is uniformly bounded in $\mathrm{V}=\mathrm{L}^{-r',q}(\mathbb{R}^d)$ by $2\sup_{y\in\mathbb{R}^d}\norm{h(\cdot,y)}_{\mathrm{L}^{-r',q}(\mathbb{R}^d)}$ which is finite as $h$ is uniformly bounded and $\<x^{-r'q}$ is integrable by assumption. To verify the `Pointwise Lipschitz' condition we argue in the same way as in the proof of \cref{thm:GV-intro}. We have, again working with the uncentred version for brevity,
\begin{equation*}
\norm{\varphi^{i,N}(t,b)-\varphi^{i,N}(s,\tilde{b})}_{\mathrm{V}}\le |X^{i,N,b}_t-X^{i,N,\tilde{b}}_s|^\beta \sup_{\delta,y\in\mathbb{R}^d,\delta\ne0}\frac{\norm{h(\cdot,y+\delta)-h(\cdot,y)}_{\mathrm{V}}}{|\delta|^\beta}
\end{equation*}
This supremum is finite by assumption, and the expectation of $|X^{i,N,b}_t-X^{i,N,\tilde{b}}|^\beta$ can be controlled as in the proof of \cref{thm:GV-intro} using also the $\beta$ appearing in the metric of $\mathcal{X}$.

It remains to check the pointwise law of large numbers. We compute
\begin{equation*}
\begin{aligned}
&\sup_{(t,b)\in X}\mathbb{E}\norm{\averN\varphi^{i,N}(t,b)}_{\mathrm{V}}\\
&=\sup_{(t,b)\in X}\mathbb{E}\left(\int_{\mathbb{R}^d}\left|\averN h(x,X^{i,b,N}_t)-\mathbb{E}h(x,X^{i,b,N}_t)\right|^q\<x^{-qr'}\,dx\right)^{1/q}\\
&\le C\sup_{(t,b)\in X}\left(\int_{\mathbb{R}^d}\mathbb{E}\left|\averN h(x,X^{i,b,N}_t)-\mathbb{E}h(x,X^{i,b,N}_t)\right|^q\<x^{-qr'}\,dx\right)^{1/q}\\
&\le C\sup_{(t,b)\in X}\left(\int_{\mathbb{R}^d}\left(\sup_{x'\in\mathbb{R}^d}\mathbb{E}\left|\averN h(x',X^{i,b,N}_t)-\mathbb{E}h(x',X^{i,b,N}_t)\right|^q\right)\<x^{-qr'}\,dx\right)^{1/q}\\
&\le C\sup_{x\in\mathbb{R}^d,(t,b)\in X}\nnorm{\averN h(x,X^{i,b,N}_t)-\mathbb{E}h(x,X^{i,b,N}_t)}_q,
\end{aligned}
\end{equation*}
where we have used that $\<x^{-qr'}$ is integrable on $\mathbb{R}^d$ twice, first on the third line to apply Jensen's inequality and then on the fourth line after bringing the expectation out of the integral. As $h$ is a uniformly bounded function, we can apply \cref{lem:sub-Gaussian-LLN} (or just the usual law of large numbers) to obtain that
\begin{equation*}
\sup_{(t,b)\in \mathcal{X}}\mathbb{E}\norm{\averN\varphi^{i,N}(t,b)}_{\mathrm{V}}\le CN^{-1/2}
\end{equation*}
as required. This completes the proof of \cref{prop:ULLN-for-h} for $p<\infty$.

Now suppose that $q=\infty$. In this case we set $\mathrm{V}=\mathbb{R}$ and 
\begin{equation*}
\mathcal{X}=([0,T]\times|\cdot|^{\alpha\beta/3})\times (\mathcal{C},\norm{\cdot}_{\mathrm{L}^{\infty}([0,T];\mathrm{L}^{-r,\infty}(\mathbb{R}^d;\mathbb{R}^d))}^\beta)\times (\mathbb{R}^d,\rho_{r',\beta})
\end{equation*}
with the product metric, where $r\in(1,p)$. Here $\rho_{r',\beta}$ is a semi-metric on $\mathbb{R}^d$ defined by
\begin{equation*}
\rho_{r',\beta}(x,y)=\frac{\min(|x-y|^\beta,1)}{(1+\min(|x|,|y|))^{\min(r',1)}}.
\end{equation*}
It is easy to check that the metric entropy of $(\mathbb{R}^d,\rho_{r',\beta})$ is logarithmic for any $r'>0$ and $\beta\in(0,1]$. Hence, as in the $p<\infty$ case, the metric entropy of $\mathcal{X}$ has the bound
\begin{equation*}
H(\varepsilon,\mathcal{X},d)\le C_\gamma\varepsilon^{-k/\beta}.
\end{equation*}
We define $\varphi^{i,N}:\mathcal{X}\to\mathbb{R}$ by
\begin{equation*}
\varphi^{i,N}(t,b,x)=h(x,X^{b,i,N}_{t})-\mathbb{E}h(x,X^{b,i,N}_{t}).
\end{equation*}
This is uniformly bounded by $2$. The `Pointwise Lipschitz' conditions for $[0,T]$ and $\mathcal{C}$ are the same as in the proof of \cref{thm:GV-intro}, except that $h$ is only $\beta$-H\"older continuous instead of Lipschitz, which is taken into account in the choice of metrics on $[0,T]$ and $\mathcal{C}$. This leaves the estimate for $(\mathbb{R}^d,\rho_{r',\beta})$. Again consider the uncentred case to ease notation. Let $x,y\in\mathbb{R}^d$ with $|x-y|\le1$, then we have, where without loss of generality $|x|\ge |y|$,
\begin{equation*}
\begin{aligned}
|\varphi^{i,N}(t,b,x)-\varphi^{i,N}(t,b,y)|&=|h(x,X^{b,i,N}_{t})-h(y,X^{b,i,N}_{t})|\\
&\le |\<x^{r'}h(x,X^{b,i,N}_{t})-\<y^{r'}h(y,X^{b,i,N}_{t})|\<x^{-{r'}}\\
&\quad+\<y^{r'}|h(y,X^{b,i,N}_{t})||\<x^{-{r'}}-\<y^{-{r'}}|\\
&\le C|x-y|^\beta\<x^{-{r'}}+C\<y^{r'}\<y^{-{r'}-1}|x-y|\\
&\le C\rho_{{r'},\beta}(x,y).
\end{aligned}
\end{equation*}
This completes the proof of \cref{prop:ULLN-for-h}.
\end{proof}

\section{Propagation of chaos}\label{sec:propagation-of-chaos}
In this section we prove the propagation of chaos results \cref{thm:propagation-of-chaos,thm:2nd-order-propagation-of-chaos}. This section is organised as follows. In \cref{subsec:first-order-prop-of-chaos} we prove \cref{thm:propagation-of-chaos} after stating a pair of preliminary lemmas without proof. In \cref{subsec:second-order-prop-of-chaos} we present the proof of \cref{thm:2nd-order-propagation-of-chaos} again after stating without proof a pair of lemmas. Note that the proof of \cref{thm:2nd-order-propagation-of-chaos} is very similar to that of \cref{thm:propagation-of-chaos} so we give only the differences. Finally in \cref{subsec:proof-of-the-time-regularity-lemmas,subsec:proof-of-the-energy-estimates} we provide the postponed proofs of the lemmas.

\subsection{The first order case}\label{subsec:first-order-prop-of-chaos}
As a first step, we must obtain a prior estimates on the time regularity of the vector field $b^N$ and show that the contribution of $b^N$ not being sufficiently regular to \eqref{eq:propagation-of-chaos} is of lower order. As the proofs are technical we present them at the end of this section.

In the first order case we expect that $b^N\in \Lambda^{0,\alpha'}_{para}(\mathrm{L}^{q}([0,T]\times\mathbb{R}^d;\mathbb{R}^d))$ for $\alpha'<\alpha$ as $X^{i,N}$ will be merely (almost $1/2$)-H\"older continuous in time due to the driving noise.
\begin{lemma}\label{lem:a-priori-time-regularity}[Time regularity (first order case)]
Let $K(x,y)\in \Lambda^{0,\alpha}(\mathrm{L}^\infty_y(\mathbb{R}^d;\mathrm{L}^{q}_x(\mathbb{R}^d)))$ with $\alpha\in(0,1]$ and $f_0\in \mathrm{P}_1(\mathbb{R}^d)$. Define the event $E_A$ by
\begin{equation*}
E_A=\{\norm{b^N}_{\Lambda^{0,\alpha'}_{para}(\mathrm{L}^{q}([0,T]\times\mathbb{R}^d))}>A\}
\end{equation*}
for any $\alpha'\in(0,\alpha)$. Then there exists $A>0$ such that we have the bound
\begin{equation*}
\nnorm{1_{E_A}\left(\sup_{t\in[0,T]}d_{\mathrm{MKW}}(\mu_t^N,f_t)-d_{\mathrm{MKW}}(\mu^N_0,f_0)\right)}\le CN^{-1/2}.
\end{equation*}
where $C$ and $A$ depend only on $\alpha'$ and the norm of $K$.
\end{lemma}

Next we bound the dependence of the laws $f^b_t$ using simple energy estimates. As we will work throughout the proof with mollified kernels, we prove the results for smooth vector fields (with constants independent of the degree of smoothness) which avoids any issues with existence or uniqueness. Again, we delay the proof until the end of this section. In the first order case we have:
\begin{lemma}[Weighted energy estimate (first order case)]\label{lem:weighted-energy-estimate}
Let $b,\tilde{b}\in \mathrm{L}^\infty([0,T]\times\mathbb{R}^d;\mathbb{R}^d)$ be continuous in $t$ and $C^1_b$ in $x$, and $f_0\in \mathrm{L}^{p+r,q}(\mathbb{R}^{d})$ for some $r,p>0$ and $q\in[2,\infty)$, then
\begin{equation*}
\norm{f^b_t-f^{\tilde{b}}_t}_{\mathrm{L}^{p,2}(\mathbb{R}^d)}\le C\int^t_0\norm{b_s-\tilde{b}_s}_{\mathrm{L}^{-r,q'}(\mathbb{R}^d;\mathbb{R}^d)}\,dt,\qquad \frac1{q'}+\frac1q=\frac12,
\end{equation*}
where $C$ depends only on $f_0$, $\norm{b}_{\mathrm{L}^\infty([0,T]\times\mathbb{R}^d)}$ and $\norm{\tilde{b}}_{\mathrm{L}^\infty([0,T]\times\mathbb{R}^d)}$.
\end{lemma}

With these lemmas we are ready to prove the main propagation of chaos result.
\begin{proof}[Proof of \cref{thm:propagation-of-chaos}]
We divide the proof into $5$ steps.

\textbf{Step 1. Mollification of the interaction kernel.} Note first that Sobolev embedding implies that $K(x,y)$ is in $\mathrm{C}^{0,\alpha}(\mathbb{R}^d;\mathbb{R}^d)$ for some $\alpha>0$ in all cases of the theorem. Now let $K_n$ be a sequence of smooth interaction kernels obeying the same bounds as $K$. Then using \cref{cor:lipschitz-dependence-of-SDEs-upon-b} on the entire $N\cdot d$ dimensional system we deduce that the solutions $X^{n,i,N}_t$ to the SDE system \eqref{eq:many-particle-system} with $K$ replaced by $K_n$ converge almost surely to the solution $X^{i,N}_t$ of the original system \eqref{eq:many-particle-system}. Using this, it is sufficient to prove propagation of chaos for a smooth kernel $K$ with constants depending only on the bounds assumed in the theorem. Thus from here on in the proof $K$ shall be assumed to be in $\mathrm{C}^1_b$. As a consequence all the considered vector fields $b$ will also lie in $\mathrm{C}([0,T];\mathrm{C}^1_b(\mathbb{R}^d;\mathbb{R}^d))$, and we can freely apply \cref{thm:GV-intro} to such fields.

\textbf{Step 2. Choice of functional space and exponents.}
We have assumed that $K(x,y)\in \Lambda^{0,s}(\mathrm{L}^\infty_y(\mathbb{R}^d;\mathrm{L}^q_x(\mathbb{R}^d)))$ (note that, as explained in \cref{rem:relax-assumption-on-K}, case (1) of \cref{thm:propagation-of-chaos} is included in case (2) as $q=\infty$). By the assumption on $f_0$ we may choose $r,r'$ such that 
\begin{equation*}
f_0\in \mathrm{L}^{r+r',q'}(\mathbb{R}^d), \quad r>d/q, \quad r'>(d/2)+1.
\end{equation*}
Note that with these choices we have the continuous inclusions:
\begin{align}\label{eq:continuous-inclusions}
\mathrm{L}^{r',2}(\mathbb{R}^d)\hookrightarrow \mathrm{L}^{1,1}(\mathbb{R}^d)\hookrightarrow (\mathrm{P}_1(\mathbb{R}^d),d_{\mathrm{MKW}}),
\end{align} 
which will be how we control the Wasserstein distance between functions (as opposed to measures).

\textbf{Step 3. Regularity of the interaction field.}
Let $s'<s$, then by \cref{lem:a-priori-time-regularity} we may assume $b^N\in\mathcal{C}$ where
\begin{equation}\label{eq:C-in-prop-chaos}
\mathcal{C}=\mathrm{C}([0,T];\mathrm{C}^1_b(\mathbb{R}^d;\mathbb{R}^d))\cap \{b:\norm{b}_{\Lambda^{0,s'}_{para}(\mathrm{L}^{q}([0,T]\times\mathbb{R}^d;\mathbb{R}^d))}\le A\},
\end{equation}
for some $A<\infty$. Note that by \cref{lem:metric-entropy-of-Lip1}(4) we have the metric entropy bound
\begin{equation}\label{eq:gamma_1}
H(\varepsilon,\mathcal{C},\mathrm{L}^\infty([0,T];\mathrm{L}^{-p',\infty}(\mathbb{R}^d;\mathbb{R}^d)))\le C\varepsilon^{-(d+2)/s'}
\end{equation}
for any $p'\in(d/q,p)$. [Recall that $p>d/q$ by assumption.] Here we have used that $q>(d+2)/s$ so that $s'$ can be chosen large enough for $q>(d+2)/s'$ to hold. Note further that, due to Sobolev embedding, all elements $b\in\mathcal{C}$ have a uniform bound in $\mathrm{C}^{0,\alpha}_{para}$, i.e. $\norm{b}_{\mathrm{C}^{0,\alpha}_{para}([0,T]\times\mathbb{R}^d;\mathbb{R}^d))}\le CA$ for $\alpha=s-(d+2)/q>0$ and an absolute constant $C$. 

\textbf{Step 4. Consistency: The uniform law of large numbers on the particles.}
Note that the particle system $(X^{i,N}_t)_{i=1}^N$ is equal to $(X^{b^N,i,N}_{t})_{i=1}^N$, and the limit process $f_t$ is equal to $f_{t}^{b^\infty}$. 

By the triangle inequality,
\begin{align*}
&\sup_{t\in[0,T]}d_{\mathrm{MKW}}(\mu^N_t,f_t)-d_{\mathrm{MKW}}(\mu^N_0,f_0)\\
&\quad\le \left(\sup_{t\in[0,T]}d_{\mathrm{MKW}}(\mu^{b^N,N}_{t},f^{b^N}_t)-d_{\mathrm{MKW}}(\mu^N_0,f_0)\right)+\sup_{t\in[0,T]}d_{\mathrm{MKW}}(f^{b^N}_t,f^{b^\infty}_t).
\end{align*}
Using \cref{thm:GV-intro} with $\mathcal{C}$ given by \eqref{eq:C-in-prop-chaos} and using \eqref{eq:gamma_1}, the first term can be bounded as
\begin{align*}
&\nnorm{\sup_{t\in[0,T]}d_{\mathrm{MKW}}(\mu^{b^N,N}_{t},f^{b^N}_t)-d_{\mathrm{MKW}}(\mu^N_0,f_0)}\\
&\quad\le \nnorm{\sup_{b\in\mathcal{C}}\sup_{t\in[0,T]}d_{\mathrm{MKW}}(\mu^{b,N}_t,f_{t}^b)-d_{\mathrm{MKW}}(\mu^N_0,f_0)}\le CN^{-\gamma_1}, 
\end{align*}
where
\begin{equation*}
\gamma_1=\frac{1}{2+\max(\tfrac{d+2}{s'},\tfrac{d}{p-1})}.
\end{equation*} 

\textbf{Step 5. Stability: Estimates on the limit equation.}  This leaves the other distance $d_{\mathrm{MKW}}(f^{b^N}_t,f^{b^\infty}_t)$, for which we will use estimates on the limit equation. 

\textbf{Step 5.1. Dependence of $f$ upon the field.}
By the energy estimate (\cref{lem:weighted-energy-estimate}) we have
\begin{equation}\label{eq:prop-of-chaos-main-gronwall-estimate-part1}
\begin{aligned}
\norm{f^{b^N}_t-f^{b^\infty}_t}_{\mathrm{L}^{1,1}(\mathbb{R}^d)}&\le C\norm{f^{b^N}_t-f^{b^\infty}_t}_{\mathrm{L}^{r',2}(\mathbb{R}^d)}\\
&\le C\int^t_0\norm{b^N_s-b^\infty_s}_{\mathrm{L}^{-r,q}(\mathbb{R}^d;\mathbb{R}^d)}\,ds
\end{aligned}
\end{equation}
where the first continuous inclusion in \eqref{eq:continuous-inclusions} is used for the first line and $f_0\in \mathrm{L}^{r+r',q'}(\mathbb{R}^d)$ is needed to apply the energy estimate for the second.

\textbf{Step 5.2. Dependence of the field upon $f$.}
For $b\in\mathcal{C}$, define $b^{b,N}_{t}$ by
\begin{equation*}
b^{b,N}_{t}(x)=\averN K(x,X^{b,i,N}_{t}),
\end{equation*} 
so that $b^N_t=b^{b^N,N}_{t}$. Next define $b^{b,\infty}_{t}$ by
\begin{equation*}
b^{b,\infty}_{t}(x)=\int K(x,y)f^b_{t}(y)\,dy,
\end{equation*}
so that $b^\infty_t=b^{b^\infty,\infty}_t$. Then we have
\begin{equation}\label{eq:prop-of-chaos-main-gronwall-estimate-part2}
\begin{aligned}
\norm{b^N_t-b^\infty_t}_{\mathrm{L}^{-r,q}(\mathbb{R}^d;\mathbb{R}^d)}&\le \norm{b^N_{t}-b^{b^N,\infty}_t}_{\mathrm{L}^{-r,q}(\mathbb{R}^d;\mathbb{R}^d)}
+\norm{b^{b^N,\infty}_t-b^{b^\infty,\infty}_t}_{\mathrm{L}^{-r,q}(\mathbb{R}^d;\mathbb{R}^d)}\\
&\le\sup_{b\in\mathcal{C}}\norm{b^{b,N}_{t}-b^{b,\infty}_{t}}_{\mathrm{L}^{-r,q}(\mathbb{R}^d;\mathbb{R}^d)}+\norm{b^{b^N,\infty}_t-b^{b^\infty,\infty}_t}_{\mathrm{L}^{-r,q}(\mathbb{R}^d;\mathbb{R}^d)}
\end{aligned}
\end{equation}
By applying \cref{prop:ULLN-for-h} to $K$ we can control the first of these by 
\begin{equation}\label{eq:prop-of-chaos-main-gronwall-estimate-part3}
\nnorm{\sup_{b\in\mathcal{C}}\sup_{t\in[0,T]}\norm{b^{b,N}_{t}-b^{b,\infty}_{t}}_{\mathrm{L}^{-r,q}(\mathbb{R}^d;\mathbb{R}^d)}}\le CN^{-\gamma_2},\qquad \gamma_2=\frac{1}{2+\tfrac{d+2}{s'^2}}.
\end{equation}
Here we have used that $b^{b,\infty}_{t}=\mathbb{E}b^{b,N}_{t}$ for $b$ deterministic, that $r>d/q$, that $K$ is bounded and that $K\in \Lambda^{0,s}(\mathrm{L}^\infty_y(\mathbb{R}^d;\mathrm{L}^q_x(\mathbb{R}^d;\mathbb{R}^d)))$ implies
\begin{equation*}
\sup_{0\ne(\delta_1,\delta_2)\in\mathbb{R}^d\times\mathbb{R}^d}\norm{\frac{K(x+\delta_1,y+\delta_2)-K(x,y)}{|\delta_1|^2+|\delta_2|^2}}_{\mathrm{L}^\infty_y(\mathbb{R}^d;\mathrm{L}^q_x(\mathbb{R}^d;\mathbb{R}^d))}<\infty
\end{equation*}
and by taking $\delta_1=0$ and using that $\mathrm{L}^q$ embeds continuously into $\mathrm{L}^{-r,q}$, we recover the assumption of \cref{prop:ULLN-for-h}.

 The second of the terms on the right of \cref{eq:prop-of-chaos-main-gronwall-estimate-part2} can be controlled by
\begin{equation}\label{eq:prop-of-chaos-main-gronwall-estimate-part4}
\begin{aligned}
\norm{b^{b^N,\infty}_t-b^{b^\infty,\infty}_t}_{\mathrm{L}^{-r,q}(\mathbb{R}^d;\mathbb{R}^d)}&\le C\sup_{x\in\mathbb{R}^d}\int |K(x,y)||f^{b^N}_t(y)-f^{b^\infty}_t(y)|\,dy\\
&\le C \norm{f^{b^N}_t-f^{b^\infty}_t}_{\mathrm{L}^1(\mathbb{R}^d)}\le C \norm{f^{b^N}_t-f^{b^\infty}_t}_{\mathrm{L}^{1,1}(\mathbb{R}^d)}
\end{aligned}
\end{equation}
where for $q=\infty$ the first inequality is clear, and for $q<\infty$ we have used that $\<x^{-rq}$ is integrable on $\mathbb{R}^d$ to obtain it.

\textbf{Step 5.3. Gr\"onwall estimate.}  Combining \eqref{eq:prop-of-chaos-main-gronwall-estimate-part4} with the previous estimates \eqref{eq:prop-of-chaos-main-gronwall-estimate-part1},\eqref{eq:prop-of-chaos-main-gronwall-estimate-part2},\eqref{eq:prop-of-chaos-main-gronwall-estimate-part3} yields
\begin{equation*}
\norm{f^{b^N}_t-f^{b^\infty}_t}_{\mathrm{L}^{1,1}(\mathbb{R}^d)}\le Y+C\int^t_0\norm{f_s^{b^N}-f_s^{b^\infty}}_{\mathrm{L}^{1,1}(\mathbb{R}^d)}\,ds
\end{equation*}
where $Y$ is a non-negative sub-Gaussian random variable with norm bound $\nnorm{Y}\le CN^{-\gamma_2}$. Therefore, applying the Gr\"onwall inequality we have
\begin{equation*}
\sup_{t\in[0,T]}\norm{f^{b^N}_t-f^{b^\infty}_t}_{\mathrm{L}^{1,1}(\mathbb{R}^d)}\le CY,
\end{equation*} 
and as the $\mathrm{L}^{1,1}$ distance controls the Wasserstein distance (this is the second continuous inclusion in \eqref{eq:continuous-inclusions}), we have proved the theorem.
\end{proof}

\subsection{The second order case}\label{subsec:second-order-prop-of-chaos}
We now move onto the second order case. We begin, as before, with estimates on the time regularity of the interaction field.

In the second order case we expect higher regularity for $b^N$ as $K$ is evaluated at the spatial positions $X^{i,N}$ which are time differentiable. However, we need additional moments to control the velocities.
\begin{lemma}\label{lem:2nd-order-a-priori-time-regularity}[Time regularity (second order case)]
Let $c$ be the constant in the claim of \cref{lem:reference-process-mkw-bound}. Then the following hold:
\begin{enumerate}
\item Let $K(x,y)\in \Lambda^{0,\alpha}(\mathrm{L}^\infty_y(\mathbb{R}^d;\mathrm{L}^{q}_x(\mathbb{R}^d)))$ for some $\alpha\in(0,1]$ and let $f_0\in \mathrm{P}_2(\mathbb{R}^d\times\mathbb{R}^d)$. Let $E_A$ be the event
\begin{equation*}
E_A=\{\norm{b^N}_{\Lambda^{0,\alpha}(\mathrm{L}^{q}([0,T]\times\mathbb{R}^d))}>A\}.
\end{equation*}
Then there exists $A>0$ such that we have the bound
\begin{equation*}
\nnorm{1_{E_A}\left[\sup_{t\in[0,T]}d_{\mathrm{MKW}}(\mu_t^N,f_t)-cd_{\mathrm{MKW}}(\mu^N_0,f_0)\right]_+}\le CN^{-1/2},
\end{equation*}
where $C$ and $A$ depend only on the norm of $K$.
\item Let $K(x,y)=W(x-y)$ for $W\in \Lambda^{1,\alpha}(\mathrm{L}^{q}(\mathbb{R}^d;\mathbb{R}^d))$ for some $\alpha\in(0,1/2)$ and let $f_0\in \mathrm{P}_4(\mathbb{R}^{2d})$. Let $E_A$ be the event
\begin{equation*}
E_A=\{\norm{b^N}_{\Lambda^{1,\alpha}(\mathrm{L}^{q}([0,T]\times\mathbb{R}^d))}>A\}.
\end{equation*}
Then there exists $A>0$ such that we have the bound
\begin{equation*}
\nnorm{1_{E_A}\left[\sup_{t\in[0,T]}d_{\mathrm{MKW}}(\mu_t^N,f_t)-cd_{\mathrm{MKW}}(\mu^N_0,f_0)\right]_+}\le CN^{-1/2},
\end{equation*}
where $C$ and $A$ depend only on the norm of $K$.
\end{enumerate}
\end{lemma} 
The second order energy estimate is:
\begin{lemma}[Weighted energy estimate (second order case)]\label{lem:2nd-order-weighted-energy-estimate}
Let $b,\tilde{b}\in \mathrm{L}^\infty([0,T]\times\mathbb{R}^d;\mathbb{R}^d)$ be continuous in $t$ and $C^1_b$ in $x$, and $f_0\in \mathrm{L}^{p+r,q}(\mathbb{R}^{d}\times\mathbb{R}^d)$ for some $r,p>0$ and $q\in[2,\infty)$, then
\begin{equation*}
\norm{f^b_t-f^{\tilde{b}}_t}_{\mathrm{L}^{p,2}(\mathbb{R}^{2d})}\le C\int^t_0\norm{b_s-\tilde{b}_s}_{\mathrm{L}^{-r,q'}(\mathbb{R}^d;\mathbb{R}^d)}\,dt,\qquad \frac1{q'}+\frac1q=\frac12,
\end{equation*}
where $C$ depends only on $f_0$,  $\norm{b}_{\mathrm{L}^\infty([0,T]\times\mathbb{R}^d;\mathbb{R}^d)}$ and $\norm{\tilde{b}}_{\mathrm{L}^\infty([0,T]\times\mathbb{R}^d;\mathbb{R}^d)}$.
\end{lemma}
\begin{proof}[Proof of \cref{thm:2nd-order-propagation-of-chaos}]
We model the proof on that of \cref{thm:propagation-of-chaos}, and thus split it into $5$ steps. Much of the proof is analogous to that of \cref{thm:propagation-of-chaos}. Therefore we only explain the differences.

\textbf{Step 1. Mollification of the interaction kernel.} This is identical to the corresponding step in the proof of \cref{thm:propagation-of-chaos}. We thus omit it.

\textbf{Step 2. Choice of functional space and exponents.}
By the assumptions on $f_0$ we may choose $r,r'$ such that the following holds:
\begin{equation*}
f_0\in \mathrm{L}^{r+r',q'}(\mathbb{R}^d\times\mathbb{R}^d), \quad r>d/q, \quad r'>d+1.
\end{equation*}
As in the proof of \cref{thm:propagation-of-chaos} we have the continuous inclusions:
\begin{align*}
\mathrm{L}^{r',2}(\mathbb{R}^d\times\mathbb{R}^d)\hookrightarrow \mathrm{L}^{1,1}(\mathbb{R}^d\times\mathbb{R}^d)\hookrightarrow (\mathrm{P}_1(\mathbb{R}^d\times\mathbb{R}^d),d_{\mathrm{MKW}}).
\end{align*} 
\textbf{Step 3. Regularity of the interaction field.} The choice of $\mathcal{C}$ depends upon which of assumptions (1) and (2) is made. More precisely which of the relaxed assumptions in \cref{rem:2nd-order-relax-assumption-on-K} is made. In each case:
\begin{enumerate}
\item By \cref{lem:2nd-order-a-priori-time-regularity}(1) we may assume that $b^N\in\mathcal{C}$ where
\begin{equation*}
\mathcal{C}=\mathrm{C}([0,T];\mathrm{C}^1_b(\mathbb{R}^d;\mathbb{R}^d))\cap \{b:\norm{b}_{\Lambda^{0,s}(\mathrm{L}^q([0,T]\times\mathbb{R}^d;\mathbb{R}^d))}\le A\},
\end{equation*}
for some $A<\infty$. Note that the metric entropy of $\mathcal{C}$ is bounded using \cref{lem:metric-entropy-of-Lip1} as
\begin{equation}\label{eq:2nd-order-metric-entropy-estimate}
H(\varepsilon,\mathcal{C},\norm{\cdot}_{\mathrm{L}^\infty([0,T];\mathrm{L}^{-p',\infty}(\mathbb{R}^d;\mathbb{R}^d))})\le C\varepsilon^{-(d+1)/s}
\end{equation}
where we have used that $q>(d+1)/s$ and used $p>d/q$ to take $p'\in (d/q,p)$.
\item By \cref{lem:2nd-order-a-priori-time-regularity}(2) we may assume that $b^N\in\mathcal{C}$ where
\begin{equation*}
\mathcal{C}=\mathrm{C}([0,T];\mathrm{C}^1_b(\mathbb{R}^d;\mathbb{R}^d))\cap \{b:\norm{b}_{\Lambda^{1,(s-1)}(\mathrm{L}^q([0,T]\times\mathbb{R}^d;\mathbb{R}^d))}\le A\},
\end{equation*}
where we have used that $p\ge4$. The corresponding metric entropy estimate provided by \cref{lem:metric-entropy-of-Lip1} is given again by \eqref{eq:2nd-order-metric-entropy-estimate}.(5) for the same choice of $p'$ and use of assumptions, but note here $s>1$.
\end{enumerate}
Note that Sobolev embedding implies the bound $\norm{b}_{\mathrm{C}^{0,\beta}([0,T]\times\mathbb{R}^d;\mathbb{R}^d)}\le CA$ for $\beta=\min(1,s-(d+1)/q)$ for any $b\in\mathcal{C}$. Moreover, Sobolev embedding also implies the bound $\norm{b}_{\mathrm{C}([0,T];\mathrm{C}^{0,\alpha}(\mathbb{R}^d;\mathbb{R}^d))}\le CA$ for $b\in\mathcal{C}$ and $\alpha=s-d/q>2/3$ by assumption. Therefore, we have sufficient regularity to apply \cref{thm:2nd-order-GV-intro} in the next step.

\textbf{Step 4. Consistency: The uniform law of large numbers on the particles.} This is identical to the corresponding step in the proof of \cref{thm:propagation-of-chaos}, except we instead apply \cref{thm:2nd-order-GV-intro} and here 
\begin{equation*}
\gamma_1=\frac{1}{2+\max(\frac{d+1}{s},d)}
\end{equation*}
(noting that $p>2$ by assumption).

\textbf{Step 5. Stability: Estimates on the limit equation.}

\textbf{Step 5.1. Dependence of $f$ upon the field.}  This is analogous to step 5 in the proof of \cref{thm:propagation-of-chaos} using the energy estimate \cref{lem:2nd-order-weighted-energy-estimate} and we leave it to the reader.

\textbf{Step 5.2. Dependence upon the field upon $f$.} The only differences between this step and the corresponding step in the proof of \cref{thm:propagation-of-chaos} are that here $b^{b,\infty}_t$ is defined by
\begin{equation*}
b^{b,\infty}_t(x)=\int K(x,y)\left(\int f^b(y,v)\,dv\right)\,dy
\end{equation*}
and \eqref{eq:prop-of-chaos-main-gronwall-estimate-part4} is replaced by
\begin{equation*}
\begin{aligned}
\norm{b^{b^N,\infty}_t-b^{b^\infty,\infty}_t}_{\mathrm{L}^{-r,q}(\mathbb{R}^d;\mathbb{R}^d)}&\le C\sup_{x\in\mathbb{R}^d}\int |K(x,y)|\left|\int f^{b^N}_t(y,v)\,dv-\int f^{b^\infty}_t(y,v)\,dv\right|\,dy\\
&\le C \norm{f^{b^N}_t-f^{b^\infty}_t}_{\mathrm{L}^1(\mathbb{R}^d\times\mathbb{R}^d)}\le C \norm{f^{b^N}_t-f^{b^\infty}_t}_{\mathrm{L}^{1,1}(\mathbb{R}^d\times\mathbb{R}^d)}.
\end{aligned}
\end{equation*}
Lastly, $\gamma_2$ is here instead given by
\begin{equation*}
\gamma_2=
\begin{cases}
\frac{1}{2+\tfrac{d+1}{s^2}},&\quad \text{if }s\le 1\\
\frac{1}{2+\tfrac{d+1}{s}}, &\quad \text{otherwise.}
\end{cases}
\end{equation*}

\textbf{Step 5.3. Gr\"onwall estimate.} This is identical to the proof of \cref{thm:propagation-of-chaos} and we omit it.
\end{proof} 
\subsection{Proof of the time regularity lemmas}\label{subsec:proof-of-the-time-regularity-lemmas}
For the proof of \cref{lem:a-priori-time-regularity,lem:2nd-order-a-priori-time-regularity} we require the following simple estimate.
\begin{lemma}\label{lem:little-lemma}
Let $E$ be an event and $K$ be bounded. Then
\begin{equation}\label{eq:little-lemma-bound}
\nnorm{1_E\left[\sup_{t\in[0,T]}d_{\mathrm{MKW}}(\mu^N_t,f_t)-cd_{\mathrm{MKW}}(\mu^N_0,f_0)\right]_+}\le C\mathbb{P}(E)+CN^{-1/2},
\end{equation}
where $c$ is chosen as in \cref{lem:reference-process-mkw-bound}.
\end{lemma}
\begin{proof}
We present only the first order case for brevity, the second order case being analogous. From an identical computation to that used in the proof of \cref{lem:mkw-triangle-with-references} and then using \cref{lem:reference-process-mkw-bound} we deduce that
\begin{align*}
&\left[\sup_{t\in[0,T]}d_{\mathrm{MKW}}(\mu^N_t,f_t)-cd_{\mathrm{MKW}}(\mu^N_0,f_0)\right]_+\\
&\quad \le \sup_{t\in[0,T]}\sup_{h\in \mathrm{Lip1}}\left(\averN (h(X^{i,N}_t)-h(\widetilde{X}^{i,N}_t))-\mathbb{E}(h(X^{i,N}_t)-h(\widetilde{X}^{i,N}_t))\right)\\
&\quad\le \averN\underbrace{C\left(1+\sup_{t\in[0,T]}|B^{i,N}_t|\right)}_{=:A^{i,N}}
\end{align*}
where we have used \cref{lem:SDE-subgaussian-bounds} to obtain the final line. Hence the left hand side of \eqref{eq:little-lemma-bound} is bounded by
\begin{align*}
\nnorm{1_E\averN A^{i,N}}&\le\nnorm{1_E\mathbb{E}A^{i,N}}+\nnorm{1_E\averN(A^{i,N}-\mathbb{E}A^{i,N})}\\
&\le(\mathbb{E}A^{i,N})\mathbb{P}(E)+\nnorm{\averN (A^{i,N}-\mathbb{E}A^{i,N})}\\
&\le C\mathbb{P}(E)+C\nnorm{A^{1,N}}N^{-1/2}
\end{align*}
where we have used the law of large numbers for sub-Gaussian random variables (\cref{lem:sub-Gaussian-LLN}) on the last line. As $A^{1,N}$ is a sub-Gaussian random variable, the proof is complete.
\end{proof}

\begin{proof}[Proof of \cref{lem:a-priori-time-regularity}]
By using \cref{lem:little-lemma} it suffices to find $A$ such that 
\begin{equation*}
\mathbb{P}(\norm{b^N}_{\Lambda^{0,\alpha'}_{para}(\mathrm{L}^{q}([0,T]\times\mathbb{R}^d;\mathbb{R}^d))}>A)\le CN^{-1/2}.
\end{equation*}
By the definition of $b^N$ we have the estimate
\begin{equation*}
\begin{aligned}
\norm{b^N}&_{\Lambda_{para}^{0,\alpha'}(\mathrm{L}^{q}([0,T]\times\mathbb{R}^d;\mathbb{R}^d))}\\
&=\norm{\averN K(x,X^{i,N}_t)}_{\Lambda^{0,\alpha'}_{para}(\mathrm{L}^{q}([0,T]\times\mathbb{R}^d;\mathbb{R}^d))}\\
&\le \averN\norm{K(x,X^{i,N}_t)}_{\Lambda^{0,\alpha'}_{para}(\mathrm{L}^{q}([0,T]\times\mathbb{R}^d;\mathbb{R}^d))}\\
&\le \norm{K}_{\Lambda^{0,\alpha}(\mathrm{L}^\infty_y(\mathbb{R}^d;\mathrm{L}^{q}_x(\mathbb{R}^d;\mathbb{R}^d)))}\left(1+\averN\norm{X^{i,N}_t}_{\mathrm{C}^{0,{\alpha'/(2\alpha)}}([0,T];\mathbb{R}^d)}\right)\\
&\le \averN\underbrace{C\left(1+\norm{B^{i,N}_t}_{\mathrm{C}^{0,{\alpha'/(2\alpha)}}([0,T];\mathbb{R}^d)}\right)}_{=:A^{i,N}}
\end{aligned}
\end{equation*}
where $(A^{i,N})_{i=1}^N$ are i.i.d. random variables with finite second moments (sub-Gaussian even, see \cite{Hytonen}). Set $A=2\mathbb{E}A^{1,N}$, then from Chebyshev's inequality we have
\begin{align*}
\mathbb{P}(\norm{b^N}_{\Lambda^{0,\alpha'}_{para}(\mathrm{L}^{q}([0,T]\times\mathbb{R}^d;\mathbb{R}^d))}>A)&\le \mathbb{P}\left(\averN (A^{i,N}-\mathbb{E}A^{i,N})>2\mathbb{E}A^{1,N}\right)\\
&\le  \frac{\var\left(\averN A^{i,N}\right)}{|\mathbb{E}A^{1,N}|^2}\le CN^{-1},
\end{align*}
which completes the proof of the lemma.

\end{proof}
To prove the second claim of \cref{lem:2nd-order-a-priori-time-regularity} we shall need a simple lemma.
\begin{lemma}\label{lem:simple-chain-rule}
Let $W\in \mathrm{W}^{1,1}_{loc}$ and $g\in \mathrm{C}^1([0,T];\mathbb{R}^d)$. Then $W(x-g(t))$ has weak time derivative given by
\begin{equation*} 
\partial_t[W(x-g(t))]=-g'(t)\cdot(\nabla W)(x-g(t)).
\end{equation*} 
\end{lemma}
\begin{proof}
Let $\varphi\in \mathcal{D}([0,T]\times\mathbb{R}^d)$ be a test function, and let the pairing of a distribution in $\mathcal{D}'([0,T]\times\mathbb{R}^d)$ and a test function in $\mathcal{D}([0,T]\times\mathbb{R}^d)$ be $\ip{\cdot}{\cdot}$. Then we have
\begin{align*}
\ip{\partial_t[W(x-g(t))]}{\varphi(t,x)}&=-\ip{W(x-g(t))}{(\partial_t\varphi)(t,x)}\\
&=-\ip{W(x)}{(\partial_t\varphi)(t,x+g(t))}\\
&=-\ip{W(x)}{\partial_t[\varphi(t,x+g(t))]-g'(t)\cdot(\nabla\varphi)(t,x+g(t))}\\
&=0+\ip{W(x)}{g'(t)\cdot(\nabla[\varphi(t,x+g(t))]}\\
&=-\ip{g'(t)\cdot\nabla W(x)}{\varphi(t,x+g(t))}\\
&=-\ip{g'(t)\cdot(\nabla W)(x-g(t))}{\varphi(t,x)}.\qedhere
\end{align*}
\end{proof}
\begin{proof}[Proof of \cref{lem:2nd-order-a-priori-time-regularity}]
We prove each claim in turn. For both claims, by \cref{lem:little-lemma} it is sufficient to bound the probability of the bad event.
\begin{enumerate}
\item As in the proof of \cref{lem:a-priori-time-regularity} we compute
\begin{equation*}
\begin{aligned}
\norm{b^N}&_{\Lambda^{0,\alpha}(\mathrm{L}^{q}([0,T]\times\mathbb{R}^d;\mathbb{R}^d))}\\
&=\norm{\averN K(\cdot,X^{i,N}_t)}_{\Lambda^{0,\alpha}(\mathrm{L}^{q}([0,T]\times\mathbb{R}^d;\mathbb{R}^d))}\\
&\le \averN \norm{K(\cdot,X^{i,N}_t)}_{\Lambda^{0,\alpha}(\mathrm{L}^{q}([0,T]\times\mathbb{R}^d;\mathbb{R}^d))}\\
&\le \norm{K}_{\Lambda^{0,\alpha}(\mathrm{L}^\infty_y(\mathbb{R}^d;\mathrm{L}^{q}_x(\mathbb{R}^d;\mathbb{R}^d)))}\left(1+\averN\norm{X^{i,N}_t}_{\mathrm{C}^{0,1}([0,T];\mathbb{R}^d;\mathbb{R}^d)}\right)\\
&\le \averN\underbrace{C(1+\sup_{t\in[0,T]}|V^{i,N}_t|)}_{=:A^{i,N}}.
\end{aligned}
\end{equation*}
Define $A=2\mathbb{E}A^{1,N}$ and $E=\{\averN A^{i,N}\ge A\}$. Then,
\begin{equation*}
\mathbb{P}(E)\le \frac{\var\left(\averN A^{i,N}\right)}{|\mathbb{E}A^{1,N}|^2}\le CN^{-1}
\end{equation*}
by Chebyshev's inequality using that $A^{i,N}$ are i.i.d. with finite second moment by \cref{lem:SDE-bounds}.
\item As $X^{i,N}$ is continuously time differentiable, we can apply \cref{lem:simple-chain-rule} to obtain
\begin{align*}
\partial_tb^N_t(x)&=\averN\partial_t[K(x,X^{i,N}_t)]=\averN\partial_t[W(x-X^{i,N}_t)]\\
&=\averN V^{i,N}_t\cdot (\nabla W)(x-X^{i,N}_t).
\end{align*}
Furthermore, the $x$ derivatives satisfy
\begin{equation*}
\nabla b^N_t(x)=\averN(\nabla W)(x-X^{i,N}_t),
\end{equation*}
which is always easier to bound than $\partial_tb^N$, so we omit these bounds.

Taking the $\mathrm{L}^{q}$ norm we have
\begin{equation*}
\norm{\partial_tb^N_t}_{\mathrm{L}^{q}([0,T]\times \mathbb{R}^d;\mathbb{R}^d)}\le C\norm{\nabla W}_{\mathrm{L}^q(\mathbb{R}^d;\mathbb{R}^{d\times d})}\averN \sup_{s\in[0,T]}|V^{i,N}_s|,
\end{equation*}
and similarly,
\begin{equation*}
\begin{aligned} 
\sup_{x,y\in\mathbb{R}^d,x\ne y}\frac1{|x-y|^\alpha}&\norm{\partial_tb^N(x+\cdot)-\partial_tb^N(y+\cdot)}_{\mathrm{L}^{q}([0,T]\times \mathbb{R}^d;\mathbb{R}^d)}\\
&\le C\norm{\nabla W}_{\Lambda^{0,\alpha}(\mathrm{L}^{q}(\mathbb{R}^d;\mathbb{R}^d))}\averN \sup_{s\in[0,T]}|V^{i,N}_s|,
\end{aligned} 
\end{equation*}
and we can estimate the $V^{i,N}$ terms in the same way as part (1). In the same way
\begin{equation*}
\begin{aligned}
\sup_{s,t\in[0,T],s\ne t}\frac1{|t-s|^\alpha}&\norm{\partial_tb_t^N-\partial_tb_s^N}_{\mathrm{L}^{q}([0,T]\times \mathbb{R}^d;\mathbb{R}^d)}\\
&\le C\norm{\nabla W}_{\Lambda^{0,\alpha}(\mathrm{L}^{q}(\mathbb{R}^d;\mathbb{R}^d))}\left(\averN \sup_{s\in[0,T]}|V^{i,N}_s|\right)^2\\
&\quad+C\norm{\nabla W}_{\mathrm{L}^{q}(\mathbb{R}^d;\mathbb{R}^{d\times d})}\averN \norm{V^{i,N}}_{\mathrm{C}^{0,\alpha}([0,T];\mathbb{R}^d)}.
\end{aligned}
\end{equation*}
All of these terms may be controlled using the methods in part (1) and the proof of \cref{lem:a-priori-time-regularity} as $\alpha<1/2$. We omit the details.\qedhere 
\end{enumerate}
\end{proof}
\subsection{Proof of the energy estimates}\label{subsec:proof-of-the-energy-estimates}
We now provide the proofs of the two energy estimates.
\begin{proof}[Proof of \cref{lem:weighted-energy-estimate}]
For brevity, let $f_t=f^b_t$ and $\tilde{f}_t=f^{\tilde{b}}_t$. [We abuse notation in this proof and use $\tilde{f}$ to refer to the definition in the previous sentence rather than the law of the reference process.] Let $g=f-\tilde{f}$, then $g_t$ solves
\begin{equation*}
\left\{\begin{aligned}
&\partial_t g_t+\nabla\cdot(b_tg_t)-\frac12\Delta g_t=-\nabla\cdot(\tilde{f}_t(b_t-\tilde{b}_t)), \quad (t,x)\in[0,T]\times\mathbb{R}^d,\\
&g_0=0.
\end{aligned}\right. 
\end{equation*}
We multiply this equation by $g_t\<x^{2p}$ and integrate by parts. This yields
\begin{equation*}
\begin{aligned}
\frac{d}{dt}\norm{g_t\<x^{p}}^2_{\mathrm{L}^2(\mathbb{R}^d)}&-\int g_tb_t\cdot\nabla(g_t\<x^{2p})\,dx+\int\nabla g_t\cdot\nabla(g_t\<x^{2p})\,dx\\
&=\int\tilde{f}_t(b_t-\tilde{b}_t)\cdot\nabla (g_t\<x^{2p})\,dx.
\end{aligned}
\end{equation*}
We bound the right hand side using H\"older's inequality by
\begin{equation*}
\left|RHS\right|\le \norm{b_t-\tilde{b}_t}_{\mathrm{L}^{-r,q'}(\mathbb{R}^d;\mathbb{R}^d)}\norm{\tilde{f}_t}_{\mathrm{L}^{p+r,q}(\mathbb{R}^d)}\left(\norm{g_t}_{\mathrm{L}^{p,2}(\mathbb{R}^d)}+\norm{\nabla g_t}_{\mathrm{L}^{p,2}(\mathbb{R}^d;\mathbb{R}^d)}\right)
\end{equation*}
and similarly the second term on the left hand side using instead $\norm{b}_{\mathrm{L}^\infty([0,T]\times\mathbb{R}^d;\mathbb{R}^d)}\le C$.
Using Young's inequality, and that $\nabla$ hitting $\<x^{2p}$ produces terms of lower order, we obtain
\begin{equation*}
\frac{d}{dt}\norm{g_t}^{2}_{\mathrm{L}^{p,2}(\mathbb{R}^d)}\le C\norm{g_t}_{\mathrm{L}^{p,2}(\mathbb{R}^d)}^2+C\norm{\tilde{f}_t}^2_{\mathrm{L}^{p+r,q}(\mathbb{R}^d)}\norm{b_t-\tilde{b}_t}^2_{\mathrm{L}^{-r,q'}(\mathbb{R}^d;\mathbb{R}^d)}.
\end{equation*}
Hence, by Gr\"onwall's inequality, we have
\begin{equation*}
\norm{g_t}^2_{\mathrm{L}^{p,2}(\mathbb{R}^d)}\le C\int^t_0\norm{\tilde{f}_s}^2_{\mathrm{L}^{p+r,q}(\mathbb{R}^d)}\norm{b_s-\tilde{b}_s}_{\mathrm{L}^{-r,q'}(\mathbb{R}^d;\mathbb{R}^d)}^2\,ds,
\end{equation*}
which implies that
\begin{equation*}
\norm{g_t}_{\mathrm{L}^{p,2}(\mathbb{R}^d)}\le C\int^t_0\norm{\tilde{f}_s}_{\mathrm{L}^{p+r,q}(\mathbb{R}^d)}\norm{b_s-\tilde{b}_s}_{\mathrm{L}^{-r,q'}(\mathbb{R}^d;\mathbb{R}^d)}\,ds.
\end{equation*}
Thus it suffices to obtain a bound, independent of $\tilde{b}$,
\begin{equation*}
\norm{\tilde{f}}_{\mathrm{L}^\infty([0,T];\mathrm{L}^{p+r,q}(\mathbb{R}^d))}\le C.
\end{equation*}
This may be done using the equation for $\tilde{f}_t$, the assumed $\mathrm{L}^q$ moment bound on $f_0$ and the same technique as above multiplying by $|\tilde{f}_t|^{q-1}\<x^{q(p+r)}$ instead of $g_t\<x^{2p}$. We omit the details.
\end{proof}
The proof of the weighted energy estimate in the second order case is slightly different.
\begin{proof}[Proof of \cref{lem:2nd-order-weighted-energy-estimate}]
As in the proof of \cref{lem:weighted-energy-estimate}, let $f_t=f^b_t$, $\tilde{f}_t=f^{\tilde{b}}_t$ and $g=f-\tilde{f}$. Then $g$ solves
\begin{equation*}
\left\{\begin{aligned}
&\partial_t g_t+v\cdot\nabla_{x}g_t-\kappa \nabla_v\cdot(vg_t)+b_t\cdot\nabla_v g_t-\frac12\Delta_v g_t=-(b_t-\tilde{b}_t)\cdot \nabla_v\tilde{f}_t,\quad (t,x,v)\in[0,T]\times\mathbb{R}^d\times\mathbb{R}^d,\\
&g_0=0.
\end{aligned}\right. 
\end{equation*}
By multiplying this equation by $g_t{\<{(x,v)}}^{2p}$ and then integrating by parts, we obtain the following weighted energy estimate
\begin{equation*}
\begin{aligned}
&\frac{d}{dt}\norm{g_t}^2_{\mathrm{L}^{p,2}(\mathbb{R}^{2d})}+\int v\cdot \nabla_x(g_t{\<{(x,v)}}^{2p})g_t\,dxdv+\kappa \int g_tv\cdot \nabla_v(g_t{\<{(x,v)}}^{2p})\,dxdv+\\
&-\int g_tb_t\cdot \nabla_v(g_t{\<{(x,v)}}^{2p})\,dxdv+\frac12\int\nabla_vg_t\cdot \nabla_v(g_t{\<{(x,v)}}^{2p})\,dxdv\\
&=\int \tilde{f}_t(b_t-\tilde{b}_t)\cdot \nabla_v(g_t{\<{(x,v)}}^{2p})\,dxdv.
\end{aligned}
\end{equation*} 
In the similar way as in the proof of \cref{lem:weighted-energy-estimate}, as $\nabla_x,\nabla_v$ hitting ${\<{(x,v)}}^{2p}$ give terms of lower order and as $b_t,\tilde{b}_t$ are independent of $v$, we have
\begin{equation*}
\begin{aligned}
&\frac{d}{dt}\norm{g_t}^2_{\mathrm{L}^{p,2}(\mathbb{R}^{2d})}+\norm{\nabla_vg_t}^2_{\mathrm{L}^{p,2}(\mathbb{R}^{2d};\mathbb{R}^d)}\le  C\norm{g_t}^2_{\mathrm{L}^{p,2}(\mathbb{R}^{2d})}+\\
&\quad+C\norm{b_t-\tilde{b}_t}_{\mathrm{L}^{-r,q'}(\mathbb{R}^d;\mathbb{R}^d)}\norm{\tilde{f}_t}_{\mathrm{L}^{p+r,q}(\mathbb{R}^{2d})}(\norm{g_t}_{\mathrm{L}^{p,2}(\mathbb{R}^{2d})}+\norm{\nabla_vg_t}_{\mathrm{L}^{p,2}(\mathbb{R}^{2d};\mathbb{R}^d)})
\end{aligned}
\end{equation*}
By using Young's inequality and then the Gr\"onwall inequality we obtain, as in the proof of \cref{lem:weighted-energy-estimate},
\begin{equation*}
\norm{g_t}_{\mathrm{L}^{p,2}(\mathbb{R}^{2d})}\le C\int^t_0\norm{\tilde{f}_s}_{\mathrm{L}^{p+r,q}(\mathbb{R}^{2d})}\norm{b_s-\tilde{b}_s}_{\mathrm{L}^{-r,q'}(\mathbb{R}^d;\mathbb{R}^d)}\,ds.
\end{equation*}
The claim of the lemma then follows from a bound on $\sup_{t\in[0,T]}\norm{\tilde{f}_t}_{\mathrm{L}^{p+r,q}(\mathbb{R}^{2d})}$ which may be obtained using the assumption that $f_0\in \mathrm{L}^{p+r,q}(\mathbb{R}^{2d})$ and similar energy estimates to the above. We leave this to the reader.
\end{proof}
\section{Counterexample}\label{sec:counterexample}
In this section we will prove \cref{prop:counterexample}. We begin by introducing a sorting problem, which if the uniform law of large numbers holds over a class $\mathcal{C}$, is unsolvable.
\begin{problem}\label{problem:sorting}
Given a class of vector fields $\mathcal{C}$ and even $N$, consider $N$ particles evolving as \eqref{eq:many-particle-system-empirical} with initial law $f_0$. Tag the first $N/2$ particles \emph{red} and the rest \emph{blue}. Can we choose a (random) $b^N\in\mathcal{C}$ (depending on $N$) so that the red and blue particles are sorted to the right and left respectively, uniformly in $N$, i.e.
	\begin{equation}\label{eq:prob:task}
	\inf_{N\to\infty}\mathbb{E}\left[\averN  h(X^{b^N,i,N}_t,\mathrm{colour}(X^{b^N,i,N}))-\mathbb{E}h(X^{b^N,i,N}_t,\mathrm{colour}(X^{b^N,i,N}))\right]>0
	\end{equation}
	where 
	\begin{equation}
	h(x_1,\dotsc,x_d,c)=g(x_1)\begin{cases}
	1&\text{ if }c=\text{red}\\
	-1&\text{ if }c=\text{blue}
	\end{cases}
	\end{equation}
	and $g(x)$ is a smoothed version of the sign function.
\end{problem}
A simple argument by contradiction implies the following lemma.
\begin{lemma}
If \cref{problem:sorting} is solvable for a class of vector fields $\mathcal{C}$, then the uniform law of large numbers for SDEs cannot hold over this class. In particular, if \cref{problem:sorting} is solvable for $\mathcal{C}^0$ given by \eqref{eq:counterexample-class} then \cref{prop:counterexample} is true.
\end{lemma}
We will now exhibit an explicit vector field $b\in\mathcal{C}^0$ that solves \cref{problem:sorting}, thus proving \cref{prop:counterexample}. This turns out to be quite simple. 

Firstly, we note that, whatever $b\in\mathcal{C}^0$ is chosen, the set of times at which any two particles are in the same position is of measure zero almost surely. This is due to the absolute continuity with respect to Brownian motion due to Girsanov's theorem. Define the function $\psi_\varepsilon(t)$ as
\begin{equation}
\psi_\varepsilon(t)=\prod_{i=1}^{N/2}\prod_{j=N/2+1}^N\psi\left(\frac{X^{b^N,i,N}_t-X^{b^N,j,N}_t}\varepsilon\right)
\end{equation}
for $\varepsilon>0$ arbitrary, and $\psi(x)$ a function that is zero for $|x|\le 1/2$ and $1$ for $|x|\ge 1$. Then $\psi_\varepsilon$ is adapted, almost surely continuous and zero whenever a red and blue particle are within $\varepsilon/2$ of each other, and $1$ when no such particles are within $\varepsilon$ of each other. Furthermore, it is a simple computation to show that 
\begin{equation*}
\lim_{\varepsilon\to0}\mathbb{E}\int^T_01_{\psi_\varepsilon(t)\ne1}\,dt=T,
\end{equation*} 
no matter what $b^N\in \mathcal{C}^0$ is chosen, although (of course) this limit will not be uniform in $N$. Now let $\varepsilon>0$ be chosen so that the expectation of the above integral is at least $3T/4$. 

Let $\eta_\varepsilon(x)=\eta(x/\varepsilon)$ where $\eta(x)$ is a smooth bump function with $\eta(0)=1$ and $\eta(x)=0$ for $|x|\ge 1/2$. Now define $b^N$ as 
\begin{equation*}
b^N(x)=\psi_\varepsilon(t)\left(\sum_{i=1}^{N/2}\eta_\varepsilon(x-X^{b^N,i,N}_t)-\sum_{i=N/2+1}^{N}\eta_\varepsilon(x-X^{b^N,i,N}_t)\right),
\end{equation*}
then $b^N$ is adapted, uniformly bounded by $1$, smooth in $x$ and continuous in $t$ almost surely. Moreover, when $\psi_\varepsilon(t)=1$, $b^N$ is equal to $1$ on every red particle and $-1$ on every blue particle. Therefore, every red particle is pushed by $b^N$ at least the distance $\int^T_01_{\psi_\varepsilon=1}-1_{\psi_\varepsilon\ne1}\,dt$ to the right and similarly every blue particle to the left. By choice of $\varepsilon$ the expectation of this is at least $T/2$. Hence, the expectation \eqref{eq:prob:task} is bounded away from zero by a fixed constant independent of $N$, completing the proof.
\appendix
\section{Metric entropy}
\label{sec:appendix}
In this section we summarise the properties and estimates of metric entropy that are used in the rest of the manuscript. The results henceforth are either well known or simple corollaries of well known results. The reader is encouraged to consult \cite{Edmunds-Triebel,Function-spaces-III} for an exposition of metric entropy in the context of functional analysis and \cite{Van-der-Vaart-Wellner} for a more statistical viewpoint (cf. \cite{Nickl}).
\begin{lemma}[Metric entropy of product spaces]\label{lem:metric-entropy-of-product-spaces}
Let $(X,d_X)$, $(Y,d_Y)$ be totally bounded metric spaces. Define the product metric $d_{X\times Y}$ on $X\times Y$, by
\begin{equation*}
d_{X\times Y}((x,y),(x',y'))=\max(d_X(x,x'),d_Y(y,y')).
\end{equation*}
Then it holds that
\begin{equation*}
H(\varepsilon,X\times Y,d_{X\times Y})\le H(\varepsilon,X,d_X)+H(\varepsilon,Y,d_Y).
\end{equation*}
In particular, if $H(\varepsilon,X,d_X)\le C\varepsilon^{-k_X}$ and $H(\varepsilon,Y,d_Y)\le C\varepsilon^{-k_Y}$ then $H(\varepsilon,X\times Y,d)\le C\varepsilon^{-\max(k_X,k_Y)}$ for any metric $d$ equivalent to $d_{X\times Y}$.
\end{lemma}
\begin{proof}
Let $x_1,\dotsc x_n$ be an $\varepsilon$-net of $(X,d_X)$ and $y_1,\dotsc, y_m$ be an $\varepsilon$-net of $(Y,d_Y)$. Then $\{(x_i,y_j):1\le i\le n, 1\le j\le m\}$ is an $\varepsilon$-net of $(X\times Y,d_{X\times Y})$. The claims follow.
\end{proof}
\begin{lemma}[Metric entropy of finite dimensional spaces]
	Let $K$ be a compact set in $\mathbb{R}^d$, and $|\cdot|$ be the Euclidean norm
	\begin{equation*}
	H(\varepsilon,K,|\cdot|)\le C\log(1/\varepsilon) .
	\end{equation*}
\end{lemma}
\begin{proof}
	It suffices to consider $K=[0,1]^d$ and by \cref{lem:metric-entropy-of-product-spaces} we need only consider $K=[0,1]$. Then an explicit $\varepsilon$-net is given by $\{k\varepsilon:k\in\mathbb{N},k\le 1/\varepsilon\}$.
\end{proof}
\begin{lemma}[Change of metric]\label{lem:change-of-metric}
	Let $\mathcal{C}$ be a totally bounded subset of a metric space $(X,d)$, then it holds that
	\begin{equation}
	H(\varepsilon,\mathcal{C},d^\alpha)\le CH(\epsilon^{1/\alpha},\mathcal{C},d)
	\end{equation}
	where $d^\alpha(x,x')=|d(x,x')|^\alpha$ for $\alpha\in(0,1]$. In particular, if $H(\varepsilon,\mathcal{C},d^\alpha)\le C\varepsilon^{-k}$ then $H(\varepsilon,\mathcal{C},d^\alpha)\le C\varepsilon^{-k/\alpha}$.
\end{lemma}
\begin{proof}
	Let $(x_n)_{n=1}^m$ be a $\varepsilon$-net with respect to $d$ of $\mathcal{C}$, then $(x_n)_{n=1}^m$ is also an $\varepsilon^\alpha$ net of $\mathcal{C}$ with respect to $d^\alpha$. The particular claim follows easily.
\end{proof}
The main estimates of metric entropy required for the rest of the manuscript are given in the proposition below.
\begin{proposition}[Metric entropy of smooth functions]\label{lem:metric-entropy-of-Lip1}
	Let $p>1$, then the following hold:
\begin{enumerate}
	\item \textbf{Lipschitz functions:} Let $p\ne2$, then the Lipschitz functions on $\mathbb{R}^d$ obey:
	\begin{equation*}
	H(\varepsilon,\mathrm{Lip1},\norm{\cdot}_{\mathrm{L}^{-p,\infty}(\mathbb{R}^d)})\le C\varepsilon^{-\min(d,d/(p-1))}.
	\end{equation*}
	\item \textbf{H\"older functions:} For $\alpha\in (0,1]$, the H\"older functions $\mathcal{C}^\alpha=\{f\in \mathrm{C}^{0,\alpha}(\mathbb{R}^d):\norm{f}_{\mathrm{C}^{0,\alpha}(\mathbb{R}^d)}\le C\},$
	obey the bound:
	\begin{equation*}
	H(\varepsilon,\mathcal{C}^{\alpha},\norm{\cdot}_{\mathrm{L}^{-p,\infty}(\mathbb{R}^d)})\le C\varepsilon^{-d/\alpha}.
	\end{equation*}
	\item \textbf{Parabolic H\"older functions:} For $\alpha\in (0,1]$, the parabolic H\"older functions $\mathcal{C}^\alpha_{para}=\{f\in \mathrm{C}^{0,\alpha}_{para}([0,T]\times\mathbb{R}^d):\norm{f}_{\mathrm{C}^{0,\alpha}_{para}([0,T]\times\mathbb{R}^d)}\le C\},$ obey the bound:
		\begin{equation*}
		H(\varepsilon,\mathcal{C}^{\alpha}_{para},\norm{\cdot}_{\mathrm{L}^\infty([0,T];\mathrm{L}^{-p,\infty}(\mathbb{R}^d))})\le C\varepsilon^{-(d+2)/\alpha}.
		\end{equation*}
	\item \textbf{Parabolic $\mathrm{L}^q$ H\"older functions:} Let $k\in\{0,1,\dotsc\}$, $\alpha\in(0,1]$ and $q\in[1,\infty]$ with $k+\alpha>(d+2)/q$. The set of parabolic $\mathrm{L}^q$ H\"older functions $\mathcal{C}^{k,\alpha}_{q,para}=\{f\in \Lambda^{k,\alpha}_{para}(\mathrm{L}^q([0,T]\times\mathbb{R}^d)):\norm{f}_{\Lambda^{k,\alpha}_{para}(\mathrm{L}^q([0,T]\times\mathbb{R}^d))}\le C\}$ obeys the bound:
	\begin{equation*}
	H(\varepsilon,\mathcal{C}^{k,\alpha}_{q,para},\norm{\cdot}_{\mathrm{L}^\infty([0,T];\mathrm{L}^{-p,\infty}(\mathbb{R}^d))})\le C\varepsilon^{-(d+2)/(k+\alpha)}.
	\end{equation*}
	for any $p>d/q$.
	\item \textbf{$\mathrm{L}^q$ H\"older functions:} Let $k\in\{0,1,\dotsc\}$, $\alpha\in(0,1]$ and $q\in[1,\infty]$ with $k+\alpha>(d+1)/q$. The set of $\mathrm{L}^q$ H\"older functions $\mathcal{C}^{k,\alpha}_{q}=\{f\in \Lambda^{k,\alpha}(\mathrm{L}^q([0,T]\times\mathbb{R}^d)):\norm{f}_{\Lambda^{k,\alpha}(\mathrm{L}^q([0,T]\times\mathbb{R}^d))}\le C\}$ obeys the bound:
		\begin{equation*}
		H(\varepsilon,\mathcal{C}^{k,\alpha}_{q},\norm{\cdot}_{\mathrm{L}^\infty([0,T];\mathrm{L}^{-p,\infty}(\mathbb{R}^d))})\le C\varepsilon^{-(d+1)/(k+\alpha)},
		\end{equation*}
		for any $p>(k+\alpha)-(d+1)/q$.
\end{enumerate}
In all cases the estimates for vector valued (i.e. in $\mathbb{R}^n$) functions are the same up to change in constants due to \cref{lem:metric-entropy-of-product-spaces}.
\end{proposition}
\begin{proof}
	We prove each in turn.
	\begin{enumerate}
		\item Let $h\in \mathrm{Lip1}$ be arbitrary, then we have $h(0)=0$ and therefore $|h(x)|\le |x|$. As a result we have $x\mapsto h(x)\<x\in \mathrm{C}^1_b(\mathbb{R}^d)$ with a bound on the $\mathrm{C}^1_b$ norm independent of $h\in \mathrm{Lip1}$. Let $B$ be the unit ball in $\mathrm{C}^1_b$. We deduce that
		\begin{equation*}
		H(\varepsilon,\mathrm{Lip1},\norm{\cdot}_{\mathrm{L}^{-p,\infty}(\mathbb{R}^d)})\le H(\varepsilon,B,\norm{\cdot}_{\mathrm{L}^{-(p-1),\infty}(\mathbb{R}^d)}).
		\end{equation*}
		The bound then follows from the estimates on entropy numbers in \cite{Edmunds-Triebel} (cf. \cite{Nickl} for an exposition in terms of metric entropy).
		\item This result follows directly from \cite[Corollary 3.1]{Nickl}.
		\item This follows from \cref{lem:metric-entropy-of-weighted-anisotropic-spaces} below and the identification \eqref{eq:identification-of-parabolic-spaces1},\eqref{eq:identification-of-parabolic-spaces1} with $d_1=1$, $d_2=d$, $p=\infty$, and the parabolic anisotropy $\mathfrak{p}$ defined below.
		\item This follows from \cref{lem:metric-entropy-of-weighted-anisotropic-spaces} in the same way as (3).
		\item This follows from \cite[Theorem 1.1]{Nickl}.
		\qedhere
	\end{enumerate}
\end{proof}

We now provide a simple estimate on the metric entropy of weighted spaces of anisotropic regularity, which was we needed for \cref{lem:metric-entropy-of-Lip1}(3)-(4). We make no claim of optimality or originality in this result, which the author has included because of the inability to find a reference.
\begin{definition}[Anisotropy]
An \emph{anisotropy} of $\mathbb{R}^d$ is tuple $\mathfrak{a}=(\mathfrak{a}_1,\dotsc,\mathfrak{a}_d)\in\mathbb{R}^d$ such that
\begin{equation*}
\mathfrak{a}_i>0 \text{ for each }i=1,\dotsc,d,\text{ and }\sum_{i=1}^d\mathfrak{a}_i=d.
\end{equation*}
\end{definition}
An anisotropy $\mathfrak{a}$ corresponds to the anisotropic distance $|\cdot|_{\mathfrak{a}}$ on $\mathbb{R}^d$ given by
\begin{equation*}
|x|_{\mathfrak{a}}=|(x_1,\dotsc,x_d)|_{\mathfrak{a}}=\sum_{i=1}^d|x_i|^{\mathfrak{a}_i}.
\end{equation*}
Note that for $\mathfrak{a}=(1,\dotsc,1)$ the distance $|\cdot|_{\mathfrak{a}}$ is equivalent to the usual Euclidean distance on $\mathbb{R}^d$.

We record in particular that 
\begin{equation*}
\mathfrak{p}:=\frac{d+1}{d+(1/2)}(1/2,\underbrace{1,1,\dotsc,1}_{d\text{ times}})\in\mathbb{R}^{1+d}
\end{equation*}
is the \emph{parabolic} anisotropy. Here the prefactor is to ensure that the sum of the indices is $d+1$.

Given an anisotropy $\mathfrak{a}$ and a subset $U\subseteq \mathbb{R}^d$ it is possible to define the Besov space of anisotropic regularity $B^{s,\mathfrak{a}}_{p,q}(U)$ for $p,q\in(0,\infty]$ and $s\in\mathbb{R}$.  As we do not require the full (somewhat lengthy) definition of these spaces we do not provide them. Instead we refer the reader to  \cite[\S5]{Function-spaces-III} for their full definition and for more details about anisotropies and spaces of anisotropic regularity.

For our purposes it is sufficient to note that
\begin{equation}\label{eq:identification-of-parabolic-spaces1}
\Lambda^{k,\alpha}_{para}(\mathrm{L}^q([0,T]\times\mathbb{R}^d))=B^{\mathfrak{p},s}_{q,q}([0,T]\times\mathbb{R}^d)
\end{equation}
for any $q\in[1,\infty]$, non-negative integer $k$, $\alpha\in(0,1]$ and
\begin{equation}\label{eq:identification-of-parabolic-spaces2}
s=\frac{d+1}{d+2}(k+\alpha).
\end{equation}

\begin{lemma}\label{lem:metric-entropy-of-weighted-anisotropic-spaces}
Fix $d_1,d_2\in\mathbb{N}$, $s>0,\in[1,\infty],r>0$ and an anisotropy $\mathfrak{a}$ such that $s>d/q$ and $r>d_2/q$ where $d=d_1+d_2$. Define the set $\mathcal{C}$ as those functions $h\in B_{q,q,loc}^{\mathfrak{a},s}([0,1]^{d_1}\times\mathbb{R}^{d_2})$ satisfying the estimate
\begin{equation*}
\sup_{Q}\norm{h}_{B^{s,\mathfrak{a}}_{q,q}([0,1]^{d_1}\times Q)}\le C'
\end{equation*}
for a fixed constant $C'$, where the supremum is over unit cubes $Q\subseteq\mathbb{R}^{d_2}$. Then,
\begin{equation*}
H(\varepsilon,\mathcal{C},\norm{\cdot}_{\mathrm{L}^{-r,\infty}([0,1]^{d_1}\times\mathbb{R}^{d_2})})\le C\varepsilon^{-d/s}.
\end{equation*}
\end{lemma}
\begin{proof}[Proof of \cref{lem:metric-entropy-of-weighted-anisotropic-spaces}]
Fix $\varepsilon>0$ and let $R=R(\varepsilon)$ to be chosen be a constant. (All constants $C$ will be uniform in $\varepsilon$ and $R$). For $n\in\mathbb{Z}^{d_2}$, let $Q_n$ be the cube $[0,1]^{d_1}\times (n+[-1/2,1/2]^{d_2})$ and let $\mathcal{Q}=\{Q_n:|n|\le R\}$. Define the set $E$ by
\begin{equation*}
E=\prod_{Q_n\in \mathcal{Q}}E_n
\end{equation*}
where $E_n$ is a $\varepsilon |n|^r$-net of $\mathcal{C}\vert_{Q_n}$ (the restriction of functions in $\mathcal{C}$ to $Q_n$). Note that $H(\varepsilon|n|^r,\mathcal{C}\vert_{Q_n},\norm{\cdot}_{\mathrm{L}^\infty(Q_n)})\le C|n|^{-rd/s}\varepsilon^{-d/s}$ by \cite[theorem 5.30.]{Function-spaces-III}. Therefore, $E_n$ can be chosen so that $\log|E_n|\le C|n|^{-rd/s}\varepsilon^{-d/s}$ and hence
\begin{equation*}
\log|E|\le C\varepsilon^{-k}\sum_{n\in\mathbb{Z}^{d_2},|n|<R}|n|^{-rd/s}\le C\varepsilon^{-d/s}\sum_{n\in\mathbb{Z}^{d_2}}|n|^{-rd/s}\le C\varepsilon^{-d/s}
\end{equation*}
as $rd/s>d_2$ by assumption, so the sum is finite. 

Define the set of functions $F$, by collecting for each $(e_n)_{n\in\mathbb{Z}^{d_2},|n|\le R}\in E$ a function $h\in \mathcal{C}$ with the property that for each $n$, $\norm{h-e_n}_{\mathrm{L}^\infty(Q_n)}\le \varepsilon|n|^r$ if such a function exists. Then $\log|F|\le \log|E|\le C\varepsilon^{-d/s}$ and we claim that $F$ is a $C\varepsilon$-net of $\mathcal{C}$ in the $\mathrm{L}^{-r,\infty}([0,1]^{d_1}\times\mathbb{R}^{d_2})$. Indeed, suppose that $h\in\mathcal{C}$, then for each cube $Q_n$ ($|n|<R$) we have a function $e_n\in E_n$ with $\norm{h-e_n}_{\mathrm{L}^{\infty}(Q_n)}\le \varepsilon|n|^r $ by construction. Then by construction of $F$ we have $g\in F$ with $\norm{g-e_n}_{\mathrm{L}^{\infty}(Q_n)}\le \varepsilon |n|^r\varepsilon$, (such a function must exist as we could have chosen $h$ in the construction of $F$). Therefore,
\begin{equation*}
\norm{g-h}_{\mathrm{L}^{-r,\infty}(\bigcup_{Q_n\in\mathcal{Q}}Q_n)}\le \norm{g-e}_{\mathrm{L}^{-r,\infty}(\bigcup_{Q_n\in\mathcal{Q}}Q_n)}+\norm{g-e}_{\mathrm{L}^{-r,\infty}(\bigcup_{Q_n\in\mathcal{Q}}Q_n)}\le  C\varepsilon
\end{equation*}
where $e(x)=e_n(x)$ for $x\in Q_n$.  While on $\{x:|x|>R\}$ we have that $|h-g|\<x^{-r}\le 2CR^{-r}$ as functions in $\mathcal{C}$ are uniformly bounded by Sobolev embedding. By choosing $R$ sufficiently large we may ensure that this is less than $\varepsilon$.
\end{proof}

\bibliographystyle{plain}
\bibliography{paper}
\end{document}